\documentclass[11pt,a4paper]{article}
\usepackage{graphicx} 
\usepackage{amsmath, color, verbatim}
\usepackage[numbers]{natbib}
\usepackage{amsfonts, mathtools}
\usepackage[colorlinks=true]{hyperref}
\usepackage{amssymb}
\usepackage{amsthm, sectsty}
\usepackage{subfigure}

\usepackage[nameinlink,noabbrev,capitalize]{cleveref}

\usepackage[hyperref, dvipsnames]{xcolor}
\hypersetup{
  colorlinks=true,
}

\usepackage[left=2cm,right=2cm,top=2cm,bottom=2cm]{geometry}

\newcommand{\norm}[2]{\left\lVert #1\right\rVert_{#2}}

\newcommand{\T}{\mathcal{T}}

\newcommand{\md}{\partial^\bullet}

\newtheorem{theorem}{Theorem}[section]
\newtheorem{defn}[theorem]{Definition}
\newtheorem{lem}[theorem]{Lemma}
\newtheorem{prop}[theorem]{Proposition}

\newtheorem{ass}[theorem]{Assumption}

\newtheorem{cor}[theorem]{Corollary}
\newtheorem{remark}[theorem]{Remark}

\renewcommand{\vec}[1]{\ensuremath{\boldsymbol{#1}}}

\newcommand{\R}{\mathbb{R}}
\renewcommand{\d}{\mathrm{d}}
\newcommand{\eps}{\varepsilon}
\newcommand{\tgrad}{\nabla_\Gamma}
\usepackage{todonotes}
\newcounter{todocounter}

\renewcommand{\bar}{\overline}

\usepackage{enumitem}

\usepackage{multicol}
\usepackage{framed}

\newcommand{\grad}{\nabla}
\newcommand{\sgrad}{\nabla_\Gamma}

\newcommand{\slap}{\Delta_\Gamma}

\newcommand{\weaklyto}{\rightharpoonup}
\newcommand{\weakstar}{\stackrel{*}\rightharpoonup}
\newcommand{\cts}{\hookrightarrow}

\renewcommand{\O}{\ensuremath{{\Omega}}}
\newcommand{\G}{\ensuremath{{\Gamma}}}

\newcommand{\dVp}{\sgrad \cdot \mathbf{V}_p}
\newcommand{\dVpO}{\grad \cdot \mathbf{V}_p}
\newcommand{\dVG}{\sgrad \cdot \mathbf{V}_\Gamma}
\newcommand{\dVO}{\grad \cdot \mathbf{V}_\Omega}
\newcommand{\linfdVG}{\norm{\dVG}{\infty}}
\newcommand{\linfdVO}{\norm{\dVO}{\infty}}
\newcommand{\JG}{\mathbf{J}_\Gamma}
\newcommand{\JO}{\mathbf{J}_\Omega}

\newcommand{\intOt}{\int_{\Omega(t)}}

\newcommand{\intGt}{\int_{\Gamma(t)}}

\newcommand{\intGz}{\int_{\Gamma_0}}

\newcommand{\dt}{\;\mathrm{d}t}
\newcommand{\ds}{\;\mathrm{d}s}

\newcommand{\Vp}{\mathbf{V}_p}
\newcommand{\VG}{\mathbf{V}_\Gamma}
\newcommand{\VO}{\mathbf{V}_\Omega}

\newcommand{\delo}{\delta_\Omega}
\newcommand{\delz}{\delta_{k'}}
\newcommand{\delk}{\delta_{k}}
\newcommand{\delg}{\delta_{\G}}
\newcommand{\delgp}{\delta_{\G'}}

\newcommand{\mdd}{\partial^\bullet}

\newcommand{\V}{\ensuremath{{\mathbb{V}}}}
\renewcommand{\P}{\ensuremath{{\mathbb{P}}}}

\allowdisplaybreaks

\definecolor{dred}{HTML}{7F0000}
\definecolor{dblue}{HTML}{073316}

\numberwithin{equation}{section}
\begin{document}
\hypersetup{
  colorlinks=true,
  citecolor=Cerulean,
  linkcolor=purple!80
}
\title{Free boundary limits of coupled bulk-surface models for receptor-ligand interactions on evolving domains}
\author{Amal Alphonse\thanks{Weierstrass Institute, Mohrenstrasse 39, 10117 Berlin, Germany ({\tt alphonse@wias-berlin.de})}  \and Diogo Caetano\thanks{Mathematics Institute, University of Warwick, Coventry CV4 7AL, United Kingdom ({\tt Diogo.Caetano@warwick.ac.uk})} \and Charles M. Elliott\thanks{Mathematics Institute, University of Warwick, Coventry CV4 7AL, United Kingdom ({\tt C.M.Elliott@warwick.ac.uk})} \and Chandrasekhar Venkataraman\thanks{Department of Mathematics, School of Mathematical and Physical Sciences, University of Sussex, Falmer BN1 9RF, United Kingdom (\tt{c.venkataraman@sussex.ac.uk})}}
\maketitle

\begin{abstract}
We derive various novel free boundary problems as limits of a coupled bulk-surface reaction-diffusion system modelling ligand-receptor dynamics on evolving domains. These limiting free boundary problems may be formulated as Stefan-type problems on an evolving hypersurface. Our results are new even in the setting where there is no domain evolution.  The models are of particular relevance to a number of applications in cell biology. The analysis utilises $L^\infty$-estimates in the manner of De Giorgi iterations and other technical tools, all in an evolving setting. We also report on numerical simulations.
\end{abstract}
\tableofcontents

\section{Introduction} In this paper we obtain a sequence of  novel free boundary limits of a system of reaction-diffusion equations holding on time-evolving bulk-surface domains by sending parameters in the equations to zero. The limit problems have  unknown  moving boundaries on the given evolving surface of the evolving bulk domain. The final example of these problems that we derive can be written as a degenerate parabolic equation similar to a Stefan or Hele-Shaw problem  holding on an evolving  hypersurface associated with a Dirichlet-to-Neumann map. Indeed, it consists of a function $u \geq 0$ which is harmonic on the evolving domain $\Omega(t)$  coupled through its normal derivative and a monotone inclusion to a surface quantity $v \leq 0$ that satisfies a parabolic equation on a portion $\Gamma(t)$ of the boundary $\partial\Omega(t)$:
\begin{equation*}
    \begin{aligned}
     \Delta u &= 0 &\text{ in } \Omega(t),\\
    \mdd v + v\dVp + \sgrad \cdot (\JG v)  &=  -\grad u \cdot \nu &\text{ on } \Gamma(t),\\
    v &\in \beta(u) &\text{ on } \Gamma(t).
    \end{aligned}
\end{equation*}
Here the graph $\beta$ is such that the last line essentially encodes a \textit{complementarity condition} of the form $uv=0$ on $\Gamma(t)$. This type of problem is related to  earlier works by two of the coauthors on degenerate  equations posed on an evolving surface \cite{AEStefan,AlpEll16}. Full details of all of these notations and concepts will, of course, be expounded upon in the course of the paper.

\subsection{Mathematical setting}
 For each $t \in [0,T]$, let $D(t) \subset \mathbb{R}^{d+1}$ be a smooth domain containing a $C^2$-hypersurface $\Gamma(t)$ which separates $D(t) = I(t) \cup \Omega(t)$ into an interior region $I(t)$ and an exterior Lipschitz domain $\Omega(t)$ (see \cref{fig:1}). 
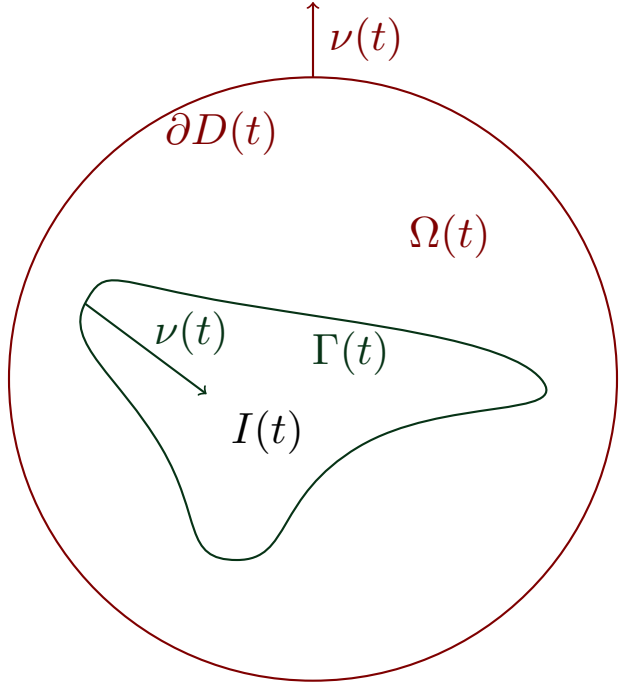
\begin{figure}[h]
\centering
\begin{tikzpicture}[thick, scale=2]

	 \draw[dred,thick,fill=none](0.5,-1.5) circle (2.0);
        \draw [dred,thick, ->] (0.5, 0.5) -- (0.5, 1)  ;
	\node[dred,scale=1.5,right] at (0.5,0.75) {$\nu(t)$};

        \draw [dblue,thick, ->]  (-1,-1.)  --  (-0.2,-1.6)  ;
	\node[dblue,scale=1.5,above] at (-0.3,-1.45) {$\nu(t)$};

\draw [dblue, thick] plot [smooth cycle, tension=1] coordinates {(-1,-1.) (-.5,-1.95) (0,-2.7) (.8,-1.95) (2,-1.5) (0,-1.)} ;
\node[scale=1.5,below] at (0.75,-1.08) {${\color{dblue}\Gamma(t)}$};
\node[scale=1.5,below] at (0.2,-1.6) {${\color{black}I(t)}$};
\node[scale=1.5] at (-0.1,0.13) {${\color{dred}\partial D(t)}$};
\node[scale=1.5,below] at (1.4,-.3) {${\color{dred}\Omega(t)}$};
\end{tikzpicture}
\caption{A sketch of the geometry.}
\label{fig:1}
\end{figure}
We suppose that the surface $\Gamma(t)$ and the outer boundary $\partial D(t)$ both evolve in time through prescribed kinematic normal velocity fields $\mathbf V$ and $\mathbf V_o$ respectively.   
  Motivated by a model of ligand-receptor dynamics we consider a system of reaction advection-diffusion equations. Within the evolving bulk  domain  $\Omega(t)$ there is a   
 concentration field $u$  subject to the diffusive and advective flux
 $$q_u:=-\nabla u+u\mathbf V_\Omega $$
 and within the evolving surface domain  $\Gamma(t)$ there are concentration fields $w,z$ 
  that  are subject to  the diffusive and advective fluxes
$$q_w:=-\delta_\Gamma\nabla_\Gamma w+w\mathbf V^\tau_\Gamma,\qquad q_z:=-\delta_\Gamma\nabla_\Gamma z+w\mathbf V^\tau_\Gamma.$$
Here $\mathbf V_\Omega$ and  $\mathbf V^\tau_\Gamma$ are prescribed advective velocity fields 
 satisfying
\[
\mathbf V_\Gamma := \mathbf V + \mathbf V^\tau_\Gamma,\qquad (\mathbf V_\Gamma(t) \cdot \nu(t))\nu(t) = \mathbf V(t) \text{ on $\Gamma(t)$},
\qquad
\text{and}
\qquad
(\VO(t) \cdot \nu(t))\nu(t) = \mathbf{V}_o(t) \text{ on $\partial D(t)$},\] 
with $\nu(t)$ denoting  the unit outward-pointing normal vector on $\partial\Omega(t).$ Here $\mathbf{V}_\Gamma^\tau$ can be understood as the tangential component of $\mathbf{V}_\Gamma$.


Balances of mass and reaction kinetics on the surface $\Gamma(t)$ lead to the following non-dimensional  system:
\begin{equation}\label{eq:model_without_Vp}
\begin{aligned}
 \delta_\Omega(u_t +\nabla\cdot\left(u\mathbf V_\Omega\right)) - \Delta u &=0 &&\text{in }\Omega(t),\\
\nabla u \cdot \nu - \delta_\Omega u\left(\mathbf V_\Omega- \mathbf V_{\Gamma}\right) \cdot \nu &=  \frac{1}{\delta_{k'}} z-\frac{1}{\delta_k}g(u,w) &&\text{on }\Gamma(t),\\
\grad u \cdot \nu &=0 &&\text{on $\partial D(t),$}\\
\partial^\circ w + w\sgrad \cdot \mathbf V  + \nabla_\Gamma\cdot\left(w\mathbf V_\Gamma^\tau\right)-\delta_\Gamma \Delta_\Gamma w&=\frac{1}{\delta_{k'}} z-\frac{1}{\delta_k}g(u,w)&&\text{on }\Gamma(t),\\
\partial^\circ z + z\sgrad \cdot \mathbf V  + \nabla_\Gamma\cdot\left(z\mathbf V_\Gamma^\tau\right)-\delta_{\Gamma'}\Delta_\Gamma z&=\frac{1}{\delta_k}g(u,w)-\frac{1}{\delta_{k'}} z&&\text{on }\Gamma(t),
\end{aligned}
\end{equation}
where $\partial^\circ w = w_t+\nabla w \cdot \mathbf{V}$ is the normal time derivative (see \cite{CFG, Dziuk2013}), $\sgrad=\grad_{\Gamma(t)}$ and $\Delta_\Gamma = \Delta_{\Gamma(t)}$ denote the tangential gradient and Laplace--Beltrami operator respectively on $\Gamma(t)$ (for ease of reading, we omit the dependence on time and simply write $\sgrad$ and $\slap$), 
 the dimensionless constants $\delta_\Omega, \delta_\Gamma, \delta_{\Gamma'},\delta_k$ and $\delta_{k'}$ are positive and  $g\colon \R^2 \to \R$ is a given function satisfying \cref{ass:ong} below. We endow the system above with non-negative and bounded initial data:
\begin{equation*}\label{IC}
\begin{aligned}
(u(0), w(0), z(0)) &=(u_0, w_0, z_0) \in L^{\infty}(\Omega_0)\times L^\infty(\Gamma_0)^2  &&\text{and }\;\; u_0, w_0, z_0\geq0,
\end{aligned}
\end{equation*}
where $\Omega_0:=\Omega(0)$ and $\Gamma_0:=\Gamma(t)$. In terms of the biological application we have in mind, the quantity $u$ represents the concentration of ligands in the domain, $w$ the concentration of receptors on the surface and $z$ the concentration of the complex formed by ligand-receptor binding (see \cite{ERV, AET} and references therein for more details). The system \eqref{eq:model_without_Vp} is a reaction-diffusion system on an evolving domain 
and details of its derivation can be found in \cref{sec:nondim}. 
Regarding $g$, there are two important examples to have in mind: 
\begin{equation}
g(u,w) = uw\label{eq:g_usual}
\end{equation}
and
\begin{equation}
g(u,w) = \frac{u^nw}{1+u^n} \text{ for $n>1$,}\label{eq:g_mm}
\end{equation}
representing the biologically relevant cases of quadratic binding and cooperative (Hill function or Michaelis--Menten) binding respectively.

The various parameters appearing in \eqref{eq:model_without_Vp} correspond to rates of reaction or diffusion. 
The purpose of this paper is  to derive various limit problems that exhibit free boundary features on the evolving surface $\Gamma(t)$
by sending combinations of the parameters $ \delta_\Omega, \delta_k, \delta_{k'}^{-1}, \delta_\Gamma, \delta_\Gamma' $   to zero.  Let us highlight some of the difficulties:
\begin{itemize}\itemsep=0pt
    \item Due in part to the nonlinear nature of $g$, it become necessary to obtain $L^\infty$-estimates on the solutions of \eqref{eq:model_without_Vp}. This is non-trivial because of the bulk-surface coupling and the Robin boundary condition for $u$. We overcome this with De Giorgi arguments along the same vein as in \cite{AET}.
    \item As (uniform) estimates on the time derivatives of the solutions is an issue, lack of compactness causes difficulties. We work around this by deriving estimates on difference quotients on pulled back equations (which are essentially parabolic PDEs with time-dependent coefficients).
    \item In the $\delta_k \to 0$ limit the diffusion coefficients for the two surface quantities $w$ and $z$ are generally different (the system \eqref{eq:model_without_Vp} is of \textit{cross-diffusion} type); this means that one cannot simply add the two equations in order to `cancel out' the right-hand sides to derive estimates. Instead, we utilise a duality approach.
    \item The geometry involved in the problem is non-standard: $\Omega(t)$ is an annular domain and the interface $\Gamma(t)$ is only one part of $\partial\Omega(t)$, which sometimes leads to complications in the analysis. Furthermore, the domain is evolving in time, calling for the use of special time-evolving Bochner spaces and related theory in the spirit of \cite{AESAbstract, MR4538417}.
\end{itemize}
It is worth writing the system in the case where there is no movement or evolution of the domain:
\begin{subequations}\label{eq:modelStationary}
\begin{align}
 \delta_\Omega u_t - \Delta u &=0 &&\text{in }\Omega,\\
\nabla u \cdot \nu &=  \frac{1}{\delta_{k'}}z-\frac{1}{\delta_k}g(u,w) &&\text{on }\Gamma,\\
\grad u \cdot \nu &=0 &&\text{on $\partial D$},\label{eq:stationaryModelBCD}\\
w_t -\delta_\Gamma \Delta_\Gamma w&=\frac{1}{\delta_{k'}}z-\frac{1}{\delta_k}g(u,w)&&\text{on }\Gamma,\\
z_t-\delta_{\Gamma'}\Delta_\Gamma z&=\frac{1}{\delta_k}g(u,w)-\frac{1}{\delta_{k'}}z&&\text{on }\Gamma,\\
(u(0),w(0),z(0)) &= (u_0,w_0,z_0).
\end{align}
\end{subequations}
We emphasise that the results of this paper are new even in this case where there is no evolution. 
\subsection{Outline of paper and our contributions}

In \cref{sec:bio_background} we detail some background on the biological motivation behind the system \eqref{eq:model_without_Vp}. We then set up the function spaces and define some notation in \cref{sec:notation_problem_setting}. Our main results are set out in the subsections of \cref{MainResults}. We present in \cref{sec:delk_limit} the system that arises in the $\delta_k \to 0$ limit, in \cref{sec:delk_delg_limit} the $\delta_k = \delta_\Gamma = \delta_{\Gamma'} \to 0$ limit, and in \cref{sec:delo_delk_delkp_delg_limit} the $\delta_\Omega = \delta_k= \delta_{k'}^{-1}=\delta_\Gamma = \delta_{\Gamma'}  \to 0$ limit. In \cref{sec:dirichlet_case_three_limit}, we replace the Neumann condition on the outer boundary $\partial D(t)$ with a Dirichlet condition and study the same limit as \cref{sec:delo_delk_delkp_delg_limit}. In \cref{sec:deltakFBP}, we give free boundary reformulations of the limiting systems that we derived as degenerate parabolic equations on  evolving surfaces. \cref{sec:nonMoving} presents our results as applied to the stationary, non-evolving setting (which is the usual setting in much of the literature), i.e., to the system \eqref{eq:modelStationary}. We explore in \cref{sec:bio_implications} the implications of our results in the context of the biological application that motivated this work.

Regarding the rest of the paper, in \cref{sec:preliminary_results} we give some $L^2$- and $L^\infty$-estimates on solutions of the heat equation on evolving domains and surfaces. \cref{sec:uniform_bounds} contains various estimates related to \eqref{eq:model_without_Vp} that we shall need in obtaining the existence of limits of the solutions as well as that of other related quantities. From \cref{sec:deltakLimit} to \cref{sec:thirdLimit}, we prove the results on each of the limits that were stated in \cref{sec:delk_limit}---\cref{sec:dirichlet_case_three_limit}. We conclude the paper proper by presenting some numerical simulations illustrating our theory in \cref{sec:numerical}.  We collect in \cref{sec:app_identities} some useful identities related to integration by parts on evolving spaces, in \cref{sec:app_on_g} results related to the function $g$, and in \cref{sec:nondim} we justify our initial model and our motivation for taking limits by nondimensionalisation.

Let us briefly situate our work. This paper is closely related to \cite{AET, ERV}. In \cite{ERV} singular limits of a simplified system (where $z$ is neglected) are studied on a non-evolving domain. The paper \cite{AET} considers the well posedness of \eqref{eq:newModel} in an evolving domain but not the limiting behaviour. Note that in both works the choice $g(u,w)=uw$ is fixed; here we allow for greater generality in that too. See, as mentioned, \cite{AET} for existence results in the case of $g(u,w)=uw$ and \cite{CET} for existence results for a larger class of $g$ (but with a different geometric setup). 
\subsection{Biophysical background}\label{sec:bio_background}

A number of recent theoretical  and computational works seek to model receptor-ligand interactions using coupled bulk-surface systems of PDEs in the context of cell biology, employing models which are similar in structure to those considered in this work typically on fixed domains e.g., \cite{garcia2013mathematical}. Models with similar features arise in the modelling of signalling networks  coupling the dynamics of ligands  within the cell (e.g., G-proteins) with those on the cell surface \citep{levine2005membrane, jilkine2007mathematical,goryachev2008dynamics,mori2008wave,MadChuVen14,sharma2016global,fellner2018well}; these  again have primarily been considered on fixed domains. The rigorous derivation of reduced models  based on biologically relevant asymptotic limits of coupled bulk-surface models on fixed domains in the context of cell biology has been the subject of a large number of recent works applied to phenomena such as pattern formation \cite{ratz2013symmetry, borgqvist2021cell}, lipid raft formation \citep{garcke2016coupled}, cell signalling \citep{marciniak2008derivation,ptashnyk2020multiscale}, and polarisation \citep{logioti2021parabolic,logioti2023qualitative,logioti2024interface}.
In most applications however, the motion of cells necessitates the consideration of evolving domains. For example, in chemotaxis the motion of cells has been conjectured to act as a barrier to successful gradient sensing that must be overcome \cite{endres2008accuracy},  possibly via the ability of cells to  create their own chemotactic gradients, i.e., to influence the bulk ligand field \cite{mclennan2012multiscale,mclennan2015vegf,mclennan2015neural}. We note the recent computational works in this direction \cite{macdonald2016computational,mackenzie2016local}. We also mention \cite{MR2997555} for a framework and theoretical study of bulk-surface systems (among other problems) and their relationship to Onsager systems. Recently, existence and well posedness theory has been developed for bulk-surface systems in the evolving domain setting \cite{AET,CET}.  There are however, to our knowledge, no corresponding results on asymptotic limits of bulk-surface systems on evolving domains. Motivated by this need,  in this work we focus on understanding asymptotic limits of  theoretical models for receptor-ligand dynamics in cell 
biology consisting of a coupled system of bulk-surface partial differential equations on evolving domains. 

Whilst our focus is on receptor-ligand dynamics in biological cells, problems of a similar structure arise in fields such as semiconductors \cite{MR4301806} and, in particular, in ecology where one considers populations consisting of two or more 
competing species \cite{holmes1994partial}. Such an ecological scenario can be modelled by so-called spatial segregation models and the corresponding asymptotic limits have been the subject 
of much mathematical study, e.g., \cite{conti2005asymptotic,crooks2004spatial,dancer1999spatial,hutridurga2018heterogeneity}. Although domain evolution may be less relevant in this setting, our results are new even on fixed domains, and hence potentially relevant. 

A number of Stefan-type free boundary problems have been derived as fast-reaction limits of systems of parabolic equations in fixed domains, for a survey see \cite[Chapter  10]{perthame2015parabolic}. We further note  that free boundary problems of Stefan or Hele-Shaw type arise as singular limits of models for tumour growth \cite{perthame2014hele}.

\subsection{Notation and functional framework}\label{sec:notation_problem_setting}

It is convenient to parametrise the  evolving geometric domain and function spaces defined on it  using  a continuously differentiable velocity field  $\mathbf V_p\colon [0,T] \times \mathbb{R}^{d+1} \to \mathbb{R}^{d+1}$  
for which there exists a flow $\Phi\colon [0,T]\times\R^{d+1}\rightarrow\R^{d+1}$ such that 
\begin{enumerate}[label=({\roman*}), itemsep=0pt]
\item $\Phi_t:=\Phi(t,\cdot)\colon \overline{\Omega_0}\to \overline{\Omega(t)}$ is a $C^2$-diffeomorphism with $\Phi_t(\Gamma_0) = \Gamma(t)$ and $\Phi_t(\partial D_0) = \partial D(t)$,
\item $\Phi_t$ solves the ODE 
\begin{align*}
\frac{d}{dt}\Phi_t(\cdot)&=\mathbf V_p(t,\Phi_t(\cdot)),\\
\Phi_0&=\rm{Id}.
\end{align*}
\end{enumerate}
Note that
\begin{equation}\label{VP}(\Vp \cdot \nu)\nu = \mathbf{V} \text{ on $\Gamma(t)$} \quad\text{and}\quad (\Vp\cdot \nu)\nu = \mathbf{V}_o \text{ on $\partial D(t)$}.\end{equation}
Functions on the evolving domain are paramerised over the initial domain in the following way.  We define the mapping $\phi_{-t}$ acting on a function $u\colon \Omega(t) \to \mathbb{R}$ by $\phi_{-t}u := u \circ \Phi_t$; this is a linear homeomorphism (its inverse will be denoted by $\phi_t$) between the spaces $L^p(\Omega(t))$, $H^1(\Omega(t))$ and the reference spaces $L^p(\Omega_0)$, $H^1(\Omega_0)$ respectively \cite{MR4538417, AESIFB, AEStefan}. The same holds for the corresponding $L^p$ and Sobolev spaces on $\Gamma$ and $\partial D$ \cite{AEStefan} for $\phi_{-t}$ given by the same expression. If we also assume
\begin{enumerate}[label=({\roman*}), itemsep=0pt]
\item $\Phi_t:=\Phi(t,\cdot)\colon \Gamma_0 \to \Gamma(t)$ is a $C^3$-diffeomorphism,
\item $\Phi_{(\cdot)} \in C^3([0,T]\times \Gamma_0)$,
\end{enumerate}
then the above property also holds for the fractional Sobolev space $H^{1\slash 2}(\Gamma(t))$ \cite[\S 5.4.1]{AESIFB}. We may then define the Banach spaces $L^p_{L^q(\Omega)}$, $L^2_{H^1(\Omega)}$, $L^p_{L^q(\Gamma)}$, $L^2_{H^1(\Gamma)}$, $L^2_{H^{1\slash 2}(\Gamma)}$, which are Hilbert spaces for $p=q=2$; these are the evolving versions of the usual Bochner spaces to handle (sufficiently regular) time-evolving Banach spaces $X \equiv \{X(t)\}_{t \in [0,T]}$. To be precise, the notation $L^p_X$ stands for 
\begin{align*}
L^p_X &:= \left\{u:[0,T] \to \bigcup_{t \in [0,T]}\!\!\!\! X(t) \times \{t\},\;\; t \mapsto (\hat u(t), t)\;\; \mid \;\; \phi_{-(\cdot)} \hat u(\cdot) \in L^p(0,T;X_0 )\right\}
\end{align*}
where for each $t \in [0,T]$, the map $\phi_{-t}\colon X(t) \to  X_0$ is a given linear homeomorphism satisfying certain properties; the corresponding norm is (identifying $\hat u$ with $u$)
\[\norm{u}{L^p_{X}} := \left(\int_0^T \norm{u(t)}{X(t)}^p\right)^{\frac 1p}\]
(with the obvious modification for $p=\infty$). 
We also define, for $k\in \mathbb{N} \cup \{0\}$, the space $C^k_X$ of $k$-times continuously differentiable functions on $[0,T]$:
\begin{align*}
C^k_X &= \left\{\eta\colon [0,T] \to \bigcup_{t \in [0,T]}\!\!\!\! X(t) \times \{t\},\;\; t \mapsto (\eta(t), t)\;\; \mid \;\; \phi_{-(\cdot)}\eta(\cdot) \in C^k([0,T];X_0)\right\}.
\end{align*}
For a sufficiently smooth quantity $u$ defined in $\Omega(t)$ its (classical or strong) parametric material derivative is given by
\begin{equation}\label{eq:smd}
\md_\Omega u = u_t +\nabla u\cdot\mathbf V_p
\end{equation}
(see \cite{MR4538417, AESAbstract, AESIFB} and references therein). This derivative takes into account that spatial points $x=x(t)$ also depend on time and that their trajectory has been parametrised by the velocity field $\mathbf V_p$. For a sufficiently smooth quantity $w$ defined on the evolving surface $\Gamma(t)$, the classical parametric material derivative has the expression
\[\md_\Gamma w := \partial^\circ w + \sgrad w \cdot \Vp.\]The paper \cite{AESAbstract} (see also \cite{MR4538417}) also defines a weaker version of the above-introduced parametric material derivative (i.e., a notion of a weak time derivative in an evolving space) which we can call the weak parametric material derivative. This generalises \eqref{eq:smd}: a function $u \in L^2_{H^1(\Omega)}$ has a weak parametric material derivative (or weak time derivative) $\md_\Omega u \in L^2_{H^{1}(\Omega)^*}$ if and only if
\begin{equation*}
\int_0^T \langle \md_\Omega u(t), \eta(t) \rangle_{H^{1}(\Omega(t))^*, H^1(\Omega(t))} = - \int_0^T \int_{\Omega(t)}u(t) \md_\Omega \eta(t) - \int_0^T \int_{\Omega(t)}u(t)\eta(t)\grad \cdot \mathbf V_p
\end{equation*}
holds for all functions $\eta$ that are smooth and compactly supported (in time) with values in $H^1(\Omega(t))$, where $\md_\Omega \eta(t)$ is given by the formula \eqref{eq:smd}. A similar definition with the right modifications defines the weak parametric material derivative $\md_\Gamma w \in L^2_{H^{-1}(\Gamma)}$ for a function $w \in L^2_{H^1(\Gamma)}$:
\begin{equation*}
\int_0^T \langle \md_\Gamma w(t), \eta(t) \rangle_{H^{-1}(\Gamma(t)), H^1(\Gamma(t))} = - \int_0^T \int_{\Gamma(t)}w(t) \md_\Gamma \eta(t) - \int_0^T \int_{\Gamma(t)}w(t)\eta(t)\sgrad \cdot \mathbf V_p.
\end{equation*}
We shall not use the notation $\md_\Omega, \md_\Gamma$ further  but will simply write the weak parametric material derivative as $\mdd$. With this in hand, we define the evolving versions of Sobolev--Bochner spaces:
\begin{align*}
\mathbb{W}(H^1(\Omega), H^1(\Omega)^*) &:= \left\{ u \in L^2_{H^{_1}(\Omega)} \mid \mdd u \in L^2_{H^{1}(\Omega)^*}\right\}\\
\mathbb{W}(H^1(\Omega), L^2(\Omega) ) &:= \left\{ u \in L^2_{H^1(\Omega)} \mid \mdd u \in L^2_{L^2(\Omega)}\right\}
\end{align*}
(and similarly the corresponding spaces on $\{\Gamma(t)\}$) and more generally
\begin{align*}
\mathbb{W}(X, Y) &= \left\{ u \in L^2_{X} \mid \mdd u \in L^2_{Y}\right\}
\end{align*}
given families of time-evolving Banach spaces $X$ and $Y$. We alternatively sometimes use the notation
\[H^1_Y := \mathbb{W}(Y, Y).\]
We do not give the precise technical details and properties of these spaces here but refer to \cite{MR4538417} (and also \cite{AESAbstract, AESIFB, AEStefan}) for the interested reader. At this point note that in \cref{sec:app_identities} we collect some useful integration by parts identities that will be used throughout the paper.

\section{Free boundary problems on moving domains as limit systems; main results}\label{MainResults}

From \eqref{eq:model_without_Vp}, by adding and subtracting the parametrised velocity $\Vp$ (see \cref{sec:notation_problem_setting}), we obtain the model
\begin{subequations}\label{eq:newModel}
\begin{align}
 \delta_\Omega (\mdd{u}+ u\nabla\cdot\mathbf V_p+\nabla \cdot\left(\mathbf{J}_\Omega u\right)) - \Delta u  &=0&\text{ in }\Omega(t),\\
\nabla u \cdot \nu - \delta_\Omega ju&=  \frac{1}{\delta_{k'}} z-\frac{1}{\delta_k}g(u,w) &\text{ on }\Gamma(t),\\
\grad u \cdot \nu &=0 &\text{on $\partial D(t)$,}\label{eq:newModelNeumannBC}\\
 \mdd{w} + w\nabla_\Gamma\cdot\mathbf V_p-\delta_\Gamma\slap w +\nabla_\Gamma \cdot\left(\mathbf{J}_\Gamma w\right) &=   \frac{1}{\delta_{k'}} z-\frac{1}{\delta_k}g(u,w)&\text{ on }\Gamma(t),\\
 \mdd{z} + z\nabla_\Gamma\cdot\mathbf V_p-\delta_{\Gamma'}\slap z+\nabla_\Gamma \cdot\left(\mathbf{J}_\Gamma z\right) &=   \frac{1}{\delta_k}g(u,w)-\frac{1}{\delta_{k'}} z&\text{ on }\Gamma(t),\\
 (u(0),w(0),z(0)) &= (u_0,w_0,z_0).
\end{align}
\end{subequations}
In the above we use the notation 
\begin{equation}\label{eq:defn_jumps}
\mathbf J_\Omega := \mathbf V_\Omega - \mathbf V_p, \qquad \mathbf{J}_{\Gamma} :=\mathbf V_\Gamma - \mathbf V_p\qquad\text{and}\qquad j:= (\mathbf{V}_\Omega-\mathbf{V}_\Gamma)\cdot \nu = \mathbf J_\Omega|_\Gamma \cdot \nu.
\end{equation}
Observe that $j$ is the jump in the normal velocities on $\Gamma$ and that due to the compatibility conditions, $\mathbf{J}_{\Gamma} = \mathbf V_\Gamma^\tau - \mathbf V_p^\tau$ has no normal component, i.e., $\JG\cdot \nu =0$ on $\Gamma$, and we also have $\JO \cdot \nu = 0$ on $\partial D$. 

\begin{remark} The system is  independent of the choice of  $\Vp$ provided \eqref{VP} holds. The analysis uses function spaces (introduced in \cref{sec:notation_problem_setting}) that do depend on $\Vp$. 
Introducing this  velocity field may be convenient in applications. For example  in numerical simulations it may be used to avoid degenerate meshes \cite{Dziuk2013, Elliott2012}. \end{remark}

\begin{remark}
    The mass of $w+z$ is conserved: using 
    the equation satisfied by $w+z$ and the fact that $\JG$ is purely tangential, we can show
\begin{align*}
\frac{d}{dt}\intGt w+z = 0. 
\end{align*}
In a similar way 
the mass of $\delta_\Omega u + z$ is conserved:
\[\frac{d}{dt}\left(\delta_\Omega \intOt u + \intGt z\right) =0.\]
\end{remark}In the following, we will take various limits in \eqref{eq:newModel} and end up with systems that resemble this system. 
We do not, in this work, consider the well posedness of the problem \eqref{eq:newModel} but instead we will assume it. In fact, we make the following standing assumption.
\begin{ass}[Standing assumption]
Regarding \eqref{eq:newModel}, we take non-negative initial data $u_0 \in H^1(\Omega_0) \cap L^\infty(\Omega_0)$ and $w_0, z_0 \in H^1(\Gamma_0)\cap L^\infty(\Gamma_0)$, and we will \textit{assume} the existence of a unique non-negative solution 
\[u \in L^2_{H^1(\Omega)}\cap H^1_{H^1(\Omega)^*} \quad\text{and}\quad w,z \in L^2_{H^1(\Gamma)}\cap H^1_{L^2(\Gamma)}\]
with
\[g(u,w) \in L^2_{L^2(\Gamma)}.\]
\end{ass}
Note in particular that we assume non-negativity of solutions; this is reasonable for the application we have in mind and can be proved under (the assumed) non-negative initial data, uniqueness of solutions and reasonable assumptions on $g$ (see \cite[Theorem 1.1]{CET}); in particular
\[g(u,0) \leq 0 \quad\text{and}\quad g(0,w) \leq 0 \qquad \forall u, w \geq 0\]
would be sufficient. Indeed, we may replace the reaction terms in \eqref{eq:newModel} by $\frac{1}{\delta_{k'}}z^+ - \frac{1}{\delta_k}g(u^+,w^+)$ and by testing the resulting system with the negative parts of the solutions, we can show that its solutions are non-negative under the condition on $g$ above. By uniqueness, it follows that the solution of \eqref{eq:newModel} is also non-negative.

We state an assumption on the function $g$. 
\begin{ass}[Assumptions on $g$]\label{ass:ong}
Throughout the paper, the function $g\colon \mathbb{R}^2 \to \mathbb{R}$ satisfies
\begin{align*}
u, w \geq 0 \implies g(u,w) \geq 0.
\end{align*}
When viewing $g$ as a superposition operator, we will in different theorems use one of the following assumptions as and when needed:
\begin{align}
\text{$u_n \weaklyto u$ in $L^2_{H^{1\slash 2}(\Gamma)}$, $\quad v_n \to v$ in $L^2_{L^2(\Gamma)}, \quad g(u_n, v_n) \to 0$ in $L^1_{L^1(\Gamma)}  \implies g(u,v)=0$}\label{ass:new_on_g_1}\\
\nonumber \text{$u_n \to u$ in $L^2_{L^2(\Gamma)}, \quad u_n \weakstar u $ in $ L^\infty_{L^\infty(\Gamma)},$ $\quad v_n \weakstar v$ in $L^\infty_{L^\infty(\Gamma)}, \quad g(u_n, v_n) \to 0$ in $L^1_{L^1(\Gamma)}$}\\
\qquad\text{ $\implies g(u,v) = 0$ in $L^1_{L^1(\Gamma)}$}\label{ass:new_on_g_2}
\end{align}
\end{ass}
\begin{prop}\label{lem:verification_of_g1}
Both choices of $g$ in \eqref{eq:g_usual} and \eqref{eq:g_mm} satisfy assumptions \eqref{ass:new_on_g_1} and \eqref{ass:new_on_g_2}.
\end{prop}
In order not to break the flow of the section, the proof of this result has been placed in \cref{sec:app_on_g}. Now, to write down the limiting systems that we obtain, let us first introduce the space of test functions:
\begin{equation}
\mathcal{V} := \left\{\eta \in L^2_{H^1(\Omega)}\cap H^1_{L^2(\Omega)} : \eta|_{\Gamma} \in L^2_{H^1(\Gamma)} \cap H^1_{L^2(\Gamma)}\right\}\label{eq:defnTestFnSpace}
\end{equation}
where by the restriction above we mean in the sense of the Sobolev trace, applied pointwise a.e. in time:
\[(\eta|_{\Gamma})(t) := \tau_t\eta(t)\]
with $\tau_t\colon H^1(\Omega(t)) \to H^{1\slash 2}(\Gamma(t))$ the trace operator. 
We will usually omit explicitly writing the trace operator in what follows.

\subsection{The $\delta_k \to 0$ limit}\label{sec:delk_limit}
We first consider the \textbf{fast binding} limit $\delta_k \to 0$. 
\begin{defn}
We say that  $(u, w, z) \in L^2_{H^1(\Omega)} \times (L^2_{H^1(\Gamma)})^2$ is a \emph{weak solution} of the problem
\begin{subequations}\label{eq:systemdeltakLimit}
\begin{align}
\delta_\Omega(\mdd u + u\dVpO + \grad \cdot (\JO u)) - \Delta u  &=0 &\text{ in } \Omega(t), \\
\grad u \cdot \nu &=0 &\text{on $\partial D(t),$}\\
\mdd w - \delta_\Gamma \slap w + w\dVp + \sgrad \cdot (\JG w) &= \grad u \cdot \nu -\delta_\Omega ju &\text{ on } \Gamma(t), \label{eq:deltakLimitSystem_w}\\
\mdd z - \delta_{\Gamma'}\slap z + z\dVp + \sgrad \cdot (\JG z) &=- \grad u \cdot \nu + \delta_\Omega ju&\text{ on } \Gamma(t), \label{eq:deltakLimitSystem_z}\\
g(u,w) &=0 &\text{ on } \Gamma(t),\label{eq:gIsZero}\\
(u(0),w(0),z(0)) &= (u_0,w_0,z_0),
\end{align}
\end{subequations}
if for every $\eta \in \mathcal{V}$ with  $\eta(T)=0$, 
\begin{subequations}
\begin{align}
\nonumber &-\int_{\Omega_0}u_0\eta(0) + \frac{1}{\delta_\Omega}\int_{\Gamma_0}w_0\eta(0)      -\int_0^T \int_{\Omega(t)} u \mdd\eta + \frac{1}{\delta_\Omega}\int_0^T \intGt w \mdd \eta + \frac{1}{\delta_\Omega} \int_0^T\int_{\Omega(t)} \nabla u\cdot \nabla \eta \\
& - \frac{\delta_{\Gamma}}{\delta_\Omega}\int_0^T\intGt \tgrad w\cdot \tgrad \eta  + \int_0^T\int_{\Omega(t)} \grad \cdot (\JO u)\eta - \frac{1}{\delta_\Omega}\int_0^T\int_{\Gamma(t)} \sgrad \cdot (\JG w)\eta  - \int_0^T\int_{\Gamma(t)}ju\eta  = 0,\label{eq:vws_u_and_w}\\
\nonumber &-\int_{\Omega_0}u_0\eta(0) -\frac{1}{\delta_\Omega}\int_{\Gamma_0}z_0\eta(0)      -\int_0^T \int_{\Omega(t)} u \mdd\eta - \frac{1}{\delta_\Omega}\int_0^T \intGt z \mdd \eta + \frac{1}{\delta_\Omega} \int_0^T\int_{\Omega(t)} \nabla u\cdot \nabla \eta \\
&   + \frac{\delta_{\Gamma'}}{\delta_\Omega}\int_0^T\intGt \tgrad z\cdot \tgrad \eta + \int_0^T\int_{\Omega(t)} \grad \cdot (\JO u)\eta + \frac{1}{\delta_\Omega}\int_0^T\int_{\Gamma(t)} \sgrad \cdot (\JG z)\eta - \int_0^T\int_{\Gamma(t)}ju\eta  = 0,\label{eq:vws_u_and_z}\\
&g(u,w) =0 \text{ in } L^1_{L^1(\Gamma)}\label{eq:vws_gIsZero}.
\end{align}
\end{subequations}    
\end{defn}

Our notion of solution is very weak in the sense that we put temporal derivatives onto the test function and the normal derivative of $u$ has to be understood also in a very weak sense (see \cref{rem:veryweakND}). 

In \cref{sec:deltakLimit}, we will prove the following. The result requires a restriction on the dimension because we need an $L^\infty$-estimate on $w$ (see \cref{lem:wLinfty}). The said $L^\infty$ estimate can be obtained in a different (easier) way for the other limits, hence no restriction on the dimension is needed in those limiting regimes.

\begin{theorem}[The $\delta_k$ limit]\label{thm:deltakLimit}
Assume \eqref{ass:new_on_g_1} and let $d \leq 3$. As $\delta_k\to 0$, the solution $(u_k, w_k, z_k)$ of \eqref{eq:newModel} converges to a weak solution $(u, w, z) \in L^2_{H^1(\Omega)} \times (L^2_{H^1(\Gamma)})^2$ of the problem \eqref{eq:systemdeltakLimit}.
\end{theorem}
\begin{remark}
Note that $z$ is uncoupled from the other quantities; indeed we can solve for $u$ and define $z$ as the solution of \eqref{eq:deltakLimitSystem_z} along with $z(0) = z_0$ in some sense. Both $u$ and $w$ are linked through the complementarity condition \eqref{eq:gIsZero}.
\end{remark}


\begin{remark}\label{rem:veryweakND}
The normal derivatives appearing in \eqref{eq:deltakLimitSystem_w} and \eqref{eq:deltakLimitSystem_z} are to be understood in a formal sense; they are not even weak normal derivatives (we would need $\Delta u$ in $L^2(\Omega(t))$ for that).
\end{remark}

\subsection{The $\delta_k=\delta_\Gamma = \delta_{\Gamma'} \to 0$ limit}\label{sec:delk_delg_limit}
In this section we look at the \textbf{fast binding, vanishing surface diffusion} limit $\delta_k=\delta_\Gamma = \delta_{\Gamma'} \to 0$. 

\begin{defn}
We say that $(u, w, z) \in L^2_{H^1(\Omega)} \times (L^2_{L^2(\Gamma)})^2$ is a \emph{weak solution} of the problem
\begin{subequations}\label{eq:systemthreeDeltasLimit}
\begin{align}
\delta_\Omega(\mdd u  + u\dVp + \grad \cdot (\JO u))- \Delta u  &=0 &\text{ in } \Omega(t), \\
\grad u \cdot \nu &=0 &\text{on $\partial D(t),$}\\
\mdd w + w\dVp + \sgrad \cdot (\JG w)  &=  \grad u \cdot \nu -\delta_\Omega ju &\text{ on } \Gamma(t), \\
\mdd z + z \dVp + \sgrad \cdot (\JG z) &=- \grad u \cdot \nu + \delta_\Omega ju &\text{ on } \Gamma(t), \\
g(u,w) &=0 &\text{ on } \Gamma(t),\label{eq:gIsZeroThree}\\
(u(0),w(0),z(0)) &= (u_0,w_0,z_0),
\end{align}
\end{subequations}    
if for every $\eta \in \mathcal{V}$ with $\eta(T)=0$,
\begin{subequations}
\begin{align}
\nonumber -\int_{\Omega_0}u_0\psi(0) + \frac{1}{\delta_\Omega}\int_{\Gamma_0}w_0\psi(0)  &- \int_0^T \int_{\Omega(t)} u \mdd\eta + \frac{1}{\delta_\Omega}\int_0^T \int_{\Gamma(t)} w \mdd \eta  + \frac{1}{\delta_\Omega}\int_0^T\int_{\Omega(t)} \nabla u\cdot \nabla \eta \\ & + \int_0^T\int_{\Omega(t)}   \grad \cdot (\JO u)\eta + \frac{1}{\delta_\Omega} \int_0^T \int_{\Gamma(t)} w\JG \cdot \sgrad \eta - \int_0^T\int_{\Gamma(t)} ju \eta = 0,\label{eq:vws_u_and_w_2ndlimit}\\
\nonumber    -\int_{\Omega_0}u_0\psi(0)-\frac{1}{\delta_\Omega}\int_{\Gamma_0}z_0\psi(0)  &-\int_0^T \int_{\Omega(t)} u \mdd\eta - \frac{1}{\delta_\Omega}\int_0^T \int_{\Gamma(t)} z \mdd \eta  + \frac{1}{\delta_\Omega}\int_0^T\int_{\Omega(t)} \nabla u\cdot \nabla \eta \\ & + \int_0^T\int_{\Omega(t)}   \grad \cdot (\JO u)\eta -\frac{1}{\delta_\Omega} \int_0^T \int_{\Gamma(t)} z\JG \cdot \sgrad \eta - \int_0^T\int_{\Gamma(t)} ju \eta = 0,\label{eq:vws_u_and_z_2ndlimit}\\
&g(u,w) =0 \text{ in } L^1_{L^1(\Gamma)}\label{eq:vws_gIsZero_2ndlimit}.
\end{align}
\end{subequations}

\end{defn}

Note that we have weakened the spatial regularity for $w$ and $z$ from \cref{thm:deltakLimit} even further: the divergence of the jump terms are satisfied in a weak sense, i.e., the test function carries the derivatives. 
The next theorem will be shown in \cref{sec:deltakdeltaGammasLimit}.
\begin{theorem}[The $\delta_k=\delta_\Gamma = \delta_{\Gamma'}$ limit]\label{thm:threeDeltasLimit}
Assume \eqref{ass:new_on_g_2}. As $\delta_k = \delta_\Gamma = \delta_{\Gamma'} \to 0$, the solution $(u_k, w_k, z_k)$ of \eqref{eq:newModel} converges to a  weak solution $(u, w, z) \in L^2_{H^1(\Omega)} \times (L^2_{L^2(\Gamma)})^2$ of the problem \eqref{eq:systemthreeDeltasLimit}.
\end{theorem}

\subsection{The $\delta_\Omega = \delta_k= \delta_{k'}^{-1}=\delta_\Gamma = \delta_{\Gamma'}  \to 0$ limit}\label{sec:delo_delk_delkp_delg_limit}
Finally, we study the \textbf{fast binding, vanishing surface diffusion, quasi-steady bulk, slow unbinding} limit $\delta_\Omega = \delta_k= \delta_{k'}^{-1}=\delta_\Gamma = \delta_{\Gamma'}  \to 0$.

\begin{defn}
    We say that $(u, w, z) \in L^2_{H^1(\Omega)} \times (L^2_{L^2(\Gamma)})^2$ is a \emph{weak solution}  of the problem 
\begin{subequations}\label{eq:systemFourDeltasLimit}
\begin{align}
\grad u &\equiv 0 &\text{ in } \Omega(t), \label{eq:gradUIsZeroFour}\\
\mdd w + w\dVp + \sgrad \cdot (\JG w)  &=  0 &\text{ on } \Gamma(t), \\
\mdd z + z \dVp + \sgrad \cdot (\JG z) &= 0 &\text{ on } \Gamma(t), \\
g(u,w) &=0 &\text{ on } \Gamma(t),\label{eq:gIsZeroFour}\\
(w(0),z(0)) &= (w_0,z_0),
\end{align}
\end{subequations}
if for every $\eta \in \mathcal{V}$ with $\eta(T)=0$,
\begin{subequations}
\begin{align}
\nonumber \grad u &=  0 \text{ in } L^1_{L^1(\Omega)},\\
\nonumber  \int_{\Gamma_0}w_0\eta(0) + \int_0^T\intGt w \mdd \eta + w\JG\cdot \sgrad \eta  &= 0,\\
\nonumber   \int_{\Gamma_0}z_0\eta(0) + \int_0^T\intGt z \mdd \eta + z\JG\cdot \sgrad \eta & = 0,\\
\nonumber g(u,w) &=0 \text{ in } L^1_{L^1(\Gamma)}.
\end{align}
\end{subequations}
\end{defn}

The proof of the following theorem can be found in \cref{sec:thirdLimit}.
\begin{theorem}[The $\delta_\Omega = \delta_k= \delta_{k'}^{-1}=\delta_\Gamma = \delta_{\Gamma'}$ limit]\label{thm:fourDeltasLimit}
Assume \eqref{ass:new_on_g_2}. As $\delta_\Omega = \delta_k = \frac{1}{\delta_{k'}} = \delta_\Gamma = \delta_{\Gamma'}  \to 0$, the solution $(u_k, w_k, z_k)$ of \eqref{eq:newModel} converges to a weak solution $(u, w, z) \in L^2_{H^1(\Omega)} \times (L^2_{L^2(\Gamma)})^2$  of the problem \eqref{eq:systemFourDeltasLimit}.
\end{theorem}
We can show that in fact the masses of $w$ and $z$ are conserved:
\begin{lem}
We have
\[\intGt w(t) = \intGz w_0 \quad\text{and}\quad \intGt z(t) = \intGz z_0.\]
\end{lem}
\begin{proof}
Taking $\eta = \xi$ where $\xi \in C_c^\infty(0,T)$ as a test function in the weak form for $w$:
\[\int_0^T \xi'(t) \intGt w(t) =0 \qquad \forall \xi \in C_c^\infty(0,T),\]
hence by the du Bois-Reymond lemma we obtain
\[\intGt w(t) = c \quad\text{for a.e. $t \in (0,T)$}\]
where $c \geq 0$ is a constant. Let us now determine the value of $c$. If we select $\eta = \xi \in C_c^\infty([0,T))$ in the weak form for $w$, then we have
\[-\int_{\Gamma_0}w_0\xi(0)=c\int_0^T \xi'(t) = -c\xi(0)\]
which finally yields
$c=\int_{\Gamma_0}w_0.$ 
The argument for $z$ is similar.
\end{proof}
We can also prove that $u$ is in fact the zero function under a reasonable assumption on the initial condition for $w$, and for a large class of reaction terms $g$.
\begin{lem}
If $g$ is such that $g(u,w)=0$ implies $uw=0$, and if
\[\intGz w_0 >0,\]
then 
$u \equiv 0$.
\end{lem}
\begin{proof}
Since $\grad u(t) \equiv 0$, $u(t)$ is constant in space for a.e. fixed $t$, so we can write the complementarity condition as
\[\int_0^T u(t)\intGt w(t) = 0.\]
This implies from above that if $\intGz w_0 >0$, then $\int_0^T u(t)=0$, which given the non-negativity of $u$ yields $u \equiv 0$.
\end{proof}
Since the assumption of the previous lemma is very reasonable and it leads us to the trivial solution $u \equiv 0$, we shall also consider a modification of \eqref{eq:newModel} where we replace the Neumann boundary condition \eqref{eq:newModelNeumannBC} by a Dirichlet one in the next subsection.

\subsubsection{The Dirichlet case of the limit}\label{sec:dirichlet_case_three_limit}
We consider here the problem \eqref{eq:newModel} with a Dirichlet boundary condition on the outer boundary instead of the Neumann condition:
\begin{subequations}\label{eq:newModelDirichlet}
\begin{align}
 \delta_\Omega(\mdd{u}+u\nabla\cdot\mathbf V_p+\nabla \cdot\left(\mathbf{J}_\Omega u\right)) - \Delta u  &=0&\text{ in }\Omega(t),\\
\nabla u \cdot \nu - \delta_\Omega ju&=  \frac{1}{\delta_{k'}} z-\frac{1}{\delta_k}g(u,w) &\text{ on }\Gamma(t),\\
u &= u_D &\text{on $\partial D(t),$}\\
 \mdd{w} + w\nabla_\Gamma\cdot\mathbf V_p-\delta_\Gamma\slap w +\nabla_\Gamma \cdot\left(\mathbf{J}_\Gamma w\right) &=   \frac{1}{\delta_{k'}} z-\frac{1}{\delta_k}g(u,w)&\text{ on }\Gamma(t),\\
 \mdd{z} + z\nabla_\Gamma\cdot\mathbf V_p-\delta_{\Gamma'}\slap z+\nabla_\Gamma \cdot\left(\mathbf{J}_\Gamma z\right) &=   \frac{1}{\delta_k}g(u,w)-\frac{1}{\delta_{k'}} z&\text{ on }\Gamma(t),\\
  (u(0),w(0),z(0)) &= (u_0,w_0,z_0),
\end{align}
\end{subequations}
where $u_D \in C^0_{H^{1\slash 2}(\partial D)} \cap L^\infty_{L^\infty(\partial D)}$, $u_D \geq 0$ is a given non-negative function.  Define the space
\[H^1_{d}(\Omega(t)) := \{ v \in H^1(\Omega(t)) : v|_{\partial D(t)} = u_D\}.\]
Like before, we will eschew the details of well posedness for this system and simply assume the existence of a unique non-negative solution 
\[u \in L^2_{H^1_{d}(\Omega)}\cap H^1_{H^1_{d}(\Omega)^*} \quad\text{and}\quad w,z \in L^2_{H^1(\Gamma)}\cap H^1_{L^2(\Gamma)}\]
with
\[g(u,w) \in L^2_{L^2(\Gamma)}.\] 
Note that the weak formulation for \eqref{eq:newModelDirichlet} is obtained by testing the $u$ equation with a test function from the space $L^2_{H^1_{e}(\Omega)}$ and asking for the trace of $u(t)$ on $\partial D(t)$ to coincide with $u_D(t)$.

Now, we have to change the space $\mathcal{V}$ to take care of the Dirichlet data and this can be done with first defining
\[H^1_e(\Omega(t)) := \{ v \in H^1(\Omega(t)) : v|_{\partial D(t)} = 0\}\]
and then
\[\mathcal{V}_e := \left\{\eta \in L^2_{H^1_e(\Omega)} \cap H^1_{L^2(\Omega)} : \eta|_{\Gamma} \in L^2_{H^1(\Gamma)} \cap H^1_{L^2(\Gamma)}\right\}.\]

\begin{defn}
    We say that $(u, w, z) \in L^2_{H^1(\Omega)} \times (L^2_{L^2(\Gamma)})^2$ is a \emph{weak solution} of the problem 
    \begin{subequations}\label{eq:systemFourDeltasLimitD}
\begin{align}
 \Delta u &= 0 &\text{ in } \Omega(t), \\
u &= u_D &\text{ on } \partial D(t),\\
\mdd w + w\dVp + \sgrad \cdot (\JG w)  
&=  \grad u \cdot \nu &\text{ on } \Gamma(t), \label{eq:forwD}\\
\mdd z + z \dVp + \sgrad \cdot (\JG z) 
&= -\grad u \cdot \nu &\text{ on } \Gamma(t), \\
g(u,w) &=0 &\text{ on } \Gamma(t),\label{eq:gIsZeroFourD}\\
(w(0),z(0)) &= (w_0,z_0),
\end{align}
\end{subequations}
if for every $\eta \in \mathcal{V}_e$ with $\eta(T)=0$,
\begin{subequations}
\begin{align}
\int_0^T\int_{\Omega(t)} \nabla u \cdot \nabla \eta  + \int_{\Gamma_0}w_0\eta(0)  + \int_0^T \intGt w \mdd \eta + w\JG\cdot \sgrad \eta  &= 0,\\
\int_0^T\int_{\Omega(t)} \nabla u \cdot \nabla \eta  -\int_{\Gamma_0}z_0\eta(0)  - \int_0^T \intGt z \mdd \eta  + z\JG\cdot \sgrad \eta  &= 0,\\
u|_{\partial D} &= u_D \text { in } L^2_{H^{1\slash 2}(\partial D)},\\
g(u,w) &= 0  \text{ in } L^1_{L^1(\Gamma)}.
\end{align}
\end{subequations}
\end{defn}

In \cref{sec:thirdLimitDirichlet}, we show the following.
\begin{theorem}[The $\delta_\Omega = \delta_k= \delta_{k'}^{-1}=\delta_\Gamma = \delta_{\Gamma'}$ limit in the Dirichlet case]\label{thm:fourDeltasLimitDirichlet}
Assume 
\begin{align}
\nonumber \text{$u_n \weaklyto u$ in $L^2_{H^{1\slash 2}(\Gamma)}, \quad u_n \weakstar u$ in $L^\infty_{L^\infty(\Gamma)}$,  $\quad v_n \to v$ in $L^2_{H^{-1\slash 2}(\Gamma)}$, \quad $g(u_n, v_n) \to 0$ in $L^1_{L^1(\Gamma)}$}\\
\text{$\implies g(u,v) = 0$ in $L^1_{L^1(\Gamma)}$}\label{ass:new_on_g_3}.
\end{align}
As $\delta_\Omega = \delta_k = \frac{1}{\delta_{k'}} = \delta_\Gamma = \delta_{\Gamma'}  \to 0$, the solution $(u_k, w_k, z_k)$ of \eqref{eq:newModelDirichlet} converges to a weak solution $(u, w, z) \in L^2_{H^1(\Omega)} \times (L^2_{L^2(\Gamma)})^2$  of the problem \eqref{eq:systemFourDeltasLimitD}.
\end{theorem}
\begin{prop}\label{lem:verification_of_g3}
The choice $g(u,w)=uw$ satisfies assumption \eqref{ass:new_on_g_3}.
\end{prop}
The proof is in \cref{sec:app_on_g}.  
\subsection{Reformulation as free boundary problems on evolving surfaces.} \label{sec:deltakFBP}
We shall now formulate the limiting problems we have derived as free boundary problems. We make the following standing assumption (which clearly holds for the two exemplars \eqref{eq:g_usual} and \eqref{eq:g_mm}) for this section.
\begin{ass}
    Suppose $g$ has the property that $g(u,w) = 0$ if and only if $uw=0$.    
\end{ass}
Define the maximal monotone graph $\beta \colon \mathbb{R} \rightrightarrows \mathbb{R}$ by
\[\beta(r) := \begin{cases}
\emptyset &: r< 0\\
(-\infty, 0] &: r=0\\
0 &: r > 0.
\end{cases}\]
 If we define $v:=-w$ we can recast the complementarity condition $g(u,w)=0$ on $\Gamma(t)$  as 
 \[v \in \beta(u) \text{ on $\Gamma(t)$}.\]
\begin{remark}
Alternatively, defining
\[\alpha(r) := \begin{cases}
0 &: r < 0\\
[0, \infty) &: r=0\\
\emptyset &: r>0,
\end{cases}\]
we can rewrite $g(u,w)=0$ on $\Gamma(t)$ as
$u \in \alpha(v)$ on $\Gamma(t)$.
\end{remark}
Define a map $P\colon L^2_{H^{1\slash 2}(\Gamma)} \to L^2_{H^1(\Omega)}$ as the solution map $u \mapsto U$ of
\begin{equation*}
\begin{aligned}
\delta_\Omega(\mdd U + U\dVpO + \grad \cdot (\JO U))- \Delta U  &=0 &\text{ in } \Omega(t), \\
\grad U \cdot \nu &=0 &\text{on $\partial D(t),$}\\
U &= u &\text{on $\Gamma(t),$}\\
U(0) &= u_0 &\text{on $\Omega_0$},
\end{aligned}
\end{equation*}
and define the associated Dirichlet-to-Neumann map $A\colon L^2_{H^{1\slash 2}(\Gamma)} \to L^2_{H^{-1\slash 2}(\Gamma)}$ by 
\[A(u) := \grad P(u) \cdot \nu\]
as the normal derivative on $\Gamma$. Finally, we set $v_0 :=-u_0$ as the initial data.  

\paragraph{The $\delta_k$ limit.} We can rewrite the problem \eqref{eq:systemdeltakLimit} as: find $(u,v) \in L^2_{H^{1\slash 2}(\Gamma)} \times \mathbb{W}(H^1(\Gamma), H^1(\Gamma)^*)$ with $u \geq 0$, $v \leq 0$, such that
\begin{equation*}
\begin{aligned}
\mdd v - \delta_\Gamma \slap v + v\dVp + \sgrad \cdot (\JG v) + A(u) &= \delta_\Omega ju &\text{ on } \Gamma(t),\\
v &\in \beta(u) &\text{ on } \Gamma(t),\\
v(0) &= v_0.
\end{aligned}
\end{equation*}


\paragraph{The $\delta_k=\delta_\Gamma = \delta_{\Gamma'}$ limit.}Regarding the second limiting problem \eqref{eq:systemthreeDeltasLimit}, we can write it as: find $(u,v) \in L^2_{H^{1\slash 2}(\Gamma)} \times \mathbb{W}(H^1(\Gamma), H^1(\Gamma)^*)$ with $u \geq 0$, $v \leq 0$, such that
\begin{equation*} 
\begin{aligned}
\mdd v  + v\dVp + \sgrad \cdot (\JG v) + A(u) &= \delta_\Omega ju &\text{ on } \Gamma(t),\\
v &\in \beta(u) &\text{ on } \Gamma(t),\\
v(0) &= v_0.
\end{aligned}
\end{equation*}

\paragraph{The $\delta_\Omega = \delta_k= \delta_{k'}^{-1}=\delta_\Gamma = \delta_{\Gamma'}$ limit in the Dirichlet case.} As done in \cite{ERV}, we shall separate the two Dirichlet conditions for $u$ (coming from $u_D$ and $u|_{\Gamma(t)}$) into two separate equations. To do so, let us first define a map $E\colon L^2_{H^{1\slash 2}(\Gamma)} \to L^2_{H^1(\Omega)}$ as the solution map $u \mapsto U$ of
\begin{equation*}
\begin{aligned}
\Delta U &=0 &\text{ in } \Omega(t), \\
U  &= 0  &\text{on $\partial D(t),$}\\
U &= u &\text{on $\Gamma(t),$}
\end{aligned}
\end{equation*}
and define the associated Dirichlet-to-Neumann map $A_D\colon L^2_{H^{1\slash 2}(\Gamma)} \to L^2_{H^{-1\slash 2}(\Gamma)}$ by 
\[A_D(u) := \grad E(u) \cdot \nu\]
as the normal derivative on $\Gamma$.  We refer to \cite[\S 5.4.2]{AESIFB} for more rigorous details about these maps in a similar setting.

Define also the element $h$ by
\begin{equation*}
\begin{aligned}
\Delta h &=0 &\text{ in } \Omega(t), \\
h  &= u_D  &\text{on $\partial D(t),$}\\
h &= 0 &\text{on $\Gamma(t).$}
\end{aligned}
\end{equation*}
We can rewrite the problem \eqref{eq:systemFourDeltasLimitD} as: find $(u,v) \in L^2_{H^{1\slash 2}(\Gamma)} \times \mathbb{W}(L^2(\Gamma), H^{-1\slash 2}(\Gamma))$ with $u \geq 0$, $v \leq 0$, such that
\begin{equation}\label{eqn:surface_stefan}
\begin{aligned}
\mdd v  + v\dVp + \sgrad \cdot (\JG v) + A_D(u) + \grad h \cdot \nu &= 0 &\text{ on } \Gamma(t),\\
v &\in \beta(u) &\text{ on } \Gamma(t),\\
v(0) &= v_0.
\end{aligned}
\end{equation}
This is an example of a Stefan-type free boundary problem on an evolving surface and is related to \cite{AEStefan}, where the existence and uniqueness of a Stefan problem on an evolving surface was proved.

\subsection{The non-evolving setting}\label{sec:nonMoving}
For convenience, let us write down the limiting problems we get in case there is no domain evolution.

\begin{theorem}[The $\delta_k$ limit]\label{thm:deltakLimitS}
Assume 
\begin{align}
\nonumber    \text{$u_n \weaklyto u$ in $L^2(0,T;H^{1\slash 2}(\Gamma))$, $\quad v_n \to v$ in $L^2(0,T;L^2(\Gamma)), \quad g(u_n, v_n) \to 0$ in $L^1(0,T;L^1(\Gamma))$}\\
     \implies g(u,v)=0,\label{ass:new_on_g_1_s}
\end{align}
and let $d \leq 3$. As $\delta_k\to 0$, the solution $(u_k, w_k, z_k)$ of \eqref{eq:modelStationary} converges to a weak solution $(u, w, z) \in L^2(0,T;H^1(\Omega)) \times (L^2(0,T;H^1(\Gamma)))^2$ of the problem
\begin{subequations}\label{eq:systemdeltakLimitS}
\begin{align}
\delta_\Omega \dot u-   \Delta u   &=0 &\text{ in } \Omega,  \\
\grad u \cdot \nu &=0 &\text{on $\partial D $},\\
\dot w - \delta_\Gamma \slap w   &= \grad u \cdot \nu   &\text{ on } \Gamma,  \label{eq:deltakLimitSystem_wS}\\
\dot z - \delta_{\Gamma'}\slap z   &=- \grad u \cdot \nu  &\text{ on } \Gamma,  \label{eq:deltakLimitSystem_zS}\\
g(u,w) &=0 &\text{ on } \Gamma,  \label{eq:gIsZeroS}\\
(u(0),w(0),z(0)) &= (u_0,w_0,z_0).
\end{align}
\end{subequations}
\end{theorem}
When $g(u,w) =0$ implies $uw=0$, the $(u,w)$ system above corresponds directly to the limiting system obtained in \cite[Theorem 5.3]{ERV}. 

\begin{theorem}[The $\delta_k=\delta_\Gamma = \delta_{\Gamma'}$ limit] 
Assume 
\begin{align}
\nonumber    \text{$u_n \to u$ in $L^2(0,T;L^2(\Gamma)), \quad u_n \weakstar u $ in $ L^\infty(0,T;L^\infty(\Gamma)),$ $\quad v_n \weakstar v$ in $L^\infty(0,T;L^\infty(\Gamma)),$}\\
\text{$g(u_n, v_n) \to 0$ in $L^1(0,T;L^1(\Gamma))  $} \text{ $\implies g(u,v) = 0$ in $L^1(0,T;L^1(\Gamma)) $}\label{ass:new_on_g_2_s}.
\end{align}
As $\delta_k = \delta_\Gamma = \delta_{\Gamma'} \to 0$, the solution $(u_k, w_k, z_k)$ of \eqref{eq:modelStationary} converges to a weak solution $(u, w, z) \in L^2(0,T;H^1(\Omega)) \times (L^2(0,T;L^2(\Gamma)))^2$ of the problem
\begin{subequations}\label{eq:systemthreeDeltasLimitS}
\begin{align}
\delta_\Omega \dot u  -  \Delta u   &=0 &\text{ in } \Omega,  \\
\grad u \cdot \nu &=0 &\text{on $\partial D,$}\\
\dot w   &=  \grad u \cdot \nu   &\text{ on } \Gamma, \\
\dot z  &=- \grad u \cdot \nu  &\text{ on } \Gamma, \\
g(u,w) &=0 &\text{ on } \Gamma, \label{eq:gIsZeroThreeS}\\
(u(0),w(0),z(0)) &= (u_0,w_0,z_0).
\end{align}
\end{subequations}
\end{theorem}
When $g(u,w)=0$ implies $uw=0$, uniqueness for this system follows from uniqueness for the $(u,w)$ system (as $z$ can be treated as uncoupled), which was shown in \cite[Theorem 6.3]{ERV}. There, the authors proved uniqueness by essentially taking the difference of the weak formulations for two solutions and testing with an integral over time of the difference of the two solutions (note that such an approach would not work directly in the evolving space setting since the spatial points are time-dependent too).

\begin{theorem}[The $\delta_\Omega = \delta_k= \delta_{k'}^{-1}=\delta_\Gamma = \delta_{\Gamma'} $ limit]\label{thm:fourDeltasLimitS}
Assume \eqref{ass:new_on_g_2_s}. As $\delta_\Omega = \delta_k = \frac{1}{\delta_{k'}} = \delta_\Gamma = \delta_{\Gamma'}  \to 0$, the solution $(u_k, w_k, z_k)$ of \eqref{eq:modelStationary} converges to a weak solution $(u, w, z) \in L^2(0,T;H^1(\Omega)) \times (L^2(0,T;L^2(\Gamma)))^2$  of the problem 
\begin{subequations}\label{eq:systemFourDeltasLimitS}
\begin{align}
\grad u &=0 &\text{ in } \Omega,  \label{eq:gradUIsZeroFourS}\\
\dot w   &=  0 &\text{ on } \Gamma,  \\
\dot z   &= 0 &\text{ on } \Gamma, \\
g(u,w) &=0 &\text{ on } \Gamma,\label{eq:gIsZeroFourS}\\
(w(0),z(0)) &= (w_0,z_0).
\end{align}
\end{subequations}
\end{theorem}

\begin{theorem}[The $\delta_\Omega = \delta_k= \delta_{k'}^{-1}=\delta_\Gamma = \delta_{\Gamma'} $ limit in the Dirichlet case]
Assume that $u_D \in C^0([0,T];H^{1\slash 2}(\partial D)) \cap L^\infty(0,T;L^\infty(\partial D))$ and
\begin{align}
\nonumber \text{$u_n \weaklyto u$ in $L^2(0,T;H^{1\slash 2}(\Gamma)), \quad u_n \weakstar u$ in $L^\infty(0,T;L^\infty(\Gamma))$,  $\quad v_n \to v$ in $L^2(0,T;H^{-1\slash 2}(\Gamma))$,}\\
\text{$g(u_n, v_n) \to 0$ in $L^1(0,T;L^1(\Gamma)) \implies g(u,v) = 0$ in $L^1(0,T;L^1(\Gamma))$}\label{ass:new_on_g_3_s}.
\end{align}
As $\delta_\Omega = \delta_k = \frac{1}{\delta_{k'}} = \delta_\Gamma = \delta_{\Gamma'}  \to 0$, the solution $(u_k, w_k, z_k)$ of \eqref{eq:modelStationary} with \eqref{eq:stationaryModelBCD} replaced by
\[u =u_D \quad\text{on $\partial D$} \]
converges to a weak solution $(u, w, z) \in L^2(0,T;H^1(\Omega)) \times (L^2(0,T;L^2(\Gamma)))^2$  of the problem 
\begin{subequations}\label{eq:systemFourDeltasLimitDS}
\begin{align}
 \Delta u &= 0 &\text{ in } \Omega, \\
u &= u_D &\text{ on } \partial D,\\
\dot w &=  \grad u \cdot \nu &\text{ on } \Gamma, \label{eq:forwDS}\\
\dot z &= -\grad u \cdot \nu &\text{ on } \Gamma, \\
g(u,w) &=0 &\text{ on } \Gamma,\label{eq:gIsZeroFourDS}\\
(w(0),z(0)) &= (w_0,z_0).
\end{align}
\end{subequations}
\end{theorem}

\subsection{Biological implications}\label{sec:bio_implications}
For the representative forms of $g$ given in \eqref{eq:g_usual} and \eqref{eq:g_mm}, the reason the fast reaction nature of the limits ($\delk\to0$) leads to interesting free boundary problems is  because of the complementarity nature of the resulting limit
\begin{align*}
 {u}\ge 0, && w\ge 0, && {u} {w} =0 && \mbox{ on } \Gamma(t).
\end{align*}
In regions on $\Gamma(t)$ where $w>0$ the upshot is that the boundary condition for the ligand corresponds to a zero Dirichlet condition whilst in regions on $\Gamma(t)$ where $w=0$ the boundary condition for the ligand corresponds to a zero Neumann condition (total flux, i.e., diffusive plus advective is zero). These two boundary conditions are referred to in the biological literature as \textit{perfectly absorbing} and \textit{perfectly reflecting/monitoring} respectively and have been widely studied as relevant simplifications of receptor-ligand interactions; see e.g., \cite{endres2008accuracy} and references therein. This work, therefore, gives a rigorous justification for the consideration of these boundary conditions as singular limits of models for receptor-ligand interactions. 

The numerical results of \cref{sec:numerical} illustrate the significance of domain evolution, both in terms of generating spatial heterogeneity and effects that arise due to differences between the material velocity of the cell membrane ($\VG$) and the material velocity of the extracellular medium ($\VO$). Such aspects have received limited computational investigation \cite{macdonald2016computational,mackenzie2016local} however; symptotic limits of models in cell biology (which have been derived in a number of contexts) have been focused solely on the fixed domain setting. This work provides a framework for extending these analyses to the evolving domain setting thereby increasing their applicability to understanding biological problems.

Furthermore, this work suggests that models for receptor-ligand dynamics on evolving domains involving fast reaction kinetics can be derived using classical elements of free boundary methodology as components of the modelling.

\section{Preliminary results on PDEs on evolving spaces}\label{sec:preliminary_results}
To study the limiting behaviour of the system \eqref{eq:newModel}, we need a number of uniform estimates. In this section we derive various technical results and estimates for equations on evolving domains and surfaces that will be used later on for this purpose.
\subsection{Estimates on the heat equation on evolving surface}
\begin{prop}\label{prop:BHE}
Let $g \in L^2_{L^2(\Gamma)}$ and $b \in L^\infty_{L^\infty(\Gamma)}$ satisfy $m \leq  b \leq M$ a.e. for given constants $m \leq M$. The problem 
\begin{equation}\label{eq:bEquation}
\begin{aligned} 
    \mdd y (t) - b(t) \Delta_\Gamma y(t) + \JG \cdot \sgrad y(t)&= g(t) &&\text{on }\Gamma(t),\\
    y(0) &= 0 &&\text{on } \Gamma_0,
\end{aligned}
\end{equation}
has a unique solution $y \in \mathbb{W}(H^1(\Gamma), L^2(\Gamma))$ with $\slap y \in L^2_{L^2(\Gamma)}$  satisfying the estimates
\[ \norm{y}{L^\infty_{L^2(\Gamma)}} +  \norm{\mdd y}{L^2_{L^2(\Gamma)}}  \leq C(M, m^{-1})\norm{g}{L^2_{L^2(\Gamma)}}\]
and
\[ \norm{\sgrad y}{L^\infty_{L^2(\Gamma)}} + \sqrt{m}\norm{\Delta_\Gamma y}{L^2_{L^2(\Gamma)}}\leq C(m^{-1})\norm{g}{L^2_{L^2(\Gamma)}}.\]

\end{prop}
\begin{proof}
We approximate $b$ by $\{b_\rho\}$ where $\phi_{-(\cdot)}b_\rho \in C^2([0,T]\times\Gamma_0)$ satisfies $m \leq  b_\rho \leq M$. Consider the problem
\begin{equation}
\begin{aligned}\label{eq:pde_sec1ApproxLemma}
    \mdd y_\rho (t) - b_\rho(t) \Delta_\Gamma y_\rho(t) + \JG \cdot \sgrad  y_\rho(t)&= g(t) &&\text{on }\Gamma(t),\\
    y_\rho(0) &= 0 &&\text{on } \Gamma_0.
\end{aligned}
\end{equation}
Arguing similarly as in \cite[Lemma 2.13]{AEStefan} and with the aid of the existence and regularity result of \cite[Theorem 3.13]{ AESAbstract}, we obtain a unique solution $y_\rho \in \mathbb{W}(H^1(\Gamma), L^2(\Gamma))$ with $\slap y_\rho \in L^2_{L^2}$ of \eqref{eq:pde_sec1ApproxLemma}.   
It follows that the equation holds pointwise and we also have the weak formulation 
\begin{align}
    \int_0^T \int_{\Gamma(t)} \mdd y_\rho(t) \eta(t) -  \int_0^T \int_{\Gamma(t)} b_\rho(t) \Delta_\Gamma y_\rho(t) \eta(t)  +   \int_0^T \int_{\Gamma(t)}\JG \cdot \sgrad  y_\rho \eta(t)= \int_0^T \int_{\Gamma(t)} g(t) \eta(t) \qquad \forall \eta\in L^2_{ H^1(\Gamma)}\label{eq:weakFormForVarthetaRho}
\end{align} 
Let us obtain some estimates:
 \begin{enumerate}[label=({\roman*}), wide, labelwidth=!, labelindent=0pt]
\item Multiply the equation with $\Delta_\Gamma y_\rho(t)$, integrate over time  and manipulate (see \cref{lem:IBPIdentityLaplacian}) to give
\begin{align*}
-\frac 12\int_{\Gamma(t)}|\sgrad     y_\rho(s)|^2  &+ \frac 12\int_0^t\int_{\Gamma(s)} \sgrad  y_\rho(s)^T \mathbf{H}(s)\sgrad  y_\rho(s)   - \int_0^t\int_{\Gamma(s)}b_\rho(s)|\Delta_\Gamma y_\rho(s)|^2\\
&+ \int_0^t\int_{\Gamma(s)}  \JG\cdot \sgrad y_\rho(s) \slap y_\rho(s) =\int_0^t\int_{\Gamma(s)}g(s)\Delta_\Gamma y_\rho(s)
\end{align*}
which we further manipulate to get
\begin{align*}
\int_{\Gamma(t)}|\sgrad     y_\rho(s)|^2  +2m \int_0^t\int_{\Gamma(s)}|\Delta_\Gamma y_\rho(s)|^2 &\leq  (C_1 + \frac{\norm{\JG}{\infty}}{2\eps_2})\int_0^t\int_{\Gamma(s)} |\sgrad  y_\rho(s)|^2\\
&\quad + \int_0^t\int_{\Gamma(s)}\frac{1}{2\epsilon_1}|g(s)|^2 + (2\epsilon_1 + 2\epsilon_2)|\Delta_\Gamma y_\rho(s)|^2
\end{align*}
where we used Young's inequality with $\epsilon_1$ and $\epsilon_2$. Choosing $\epsilon_1 = \epsilon_2 = \frac 14 m$, 
\[\int_{\Gamma(t)}|\sgrad     y_\rho(s)|^2 + m \int_0^t\int_{\Gamma(s)}|\Delta_\Gamma y_\rho(s)|^2 \leq  (C_1 + \frac{2\norm{\JG}{\infty}}{m})\int_0^t\int_{\Gamma(s)} |\sgrad  y_\rho(s)|^2 + \frac 2m\int_0^t\int_{\Gamma(s)}|g(s)|^2,\]
and applying Gronwall's inequality gives
\[\int_{\Gamma(t)}|\sgrad     y_\rho(s)|^2  \leq  \frac{C_2(m^{-1})}{m}\int_0^t\int_{\Gamma(s)}|g(s)|^2.\]
Plugging this back above, we get a similar bound on the Laplacian and we have shown
%
\[\norm{\sgrad y_\rho}{L^\infty_{L^2(\Gamma)}}^2 + m\norm{\Delta_\Gamma y_\rho}{L^2_{L^2(\Gamma)}}^2 \leq C_3(m^{-1})\norm{g}{L^2_{L^2(\Gamma)}}^2.\]
\item By simply rearranging the equation, we have
\begin{align*}
\norm{\mdd y_\rho}{L^2_{L^2(\Gamma)}} &\leq \norm{g}{L^2_{L^2(\Gamma)}}  + M\norm{\slap y_\rho}{L^2_{L^2(\Gamma)}}  + \norm{\JG}{\infty}\norm{\sgrad y_\rho}{L^2_{L^2(\Gamma)}}\\
&\leq \left(1 + MC_4(m^{-1}) + C_5(m^{-1})\right)\norm{g}{L^2_{L^2(\Gamma)}}.
\end{align*}
\item Testing the equation with $y_\rho$ and integrating leads to
\begin{align*}
    \dfrac{1}{2}\dfrac{d}{dt}\int_{\Gamma(s)} y_\rho(s)^2 = \int_{\Gamma(s)} g(s)y_\rho(s) + \int_{\Gamma(s)} b_\rho(s) \Delta_\Gamma y_\rho(s) y_\rho(s) + \int_{\Gamma(s)} y_\rho(s)^2 \sgrad\cdot \mathbf{V}_p - \int_{\Gamma(t)}\JG\cdot \sgrad y_\rho y_\rho
\end{align*}
from where
\begin{align*}
    \int_{\Gamma(t)} y_\rho(t)^2     &\leq \int_0^t \int_{\Gamma(s)} g(s)^2 + C_1\int_0^t\int_{\Gamma(s)} y_\rho(t)^2 +  M^2 \int_0^t \int_{\Gamma(s)} |\Delta_\Gamma y_\rho(s)|^2 + C_2\int_0^t\int_{\Gamma(s)} |\sgrad y_\rho|^2\\
    &\leq \left(1 + \dfrac{C_2M^2}{m^2} + \frac{C_3}{m}\right)\int_0^t\int_{\Gamma(s)} g(s)^2 + C_1\int_0^t\int_{\Gamma(s)} y_\rho(s)^2,
\end{align*}
and Gronwall's inequality yields 
\[\norm{y_\rho}{L^\infty_{L^2(\Gamma)}}^2 = \int_{\Gamma(t)} y_\rho(t)^2  \leq C\left(1 + C_2(m^{-1})M^2 + C_3(m^{-1})\right)\int_0^t\int_{\Gamma(t)} g(s)^2.\]
\end{enumerate} 
Due to these estimates being uniform in $\rho$, we obtain a limit function $y \in L^\infty_{H^1(\Gamma)} \,\cap\,L^2_{H^1(\Gamma)}$ with $\slap y, \mdd y \in L^2_{L^2(\Gamma)}$ such that 
\begin{align*}
    y _\rho \overset{*}{\rightharpoonup} y  \,\, \text{ in } L^\infty_{H^1(\Gamma)}, \quad \Delta_\Gamma y _\rho \rightharpoonup \Delta_\Gamma y  \,\, \text{ in } L^2_{L^2(\Gamma)}, \quad \mdd y _\rho \rightharpoonup \mdd y  \,\, \text{ in } L^2_{L^2(\Gamma)},
\end{align*}
and by compactness we have the stronger convergence 
\begin{align*}
    y_\rho \to y  \,\, \text{ in } L^2_{L^2(\Gamma)}.
\end{align*}
It is then immediate to pass to the limit in the weak form \eqref{eq:weakFormForVarthetaRho}  (the Laplacian term can be handled by an argument involving the Dominated Convergence Theorem, making use of the uniform boundedness of $b_\rho$ and the pointwise a.e. convergence of $\phi_{-(\cdot)} b_\rho$ to $\phi_{-(\cdot)}b$) 
to conclude that the limit function $y$ satisfies for every $\eta\in L^2_{ H^1(\Gamma)}$,
\begin{align*}
    \int_0^T \int_{\Gamma(t)} \mdd y(t) \eta(t) -  \int_0^T \int_{\Gamma(t)} b(t) \Delta_\Gamma y(t) \eta(t) +  \int_0^T \int_{\Gamma(t)}\JG \cdot \sgrad  y \eta(t)= \int_0^T \int_{\Gamma(t)} g(t) \eta(t),
\end{align*}
i.e., it is a weak solution to \eqref{eq:bEquation} with $y(0)=0$. The proof for the initial condition follows as usual by making use of the stronger convergence $y_\rho \to y$ in $L^2_{L^2(\Gamma)}$.
\end{proof}

Consider the following parabolic inequality
\begin{equation}\label{eq:linearIVP}
\begin{aligned}
\mdd y  - D\Delta_\Gamma y + y\dVp + \sgrad \cdot (\JG y) &\leq g&&\text{on }\Gamma(t),\\
y(0) &= y_0&&\text{on }\Gamma_0,
\end{aligned}
\end{equation}
for a given constant $D > 0$, initial data $y_0 \in L^\infty(\Gamma_0)$ and source term $g\in L^\infty_{L^2(\Gamma)}$.  
By arguing directly as in \cite[Lemma 3.1]{AET}, we obtain the next result.
\begin{lem}\label{lem:deGiorgi}If $d \leq 3$, the weak solution of the equation \eqref{eq:linearIVP}
given $y_0 \in L^\infty(\Gamma_0)$  and $g \in L^\infty_{L^2(\Gamma)}$ satisfies
\[ 
\norm{y}{L^\infty_{L^\infty(\Gamma)}} \leq e^{(\norm{\sgrad \cdot \mathbf V_\Gamma}{\infty} + D)T}\left(\norm{y_0}{L^\infty(\Gamma_0)} + CD^{-1}\norm{g}{L^\infty_{L^{2}}}\right).
\]
\end{lem}

\subsection{$L^\infty$-estimate on the heat equation on an evolving domain with Robin boundary conditions}


Define the two quantities
\begin{equation*}\label{normQ}
\norm{u}{Q(\Gamma)} := \max_{t \in [0,T]} \norm{u(t)}{L^2(\Gamma(t))} + \norm{\sgrad u}{L^2_{L^2(\Gamma)}}
\end{equation*}
and
\[\norm{u}{Q(\Omega)} := \max_{t \in [0,T]}\norm{u(t)}{L^2(\Omega(t))} + \norm{\grad u}{L^2_{L^2(\Omega)}}.\]
Let us recall the following interpolation inequality (see \cite[Lemma B.2]{AET}).
\begin{lem}\label{lem:interpolatedSobolev}For $r_* \in [2,\infty]$ and $q_* \in [2, 2d\slash (d-1)]$ satisfying
\[\frac{1}{r_*} + \frac{d}{2q_*} = \frac{d+1}{4},\] we have
\[\norm{u}{L^{r_*}_{L^{q_*}(\Gamma)}} \leq \sqrt{C_I}\norm{u}{Q(\Omega)}.\]
\end{lem}
The proof of the next result is very similar to what is presented in \cite[\S 3.3]{AET} (which itself was based on \cite[Proposition 3.1]{Nittka}), but we give it here adapted to a slightly more general equation. Note that a similar result is given in \cite{CET} but there the geometric set up is different and here we need to keep track of how the bound depends on the coefficients and data of the problem. 
\begin{lem}\label{lem:L-infinity-Neumann}
Let $y \in L^\infty_{L^\infty(\Gamma)}$,  $a_0\in L^\infty(\Omega_0)$ and let $a$ be the nonnegative solution to
		\begin{equation*}
			\begin{aligned}
				\mdd{a} + a\grad \cdot\Vp - D\Delta a + \grad\cdot(\JO a) &= 0 &&\text{ in } \Omega(t),\\
				D\grad a \cdot \nu - aj &= y &&\text{ on } \Gamma(t),\\
                \grad a \cdot \nu &= 0 &&\text{ on } \partial D(t),\\
				a(0) &= a_0 && \text{ in } \Omega_0.
			\end{aligned}
		\end{equation*}
Then $a \in L^\infty_{L^\infty(\Omega)}$ and 
		\begin{equation*}
			a(t) \leq C_1e^{ \norm{\grad \cdot \VO}{\infty}T +  D^{-1}C_2\left(\frac 12\norm{y}{L^\infty_{L^\infty(\Gamma)}}  + 2\norm{j}{\infty}\right)^2T}\max\left(1, \norm{a_0}{\infty}, \frac 12 C_3 \norm{y}{L^\infty_{L^\infty(\Gamma)}}\right)(\min(1,D)^{-\frac{1}{2\kappa} - \frac 12} + C_4)
 		\end{equation*}
 where $C_1, C_2, C_3$ and $C_4$ are constants independent of all relevant parameters.
	\end{lem}
	\begin{proof}
With the transformation $A:=ae^{-\lambda t}$ for a $\lambda$ to be fixed later, we have
\begin{align*}
\mdd A + A (\grad \cdot \Vp + \lambda) - D \Delta A + \grad \cdot (\JO A) &= 0 &&\text{ in } \Omega(t),\\
D\grad A \cdot \nu - Aj &= Y &&\text{ on } \Gamma(t),\\
\grad A \cdot \nu &= 0&&\text{ on } \partial D(t),\\
A(0) &= a_0&&\text{ in } \Omega_0,
\end{align*}
where $Y(t) := y(t)e^{-\lambda t}$. Testing the equation for $A$ with $A_k := (A-  k)^+$ for a constant $k$, 
using the boundary condition as well as \eqref{eq:bulkPosId}, we obtain
\begin{align*}
&\frac 12 \frac{d}{dt}\intOt A_k^2 - \frac 12 \intOt A_k^2 \grad \cdot \VO + \intOt AA_k (\grad \cdot \VO  + \lambda) + D |\grad A_k|^2 + \frac 12 \intGt A_k^2 j\\
&\quad = \intGt (Y+jA)A_k.
\end{align*}		
Define the set
\[B_k(s) := \{ x \in \Gamma(s) : A_k(s) \geq 0\}.\]
Taking $k \geq 1$, using $AA_k = A_k^2 + kA_k \leq \frac 32A_k^2  + \frac 12k^2$ and $A_k = \frac 1k kA_k \leq \frac 1k(\frac{k^2}{2} + \frac{A_k^2}{2})$ we get that the difference of the two integrals over $\Gamma(t)$ can be written as
\begin{align*}
    \intGt (Y+jA)A_k - \frac 12 A_k^2j 
    &\leq  \int_{B_k(t)}\norm{Y}{L^\infty_{L^\infty(\Gamma)}} \frac 1k\frac{k^2}{2} + \norm{Y}{L^\infty_{L^\infty(\Gamma)}} \frac 1k\frac{A_k^2}{2} + \norm{j}{\infty}\frac 32 A_k^2 + \norm{j}{\infty}\frac 12 k^2 + \frac 12 \norm{j}{\infty}A_k^2\\
    &\leq  \int_{B_k(t)}\frac{\frac 12\norm{Y}{L^\infty_{L^\infty(\Gamma)}}}{k} k^2 + \left(\frac 12\norm{Y}{L^\infty_{L^\infty(\Gamma)}}  + 2\norm{j}{\infty}\right) A_k^2 + \frac 12\norm{j}{\infty} k^2\tag{using $k \geq 1$ on the second term above}.
\end{align*}
Now we apply the interpolated trace inequality to bound the  second term above as
\begin{align*}
    \intGt \left(\frac 12\norm{Y}{L^\infty_{L^\infty(\Gamma)}}  + 2\norm{j}{\infty}\right) A_k^2 \leq \frac D2 \intOt |\grad A_k|^2 + D^{-1}K_1\left(\frac 12\norm{Y}{L^\infty_{L^\infty(\Gamma)}}  + 2\norm{j}{\infty}\right)^2 \intOt A_k^2
\end{align*}
where $K_1$ is a constant independent of all relevant quantities.
Picking
\[ \lambda:=   \norm{\grad \cdot \VO}{\infty} +  D^{-1}K_1\left(\frac 12\norm{y}{L^\infty_{L^\infty(\Gamma)}}  + 2\norm{j}{\infty}\right)^2,  \]
using $AA_k = A_k^2 + kA_k   \geq A_k^2$, supposing that
$k \geq \norm{a_0}{L^\infty(\Omega_0)}$ 
and plugging the last two inequalities back in the equation above, we get 
with $m := \min(1, D)$,
\begin{align}
\frac m2 \norm{A_k}{Q(\Omega)}^2  &\leq  \left(\frac{\frac 12\norm{Y}{L^\infty_{L^\infty(\Gamma)}}}{k} + \frac 12\norm{j}{\infty}\right)k^2 \int_0^t\int_{B_k(s)}. \label{eq:prelim2}
\end{align}	
Now, take exponents $r$ and $q$ such that
\begin{equation*} 
\frac{1}{r} + \frac{d}{2q} < \frac{1}{2}
\end{equation*}
and define $\kappa$, $r_*$ and $q_*$ by
\[\frac{1}{r} + \frac{d}{2q} = \frac{1}{2}-\frac{\kappa (d+1)}{2}, \qquad  r_* = \frac{2(1+\kappa)r}{r-1}, \qquad \text{and} \qquad q_* = \frac{2(1+\kappa)q}{q-1}\]
(note that $\kappa > 0$). By H\"older's inequality,  $\int_0^T\int_{B_k(s)} \leq C_1\norm{\chi_{B_k}}{L^{r_*}_{L^{q_*}(\Gamma)}}^{2(1+\kappa)}$ where we set $C_1:=|\Gamma|^{1\slash q}T^{1\slash r}$. Hence \eqref{eq:prelim2} becomes
\begin{align*}
\frac m2 \norm{A_k}{Q(\Omega)}^2  &\leq  \left(\frac{\frac 12 C_1\norm{Y}{L^\infty_{L^\infty(\Gamma)}}}{k} + \frac 12C_1\norm{j}{\infty}\right)k^2 \norm{\chi_{B_k}}{L^{r_*}_{L^{q_*}(\Gamma)}}^{2(1+\kappa)}.
\end{align*}	
%
Taking further
\begin{equation}\label{eq:largenessOnk}
k \geq \max\left(1, \norm{a_0}{\infty}, \frac 12C_1\norm{Y}{L^\infty_{L^\infty(\Gamma)}}\right)
\end{equation}
and setting $C_2 :=  \left( 1 +\frac{C_1}{2}\norm{j}{\infty}\right)$, we deduce 
\begin{align}
\frac{m}{2}\norm{A_k}{Q(\Omega)}^2 &\leq C_2k^2 \left(\int_0^T |B_k|^{r_*\slash q_*}\right)^{2(1+\kappa)\slash r_{*}}\label{eq:A3}.
\end{align}
%
Define for a sufficiently large $N$ the sequences $\{k_n\}$ and $\{z_n\}$ via
\[k_n := (2-2^{-n})N \quad\text{and}\quad z_n := \left(\int_0^T |B_{k_n}(t)|^{r_*\slash q_*}\right)^{2\slash r_*}.\]
Now, we have from the above
\begin{align}
\norm{A_{k_n}}{Q(\Omega)}^2 &\leq  2C_2m^{-1}k_n^2 z_n^{1+\kappa}.\label{eq:pre3}
\end{align}
On $\Gamma(t)$, $|A_{k_n}|^2 = |(A-k_n)^+|^2 \geq (k_{n+1}-k_n)^2\chi_{B_{k_{n+1}}}$, which implies that
\begin{align*}
2^{-2(n+1)}N^2z_{n+1} 
&\leq \left(\int_0^T \left(\int_{\Gamma(t)}|(A(t)-k_n)^+|^{q_*}\right)^{r_*\slash q_*}\right)^{2\slash r_*} \leq C_I\norm{(A-k_n)^+}{Q(\Omega)}^2
\end{align*}
with the last inequality by the interpolation inequality of \cref{lem:interpolatedSobolev}. 
Therefore, using \eqref{eq:pre3} and  $k_n^2 \leq N^2(2^2 + 2^{-2n})$,
\[z_{n+1} \leq 2^6m^{-1}C_2C_I4^nz_n^{1+\kappa}.\]
Now if we take $\hat k$ such that it satisfies \eqref{eq:largenessOnk}, for $N > \hat k$, we calculate
\begin{align*}
(N-\hat k)^2z_0 
&= \left(\int_0^T \left(\int_{\Gamma(t)}(N-\hat k)^{q_*}\chi_{B_{k_0}}(t)\right)^{r_*\slash q_*}\right)^{2\slash r_*}\\
&\leq \left(\int_0^T \left(\int_{\Gamma(t)}|(A(t)-\hat k)^+|^{q_*}\right)^{r_*\slash q_*}\right)^{2\slash r_*}\tag{since $|(A-\hat k)^+|^2 \geq (N-\hat k)^2\chi_{B_{k_0}}$ because $k_0=N$}\\
&\leq C_I\norm{(A-\hat k)^+}{Q(\Omega)}^2\\
&\leq 2C_2m^{-1}C_I\hat{k}^2 C_3\tag{using \eqref{eq:A3}}
\end{align*}
where we defined $C_3 := |\Gamma|^{2(1+\kappa)\slash q_*}T^{2(1+\kappa)\slash r_*}$.  
Picking $N = \hat k(\sqrt{C_3}2^{\frac 12  + \frac 3\kappa + \frac{1}{\kappa^2}}(m^{-1}C_2C_I)^{\frac{1}{2\kappa} + \frac 12} + 1)$, we have
\begin{align*}
z_0 &\leq 
 (2^6m^{-1}C_2C_I)^{-1\slash \kappa}4^{-1\slash \kappa^2}.
\end{align*}
By \cite[II, Lemma 5.6]{Lady}, we get $z_n \to 0$ as $n \to \infty.$ Since $k_n \to 2N$, we obtain from \eqref{eq:pre3} that $A(t) \leq 2N$ almost everywhere on $\Omega(t)$. Putting everything together, we find the stated bound.
		\end{proof}

\section{Uniform bounds}\label{sec:uniform_bounds}
In this section, we aim to get a variety of bounds on the solutions of \eqref{eq:newModel} and related quantities that we can use later on. Note that all of these bounds do not \textit{explicitly} depend on $\delta_k$ but do depend on $L^p$ norms of $z$ and $w$.

\begin{lem}\label{lem:l1boundOng}
We have
\begin{align}\label{eq:g_l1_l1}
\dfrac{1}{\delta_k} \int_0^T \intGt g(u,w) &\leq \dfrac{1}{\delta_{k'}} \|z\|_{L^1_{L^1(\Gamma)}}  + \delta_\Omega \|u_0\|_{L^2(\Omega)}.
\end{align}
\end{lem}

\begin{proof}
Utilising the transport theorem and the equation for $u$, we derive
\begin{align*}
\frac{d}{dt}\int_{\Omega(t)} u  &= \int_{\Omega(t)} \mdd u + u\dVp \\
&= - \int_{\Omega(t)}\grad \cdot (\JO u) +  \frac{1}{\delta_\Omega}\int_{\Omega(t)} \Delta u  \\
&= -\int_{\Gamma(t)} ju + \frac{1}{\delta_\Omega}\int_{\Gamma(t)} \grad u \cdot \nu  \tag{using the divergence theorem}\\
&= 
\frac{1}{\delta_\Omega \delta_{k'}}\int_{\Gamma(t)}z - \frac{1}{\delta_\Omega\delta_k}\int_{\Gamma(t)}g(u,w),
\end{align*}
thus
\begin{align*}
\frac{1}{\delta_\Omega\delta_k}\int_0^T\int_{\Gamma(t)}g(u,w) 
&=  \frac{1}{\delta_\Omega \delta_{k'}}\int_0^T\int_{\Gamma(t)}z   + \int_{\Omega_0}u_0 - \int_{\Omega(T)} u(T),
\end{align*}
which implies the result.
\end{proof}

\begin{lem}[Energy estimate on $w$]\label{lem:w_energy}
We have
\begin{align*}
    \norm{w}{L^\infty_{L^2(\Gamma)}}^2 &\leq 
e^{T(\frac{1}{\delta_{k'}} + \linfdVG)}\left(\frac{1}{\delta_{k'}} \norm{z}{L^2_{L^2(\Gamma)}}^2 + \norm{w_0}{L^2(\Gamma_0)}^2\right),\\
2\delta_\Gamma\norm{\sgrad w}{L^2_{L^2(\Gamma)}}^2 + \frac{2}{\delta_k}\int_0^T\int_{\Gamma(t)} g(u,w)w &\leq 
\frac{1}{\delta_{k'}}(\norm{w}{L^2_{L^2(\Gamma)}}^2 +\norm{z}{L^2_{L^2(\Gamma)}}^2) + \linfdVG \norm{w}{L^2_{L^2(\Gamma)}}^2\\
&\quad  + \norm{w_0}{L^2(\Gamma_0)}^2.
\end{align*}
\end{lem}
\begin{proof}
Test the equation for $w$ by $w$ (and use the integration by parts identity \eqref{eq:veryUsefulIdentity}) to obtain
\begin{align*}
    \dfrac{1}{2} \dfrac{d}{dt} \int_{\Gamma(t)} w^2 + \delta_\Gamma \int_{\Gamma(t)} |\tgrad w|^2 + \frac 12\int_{\Gamma(t)} w^2\sgrad \cdot \mathbf{V}_\Gamma + \dfrac{1}{\delta_k}\int_{\Gamma(t)} g(u,w)w &= \dfrac{1}{\delta_{k'}} \int_{\Gamma(t)} zw\\ 
    &\leq \dfrac{1}{2\delta_{k'}} \int_{\Gamma(t)} z^2 +  w^2.
\end{align*}
Integrate now over $[0,t]$ and manipulate to obtain 
\begin{align*}
    \int_{\Gamma(t)} w(t)^2+ 2\delta_\Gamma \int_0^t \int_{\Gamma(s)} |\tgrad w|^2 + \dfrac{2}{\delta_k}\int_0^t \int_{\Gamma(s)} g(u,w)w &\leq \dfrac{1}{\delta_{k'}} \int_0^t\int_{\Gamma(s)} z^2 + w^2\\
    &\quad +   \int_0^t\int_{\Gamma(s)} w^2|\sgrad \cdot \mathbf{V}_\Gamma | + \int_{\Gamma_0} w_0^2
\end{align*}
and here we use Gronwall's lemma.
\end{proof}

\begin{lem}[Energy estimate on $u$]\label{lem:u_energy}
We have
\begin{align}\label{eq:u_energy}
&\norm{u}{L^\infty_{L^2(\Omega)}}^2 \leq \left(\dfrac{\norm{z}{L^2_{L^2(\Gamma)}}^2}{\delta_\Omega\delta_{k'}} + \norm{u_0}{L^2(\Omega_0)}^2\right) e^{\frac{1}{\delta_\Omega}\left(C\left(\dfrac{1}{\delta_{k'}} + \delta_\Omega\norm{j}{\infty}\right)^2 + \delta_\Omega\linfdVO\right)T},
\end{align}
\begin{align*}
\nonumber \norm{\grad u}{L^2_{L^2(\Omega)}}^2 + \frac{2}{\delta_k}\int_0^T\int_{\Gamma(t)}g(u,w)u &\leq \dfrac{\norm{z}{L^2_{L^2(\Gamma)}}^2}{\delta_{k'}} + \left(C\left(\dfrac{1}{\delta_{k'}} + \delta_\Omega\norm{j}{\infty}\right)^2 + \delta_\Omega\linfdVO\right)\norm{u}{L^2_{L^2(\Omega)}}^2\\
&\quad + \delta_\Omega\norm{u_0}{L^2(\Omega_0)}^2.
\end{align*}
\end{lem}
\begin{proof}
Testing the equation for $u$ by $u$ (and using the integration by parts formula \eqref{eq:bulkVeryUsefulIdentity}), we get
\begin{align*}
    \dfrac{1}{2} \dfrac{d}{dt} \int_{\Omega(t)} u^2 + \frac{1}{\delta_\Omega} \int_{\Omega(t)} |\nabla u|^2 + \frac 12 \int_{\Omega(t)} u^2 \grad \cdot \mathbf{V}_\Omega 
        &\leq \frac{1}{\delta_\Omega}\left(\dfrac{1}{\delta_{k'}} \int_{\Gamma(t)} \frac{1}{2}z^2 +  \frac{1}{2} u^2 - \dfrac{1}{\delta_k}\int_{\Gamma(t)} g(u,w) u \right)\\
        &\quad + \frac 12\int_{\Gamma(t)}u^2j.
\end{align*}
Re-arranging and integrating over $[0,t]$ leads to, denoting
\[\theta:=\dfrac{1}{\delta_{k'}} + \delta_\Omega\norm{j}{\infty},\]
the expression
\begin{align*}
    \delta_\Omega\int_{\Omega(s)} u(t)^2+  2\int_0^t \int_{\Omega(s)} |\nabla u|^2+ \dfrac{2}{\delta_k}\int_0^t \int_{\Gamma(s)} g(u,w)u  
    &\leq \dfrac{\norm{z}{L^2_{L^2(\Gamma)}}^2}{\delta_{k'}} + \theta\int_0^t \int_{\Gamma(s)} u^2 \\
    &\quad + \delta_\Omega\linfdVO\int_0^t\int_{\Omega(s)}u^2 + \delta_\Omega\int_{\Omega_0} u_0^2.
\end{align*}
Now by the interpolated trace theorem 
we have
\begin{align*}
\theta\int_{\Gamma(s)}  u^2 &\leq  C\theta^2  \int_{\Omega(s)} u^2 +  \int_{\Omega(s)} |\nabla u|^2.
\end{align*}
Inserting this above, we find
\begin{align*}
\delta_\Omega\int_{\Omega(s)} u(t)^2+  \int_0^t \int_{\Omega(s)} |\nabla u|^2+ \dfrac{2}{\delta_k}\int_0^t \int_{\Gamma(s)} g(u,w)u 
    &\leq \dfrac{\norm{z}{L^2_{L^2(\Gamma)}}^2}{\delta_{k'}} + \left(C\theta^2 + \delta_\Omega\linfdVO\right)\int_0^t \int_{\Omega(s)} u^2\\
    &\quad + \delta_\Omega\int_{\Omega_0} u_0^2 .
\end{align*}
Gronwall's inequality now finally implies the result.
\end{proof} 
The equation for the sum $v=w+z$ is
\begin{align*}
\mdd v + v\sgrad \cdot \Vp - \delta_{\Gamma'}\slap v + \sgrad \cdot (\JG v) &= (\delta_\Gamma-\delta_{\Gamma'})\slap w,\\
v(0) &= w_0 + z_0.
\end{align*}
%
Now that we have a bound for $\sgrad w$, we can test the above equation with $v$ and use Young's inequality on the right-hand side and the above bound on $w$ to get a bound
independent of $\delta_k.$ Though we do already have \cref{lem:w_energy}, the following is useful because it shows the influence of the diffusion constants and implies in particular the corollary following the result.
 
\begin{lem}[Energy estimate on $v$]\label{lem:v_energy}
We have
\begin{align*}\label{eq:v_energy}
    \norm{v}{L^\infty_{L^2(\Gamma)}}^2 \leq e^{\linfdVG  T}\left(\dfrac{|\delta_\Gamma - \delta_{\Gamma'}|^2}{\delta_{\Gamma'}} \norm{\tgrad w}{L^2_{L^2(\Gamma)}}^2 + \norm{v_0}{L^2(\Gamma_0)}^2\right), \\
\delta_\Gamma'\norm{\sgrad v}{L^2_{L^2(\Gamma)}}^2 \leq \dfrac{|\delta_\Gamma - \delta_{\Gamma'}|^2}{\delta_{\Gamma'}} \norm{\tgrad w}{L^2_{L^2(\Gamma)}}^2 + \norm{v_0}{L^2(\Gamma_0)}^2 + \linfdVG\norm{v}{L^2_{L^2(\Gamma)}}^2.
\end{align*}
\end{lem}
\begin{proof}
%
From the equation for $v$ and Young's inequality, we get
\begin{align*}
    \dfrac{1}{2} \dfrac{d}{dt}\int_{\Gamma(t)} v(t)^2 + \delta_{\Gamma'} \int_{\Gamma(t)} |\tgrad v|^2 + \frac 12\int_{\Gamma(t)}v^2 \dVG 
    &\leq \dfrac{|\delta_\Gamma - \delta_{\Gamma'}|}{2\eps} \int_{\Gamma(t)} |\tgrad w|^2 + \dfrac{\eps |\delta_\Gamma - \delta_{\Gamma'}|}{2}\int_{\Gamma(t)} |\tgrad v|^2,
\end{align*}
and choosing $\eps = \delta_{\Gamma'}\slash |\delta_\Gamma - \delta_{\Gamma'}|$, 
we obtain
\begin{align*}
     \dfrac{d}{dt}\int_{\Gamma(t)} v(t)^2 + \delta_{\Gamma'} \int_{\Gamma(t)} |\tgrad v|^2 \leq \dfrac{|\delta_\Gamma - \delta_{\Gamma'}|^2}{\delta_{\Gamma'}} \int_{\Gamma(t)} |\tgrad w|^2 + \linfdVG \int_{\Gamma(t)}v^2.
\end{align*}
Gronwall's inequality gives the claim.
\end{proof}

\begin{cor}\label{cor:nice_w_and_z_bound}
We have
\begin{align*}
    \norm{w}{L^\infty_{L^2(\Gamma)}}^2 + \norm{z}{L^\infty_{L^2(\Gamma)}}^2 \leq C\left(\dfrac{|\delta_\Gamma - \delta_{\Gamma'}|^2}{\delta_{\Gamma'}} \norm{\tgrad w}{L^2_{L^2(\Gamma)}}^2 + \norm{v_0}{L^2(\Gamma_0)}^2\right).
\end{align*}
\end{cor}
\begin{proof}
This follows directly from \cref{lem:v_energy} and the non-negativity of $z$ and $w$.
\end{proof}
Now we look for an $L^\infty$-estimate on $w$. 
\begin{cor}\label{lem:wLinfty}
If $ d \leq 3$ and $w_0 \in L^\infty(\Omega_0)$, we have
\begin{align*}
   \norm{w}{L^\infty_{L^\infty(\Gamma)}} \leq e^{(\norm{\sgrad \cdot \mathbf V_\Gamma}{\infty} + \delta_\Gamma)T}\left(\norm{w_0}{L^\infty(\Gamma_0)} + \frac{C}{\delta_\Gamma\delta_{k'}}\norm{z}{L^\infty_{L^2(\Gamma)}}^2\right).
\end{align*}
\end{cor}
\begin{proof}
This follows from the De Giorgi estimate in \cref{lem:deGiorgi}.
\end{proof}
\begin{cor}\label{lem:z_energy_estimate}
We have
\begin{align*}
\delta_\Gamma'\norm{\sgrad z}{L^2_{L^2(\Gamma)}}^2 &\leq 2\left(\dfrac{|\delta_\Gamma - \delta_{\Gamma'}|^2}{\delta_{\Gamma'}} \norm{\tgrad w}{L^2_{L^2(\Gamma)}}^2 + \norm{v_0}{L^2(\Gamma_0)}^2 + \linfdVG\norm{v}{L^2_{L^2(\Gamma)}}^2\right)\\
&\quad + \frac{2\delta_{\Gamma'}}{\delta_\Gamma}\left(
\frac{1}{\delta_{k'}}(\norm{w}{L^2_{L^2(\Gamma)}}^2 +\norm{z}{L^2_{L^2(\Gamma)}}^2) + \linfdVG \norm{w}{L^2_{L^2(\Gamma)}}^2  + \norm{w_0}{L^2(\Gamma_0)}^2\right).
\end{align*}
\end{cor} 
\begin{proof}
This is due to $z=v-w$ and the gradient bounds 
in \cref{lem:w_energy} and \cref{lem:v_energy}. 

\end{proof}



\subsection{A more refined bound}
The most obvious way to get uniform $L^2$ bounds on $w$ and $z$ is to test the $w$ and $z$ equations with the respective solutions as done in the previous section. There we saw that for $w$, we need to control $\frac{1}{\delta_{k'}}\intGt zw$ and so we would need a bound on $z$. The $z$ equation however would require us to control $\frac{1}{\delta_k}\intGt g(u,w)z$; to do this we need more refined bounds. Hence, let us try a different approach in which we can essentially`cancel' the right-hand side terms.

Recall that $v=w+z$ solves the equation
\begin{align*}
\mdd v + v\sgrad \cdot \Vp - \delta_{\Gamma'}\slap v + \sgrad \cdot (\JG v) &= (\delta_\Gamma-\delta_{\Gamma'})\slap w,\\
v(0) &= w_0 + z_0.
\end{align*}
A different form of this equation is useful here because to get a bound on $v$ by testing the above equation with $v$ (as would be the natural first step), we would need control of the gradient of $w$; but to do that we would (see \cref{lem:w_energy}) need to control $\intGt zw$ in an appropriate way.

Define the measurable function
\begin{align*}
    a(t,x) := \begin{cases}
    (\delta_\Gamma w + \delta_\Gamma'z)\slash (w+z) &: w+z > 0\\
    1 &: w+z = 0.
    \end{cases}
\end{align*}
Note that 
\[m \leq a(t,x) \leq M\]
where
\[m=: \min(1,\delta_\Gamma, \delta_\Gamma') \qquad\text{and}\qquad M:= \max(1, \delta_\Gamma, \delta_\Gamma'),\]
so the coefficient is bounded away from zero and infinity. We see that $v$ solves the equation
\begin{equation}\label{eq:equationForv}
\begin{aligned}    
\mdd v + v \sgrad \cdot \Vp  - \Delta (a(t)v) + \sgrad \cdot (\JG v) &= 0,\\
v(0) &= w_0 + z_0.
\end{aligned}
\end{equation}
We have
\begin{align*}
    \norm{v}{L^2_{L^2(\Gamma)}} = \sup_{\substack{\varphi\in L^2_{L^2(\Gamma)}\\   \|\varphi\|_{L^2(L^2)} = 1}}\int_0^T\int_{\Gamma(t)} v\varphi.
\end{align*}
For each such $\varphi$, take $\eta$ to be the solution of the backwards heat equation

\begin{equation}\label{eq:heat_eq_back}
\begin{aligned}
     \mdd \eta(t) + a(t)\slap \eta + \JG\cdot \sgrad \eta &= \varphi&&\text{on }\Gamma(t),\\\
    \eta(T) &= 0&&\text{on }\Gamma(T).\
\end{aligned}
\end{equation}
Then the above becomes
\begin{align}
\nonumber     \norm{v}{L^2_{L^2(\Gamma)}} &= \sup_{\substack{\varphi\in L^2_{L^2(\Gamma)}\\   \|\varphi\|= 1}}\int_0^T\int_{\Gamma(t)}v (\mdd \eta(t) + a(t)\slap \eta + \JG\cdot \sgrad \eta )\\
\nonumber     &= \sup_{\substack{\varphi\in L^2_{L^2(\Gamma)}\\   \|\varphi\|= 1}}-\int_{\Gamma_0}v(0)\eta(0) + \int_0^T\int_{\Gamma(t)}-\mdd v \eta - v\eta \sgrad \cdot \mathbf{V}_p + \slap (a(t)v)  \eta - \sgrad \cdot (\JG v)\eta \tag{using the divergence theorem identity \eqref{eq:identity1}}\\
\nonumber     &= -\sup_{\substack{\varphi\in L^2_{L^2(\Gamma)}\\   \|\varphi\|= 1}}\int_{\Gamma_0}v(0)\eta(0)\tag{since $v$ solves \eqref{eq:equationForv}}\\   \nonumber &\leq \norm{w_0 + z_0}{L^2(\Gamma_0)} \, \sup_{\substack{\varphi\in L^2_{L^2(\Gamma)}\\   \|\varphi\|= 1}}\norm{\eta(0)}{L^2(\Gamma_0)}\\
    &\leq \norm{w_0 + z_0}{L^2(\Gamma_0)} \sup_{\substack{\varphi\in L^2_{L^2(\Gamma)}\\   \|\varphi\|= 1}}\norm{\eta}{C^0_{L^2(\Gamma)}}.\label{eq:norm_char}
    \end{align}
    Hence we need to estimate the above norm of $\eta$. First, let us study \eqref{eq:heat_eq_back}.
\begin{lem}\label{lem:bhe}
We have $\eta\in \mathbb{W}(H^1(\Gamma), L^2(\Gamma))$ 
and 
\begin{align*}
\norm{\eta}{L^\infty_{L^2(\Gamma)}}  + \norm{\sgrad \eta}{L^\infty_{L^2(\Gamma)}}  + \norm{\slap \eta}{L^2_{L^2(\Gamma)}}  + \norm{\mdd \eta}{L^2_{L^2(\Gamma)}} \leq C\norm{\varphi}{L^2_{L^2(\Gamma)}}
\end{align*}
where $C=C(M,m^{-1})$.
\end{lem}
\begin{proof}
By reversing time, i.e. setting $\mathbf{J} = -\JG(T-t)$, $\vartheta(t)=\eta(T-t)$, $b(t)= a(T-t)$ and $g(t)=-\varphi(T-t)$, and defining the family of surfaces
\[\hat \Gamma(t) := \Gamma(T-t)\]
and the parametric material derivative
\[\hat{\partial^\bullet} \vartheta(t) := \vartheta_t(t) + \grad \vartheta(t) \cdot \hat{\mathbf{V}}(t), \qquad \hat{\mathbf{V}}(t) = -\mathbf{V}(T-t),\]
the problem \eqref{eq:heat_eq_back} is equivalent to 
\begin{equation*}
\begin{aligned}
    \hat{\partial^\bullet} \vartheta (t) - b(t) \Delta_\Gamma \vartheta(t)  + \mathbf{J}\cdot \sgrad \vartheta  &= g(t) &&\text{on }\hat\Gamma(t),\\
    \vartheta(0) &= 0 &&\text{on } \hat\Gamma_0,
\end{aligned}
\end{equation*}
with $b\in L^\infty_{L^\infty(\hat\Gamma)}$ and $g\in L^2_{L^2(\hat\Gamma)}$. The claim follows from \cref{prop:BHE}. 
\end{proof}

We finally come to the bound, which is uniform in $\delta_k$ and $\delta_{k'}$ (but does depend on the diffusion coefficients).
\begin{lem}\label{prop:bdOnvDuality}
The following bound holds
\begin{equation*}
\norm{v}{L^2_{L^2(\Gamma)}} + \norm{w}{L^2_{L^2(\Gamma)}} + \norm{z}{L^2_{L^2(\Gamma)}} \leq C(M, m^{-1})\norm{w_0 + z_0}{L^2(\Gamma_0)}
\end{equation*}
where $m:= \min(1,\delta_\Gamma, \delta_\Gamma')$ and $M:= \max(1, \delta_\Gamma, \delta_\Gamma')$.
\end{lem}
\begin{proof}
Using simply the continuous embedding $\mathbb{W}(H^1(\Gamma), H^1(\Gamma)^*) \cts C^0_{L^2(\Gamma)}$ and the estimate of \cref{lem:bhe}, we easily derive
\begin{align*}
\norm{\eta}{C^0_{L^2(\Gamma)}} \leq C\norm{\varphi}{L^2_{L^2(\Gamma)}}
\end{align*}
where $C=C(M,m^{-1})$.
Utilising this and the characterisation in \eqref{eq:norm_char},
\begin{align*}
    \norm{v}{L^2_{L^2(\Gamma)}}
    &\leq 
    C \norm{w_0 + z_0}{L^2(\Gamma_0)},
\end{align*}
so that $v$ and hence $w$ and $z$ are bounded in $L^2_{L^2(\Gamma)}$ uniformly in $\delta_k$ and $\delta_k$.
\end{proof}
\subsection{Bound on the difference quotients of $w$}\label{sec:bound_diff_quotients_w}

In the previous sections we established bounds for $w$ and $z$ in $L^2_{L^2(\Gamma)}$ which are independent of both $\delta_k$ and $\delta_k'$, but dependent on the diffusion coefficients $\delta_\Gamma$ and $\delta_\Gamma'$. These should enable us to obtain bounds on the difference quotients for the pullback of $w$; this will come in use to obtain strong convergence by application of the Aubin--Lions--Simon compactness theorem, see \cref{thm:aubin_lions}.

We introduce the following positive-definite (with a constant that is uniform in time) matrix and its determinant
\begin{equation*}
    \mathbf{A}_t :=
 (\mathbf{D}_{\Gamma_0}\Phi_t)^{\T} \mathbf{D}_{\Gamma_0}\Phi_t + \nu_0\otimes \nu_0, 
\qquad \qquad a_t := \det \mathbf{A}_t.
\end{equation*}
Its inverse has the expression \cite[Proposition 4.1]{ChuDjuEll20})
\[(\mathbf{A}_t)^{-1} = \phi_{-t}((\mathbf{D}_{\Gamma(t)}\Phi^t_0) (\mathbf{D}_{\Gamma(t)}\Phi^t_0)^T) + \nu_0 \otimes \nu_0,\]
which is utilised in the following expression for the pullback of the gradient of a sufficiently smooth function $y\colon \Gamma(t)\to \R$: 
\begin{align*}
    \phi_{-t} \left(\grad_{\Gamma(t)} y\right) &= \mathbf{D}_{\Gamma_0}\Phi_t (\mathbf A_t)^{-1}\grad_{\Gamma_0} \left(\phi_{-t}  y\right).
\end{align*}

\begin{lem} 
The pullback $\tilde{w} := \phi_{-(\cdot)}w$ of $w$ satisfies 
\begin{align*}
\intGz \tilde{w}' \psi + \delta_\Gamma \intGz [\mathbf{B}(t)\sgrad \tilde{w}]\cdot \sgrad \psi + J_t\psi[\mathbf{B}(t)\sgrad \tilde{w}]\cdot \sgrad (1\slash J_t)  + \intGz \tilde{w}\psi\phi_{-t}(\dVG)\\
 + \intGz \phi_{-t}(\JG) \cdot  \mathbf{C}(t) \sgrad \tilde{w}  \psi= \intGz \left(\frac{\tilde{z}}{\delta_{k'}} - \frac{g(\tilde{u},\tilde{w})}{\delta_k}\right)\psi
\end{align*}
for all $\psi \in H^1(\Gamma_0)$ where 
\[\mathbf{B}(t):=  (\mathbf A_t)^{-T}(\mathbf{D}_{\Gamma_0}\Phi_t)^T \mathbf{D}_{\Gamma_0}\Phi_t (\mathbf A_t)^{-1} \qquad\text{and}\qquad \mathbf{C}(t) := \mathbf{D}_{\Gamma_0} \Phi_t(\mathbf{A}_t)^{-1}.\]
\end{lem}
\begin{proof}
The weak formulation for $w$ can be written as
\begin{align*}
\intGt \mdd w \eta + \delta_\Gamma \intGt \sgrad w \sgrad \eta + \intGt w\eta \dVG + \JG \cdot \sgrad w \eta = \intGt \left(\frac{z}{\delta_{k'}}  - \frac{g(u,w)}{\delta_k}\right)\eta
\end{align*}
for every $\eta \in L^2_{H^1(\Gamma)}$. Setting $\tilde{w}:= \phi_{-(\cdot)}w$, supposing that  $\eta=\phi_t \tilde \varphi$ for arbitrary $\tilde \varphi \in H^1(\Gamma_0)$, and defining $J_t := |\mathbf{D}_{\Gamma_0}\Phi_t|$, we can pull back the integrals onto $\Gamma_0$:
\begin{align*}
\intGz \tilde{w}' \tilde \varphi J_t + \delta_\Gamma \intGz \phi_{-t}(\sgrad w \sgrad \eta)J_t + \intGz \tilde{w}\tilde \varphi \phi_{-t}(\dVG)J_t + \phi_{-t}(\JG \cdot \sgrad w)\tilde \varphi J_t\\
 = \intGz \left(\frac{\tilde{z}}{\delta_{k'}} - \frac{g(\tilde{u},\tilde{w})}{\delta_k}\right)\tilde \varphi J_t.
\end{align*}
Thus with $\mathbf{C}(t) := \mathbf{D}_{\Gamma_0} \Phi_t(\mathbf{A}_t)^{-1}$
and writing $\sgrad \tilde \varphi = \sgrad (\tilde \varphi J_t \slash J_t)$, 
\begin{align*}
\intGz \tilde{w}' \tilde \varphi J_t + \delta_\Gamma \intGz J_t[\mathbf{B}(t)\sgrad \tilde{w}]\cdot \sgrad (\tilde \varphi J_t\slash J_t) + \intGz \tilde{w}\tilde \varphi \phi_{-t}(\dVG)J_t + \phi_{-t}(\JG) \cdot  \mathbf{C}(t) \sgrad \tilde{w}  \tilde \varphi J_t\\
 = \intGz \left(\frac{\tilde{z}}{\delta_{k'}} - \frac{g(\tilde{u},\tilde{w})}{\delta_k}\right)\tilde \varphi J_t.
\end{align*}
Now since $\tilde \varphi$ is arbitrary and multiplication by $J_t$ is an isomorphism from $H^1(\Gamma_0)$ to itself, we obtain
\begin{align*}
\intGz \tilde{w}' \psi + \delta_\Gamma \intGz J_t[\mathbf{B}(t)\sgrad \tilde{w}]\cdot \sgrad (\psi\slash J_t) + \intGz \tilde{w}\psi\phi_{-t}(\dVG) + \phi_{-t}(\JG) \cdot  \mathbf{C}(t) \sgrad \tilde{w}  \psi\\
 = \intGz \left(\frac{\tilde{z}}{\delta_{k'}} - \frac{g(\tilde{u},\tilde{w})}{\delta_k}\right)\psi \qquad \forall \psi \in H^1(\Gamma_0).
\end{align*}
Expanding the term $\sgrad (\psi\slash J_t) = (1\slash J_t)\sgrad \psi  + \psi \sgrad (1\slash J_t)$, we get the desired weak formulation.
\end{proof}

\begin{lem} \label{lem:difference_quotients_w}
Set $\mathbf{J}_0 (t):= \phi_{-t}(\JG(t))$. For $h <T$, we have 
\begin{align*}\label{eq:w_diffquot}
&\frac 1h\int_0^{T-h}\int_{\Gamma_0} |\tilde{w}(t+h)-\tilde{w}(t)|^2\\
 &\leq  \frac{2}{\delta_{k'}}\norm{\tilde{z}}{L^2(0,T; L^2(\Gamma_0))}\norm{ \tilde{w}}{L^2(0,T-h;L^2(\Gamma_0))} +  \frac{2}{\delta_k}\norm{g(\tilde{u}, \tilde{w})}{L^1(0,T;L^1(\Gamma_0))}
\norm{ \tilde{w}}{L^\infty(0,T;L^\infty(\Gamma_0))} \\
&\quad + C\delta_\Gamma \|\tgrad \tilde{w}\|_{L^2(0,T;L^2(\Gamma_0))}^2 + C(\delta_\Gamma+\norm{\mathbf{J}_0}{\infty})\norm{\tilde{w}}{L^2(0,T;L^2(\Gamma_0))}\norm{\sgrad \tilde{w}}{L^2(0,T;L^2(\Gamma_0))}\\
&\quad + C \norm{\tilde{w}}{L^2(0,T;L^2(\Gamma_0))}^2
\end{align*}
where $C$ is independent of relevant parameters.
\end{lem}
\begin{proof}
We have, for $h>0$,
\begin{align*}
&\int_{\Gamma_0}|\tilde{w}(t+h)- \tilde{w}(t)|^2\\
 &= \int_0^h\frac{d}{ds}\int_{\Gamma_0}(\tilde{w}(t+h)-\tilde{w}(t))(\tilde{w}(t+s)-\tilde{w}(t))\;\mathrm{d}s\\
& =\int_0^h\langle \tilde{w}'(t+s), \tilde{w}(t+h)-\tilde{w}(t)\rangle\;\mathrm{d}s\\
&=\int_0^h\int_{\Gamma_0}\left((\delta_{k'}^{-1} \tilde{z}(t+s) - \delta_k^{-1} g(\tilde{u}(t+s),\tilde{w}(t+s))\right) (\tilde{w}(t+h)-\tilde{w}(t))\;\mathrm{d}s\\
&\quad -  \int_0^h \int_{\Gamma_0} \delta_\Gamma \mathbf{B}(t+s)\sgrad \tilde{w}(t+s) \sgrad (\tilde{w}(t+h)-\tilde{w}(t))\\
&\quad-  \int_0^h \int_{\Gamma_0}  \delta_\Gamma J_{t+s}(\tilde{w}(t+h)-\tilde{w}(t))\mathbf{B}(t+s)\sgrad  \tilde{w}(t+s) \sgrad (1\slash J_{t+s})\;\mathrm{d}s\\
&\quad   -  \int_0^h\int_{\Gamma_0}   \tilde{w}(t+s) (\tilde{w}(t+h)-\tilde{w}(t)) V_0(t+s) + \mathbf{J}_0(t+s)\cdot \mathbf C(t+s) \sgrad \tilde{w}(t+s)(\tilde{w}(t+h)-\tilde{w}(t)) \;\mathrm{d}s,
\end{align*}
where we set $V_0(t) := \phi_{-t}(\sgrad \cdot \VG(t))$ and $\mathbf{J}_0 (t):= \phi_{-t}(\JG(t))$. 
Now, we can integrate this over $t \in (0,T-h)$ and interchange the integrals $\int_0^{T-h}\int_0^h\ds\dt = \int_0^h \int_0^{T-h}\dt\ds$ on the right-hand side. 
Estimating then the resulting right-hand side, we see first of all that the first term becomes
\begin{align*}
&\int_0^h \int_0^{T-h}\int_{\Gamma_0}(\delta_{k'}^{-1} \tilde{z}(t+s) - \delta_k^{-1} g(\tilde{u}(t+s), \tilde{w}(t+s))) (\tilde{w}(t+h)-\tilde{w}(t))\;\mathrm{d}t\;\mathrm{d}s \\
&\leq  \frac{h}{\delta_{k'}}\norm{ \tilde{z}(\cdot + s)}{L^2(0,T-h; L^2(\Gamma_0))}\norm{ \tilde{w}(\cdot + h)- \tilde{w}(\cdot)}{L^2(0,T-h;L^2(\Gamma_0))}\\
&\quad +  \frac{h}{\delta_k}\norm{ g(\tilde{u}(\cdot+s), \tilde{w}(\cdot+s))}{L^1(0,T-h;L^1(\Gamma_0))}
\norm{ \tilde{w}(\cdot + h)- \tilde{w}(\cdot)}{L^\infty(0,T-h;L^\infty(\Gamma_0))}.
\end{align*}
To conclude we now estimate the remaining integral terms:
\begin{align*}
 &  \int_0^h\int_0^{T-h} \int_{\Gamma_0} \delta_\Gamma \mathbf{B}(t+s)\sgrad \tilde{w}(t+s) \sgrad (\tilde{w}(t+h)-\tilde{w}(t))\dt\ds\\
 &+  \int_0^h\int_0^{T-h} \int_{\Gamma_0} \delta_\Gamma J_{t+s}(\tilde{w}(t+h)-\tilde{w}(t))\mathbf{B}(t+s)\sgrad  \tilde{w}(t+s) \sgrad (1\slash J_{t+s})\;\mathrm{d}t\;\mathrm{d}s\\
&\quad  +  \int_0^h\int_0^{T-h}\int_{\Gamma_0}   \tilde{w}(t+s) (\tilde{w}(t+h)-\tilde{w}(t)) V_0(t+s)\dt\ds\\
&\quad + \int_0^h\int_0^{T-h}\int_{\Gamma_0} \mathbf{J}_0(t+s)\cdot \mathbf C(t+s) \sgrad \tilde{w}(t+s) (\tilde{w}(t+h)-\tilde{w}(t))\;\mathrm{d}t\;\mathrm{d}s  \\
&\leq  \int_0^h C\delta_\Gamma \|\tgrad \tilde{w}\|_{L^2(0,T;L^2(\Gamma_0))}^2 + C\delta_\Gamma\norm{\tilde{w}}{L^2(0,T;L^2(\Gamma_0))}\norm{\sgrad \tilde{w}}{L^2(0,T;L^2(\Gamma_0))} + C\norm{\tilde{w}}{L^2(0,T;L^2(\Gamma_0))}^2\\
&\quad + C\norm{\mathbf{J}_0}{\infty}\norm{\tilde{w}}{L^2(0,T;L^2(\Gamma_0))}\norm{\sgrad \tilde{w}}{L^2(0,T;L^2(\Gamma_0))}
\;\mathrm{d}s\\
&\leq \Big(C\delta_\Gamma \|\tgrad \tilde{w}\|_{L^2(0,T;L^2(\Gamma_0))}^2 + C(\delta_\Gamma+ \norm{\mathbf{J}_0}{\infty})\norm{\tilde{w}}{L^2(0,T;L^2(\Gamma_0))}\norm{\sgrad \tilde{w}}{L^2(0,T;L^2(\Gamma_0))}\\
&\quad + C \norm{\tilde{w}}{L^2(0,T;L^2(\Gamma_0))}^2\Big)h.
\end{align*}
\end{proof}

\section{The $\delta_k\to 0$ limit}\label{sec:deltakLimit}

We denote the solutions of the problem \eqref{eq:newModel} as $(u_k, w_k, z_k)$ where $k$ is supposed to represent the parameter $\delta_k$. By examining the estimates stated in \cref{lem:w_energy}, \cref{lem:u_energy} and \cref{lem:v_energy}, we see that we need a bound on $\norm{z_k}{L^2_{L^2(\Gamma)}}$ uniform in $\delta_k$ in order to get $u_k, w_k$ and $z_k$ bounded in the energy space $L^\infty_{L^2} \cap L^2_{H^1}$. The duality approach result of \cref{prop:bdOnvDuality} provides exactly such a desired bound on $z_k$. With such an estimate in place, we can then obtain an $L^\infty$ bound on $w_k$ by \cref{lem:wLinfty}. Hence, we have the existence of functions
$$u\in L^\infty_{L^2(\Omega)}\cap L^2_{H^1(\Omega)},\quad \quad w\in L^\infty_{L^\infty(\Gamma)}\cap L^2_{H^1(\Gamma)}, \quad \quad z\in L^\infty_{L^2(\Gamma)}\cap L^2_{H^1(\Gamma)},$$
such that 
\begin{equation}\label{eq:conv_u}
\begin{aligned}
u_k &\weaklyto u &&\text{in $L^2_{H^1(\Omega)}\quad$ and $\quad u_k \weakstar u $ in  $L^\infty_{L^2(\Omega)}$}
\end{aligned}
\end{equation}
and
\begin{equation}\label{eq:conv_w_z}
\begin{aligned}
w_k &\weaklyto w &&\text{in $L^2_{H^1(\Gamma)}\quad$ and $\quad w_k \weakstar w \quad$ in $L^\infty_{L^\infty(\Gamma)}$},\\
z_k &\weaklyto z &&\text{in $L^2_{H^1(\Gamma)}\quad$ and $\quad z_k \weakstar z \quad$ in $L^\infty_{L^2(\Gamma)}$}.
\end{aligned}
\end{equation}
Bearing in mind the assumptions on $g$ in \cref{ass:ong}, we see that these weak convergences are by themselves not enough to obtain the complementarity condition \eqref{eq:gIsZero}. We need a strong convergence, for which the following standard result comes in use.
\begin{theorem}[Aubin--Lions--Simon, \cite{Simon}]\label{thm:aubin_lions}
Let $\{\varphi_k\}$ be a bounded sequence of functions in $L^p(0,T; B)$, where $B$ is a Banach space and $1\leq p\leq \infty$. If in addition
 \begin{enumerate}[label=({\roman*}), itemsep=0pt]
    \item the sequence $\{\varphi_k\}$ is bounded in $L^p(0,T; X)$ where $X\subset B$ is compact,
    \item 
     as $h\to 0$,
    \begin{align*}
        \int_0^{T-h} \|\varphi_k(t+h)-\varphi_k(t)\|^2_B \; \d t \to 0 \quad \text{uniformly in $k$},
    \end{align*}
\end{enumerate}
then there exists $\varphi \in L^p(0,T; X)$ such that, up to a subsequence, 
\begin{align*}
\varphi_k \to \varphi \,\, \text{ in } L^p(0,T; B).
\end{align*}
\end{theorem}

Making use of the $L^1$ bound on $\delta_k^{-1}g(u_k,w_k)$ in \cref{lem:l1boundOng} and the $L^\infty$ bound on $w_k$, we can then bound the difference quotients of the pullback of $w_k$ uniformly via \cref{lem:difference_quotients_w}. This estimate on the difference quotients allows us, via the Aubin--Lions--Simon \cref{thm:aubin_lions}, to obtain the stronger convergence 
\begin{equation*}
    w_k \to w \,\, \text{ in } L^2_{L^2(\Gamma)}.
\end{equation*}

\subsection{Passage to the limit}
Taking $\eta \in \mathcal{V}$ (recall the definition of this space from \eqref{eq:defnTestFnSpace}) with $\eta(T) \equiv 0$, using the equation for $u_k$ and its boundary condition, we derive 
\begin{align}
\nonumber \int_0^T \langle \mdd u_k, \eta\rangle + \frac{1}{\delta_\Omega}\int_0^T\int_{\Omega(t)} \nabla u_k\cdot \nabla \eta  &+ \int_0^T\int_{\Omega(t)}u_k\eta \dVp+ \grad \cdot(\JO u_k)\eta- \int_0^T\int_{\Gamma(t)}ju_k\eta\\
    &\quad  = \dfrac{1}{\delta_\Omega}\int_0^T\intGt \left(\frac{1}{\delta_{k'}} z_k \eta - \dfrac{1}{\delta_{k}}g(u_k, w_k) \eta \right). \label{eq:deltaKLimitStep}
\end{align}
Substituting the equation for $z_k$ on the right-hand side, we obtain
\begin{align*}
      &\int_0^T \langle \mdd u_k, \eta\rangle + \frac{1}{\delta_\Omega} \int_0^T\int_{\Omega(t)} \nabla u_k\cdot \nabla \eta + \int_0^T\int_{\Omega(t)}u_k\eta \dVp + \grad \cdot (\JO u_k)\eta - \int_0^T\int_{\Gamma(t)}ju_k\eta    \\
     & + \dfrac{1}{\delta_\Omega} \int_0^T \langle \mdd z_k, \eta\rangle + \frac{\delta_{\Gamma'}}{\delta_\Omega} \int_0^T\intGt \tgrad z_k\cdot \tgrad \eta + \frac{1}{\delta_\Omega}\int_0^T\int_{\Gamma(t)}z_k\eta\dVp + \sgrad \cdot (\JG z_k)\eta = 0.
\end{align*}
We now integrate by parts on the integrals involving the time derivatives, obtaining
\begin{align}
\nonumber &-\int_{\Omega_0}u_0\eta(0)   -\int_0^T \int_{\Omega(t)} u_k \mdd\eta + \frac{1}{\delta_\Omega}\int_0^T\int_{\Omega(t)} \nabla u_k\cdot \nabla \eta + \int_0^T\int_{\Omega(t)} \grad \cdot (\JO u_k)\eta - \int_0^T\int_{\Gamma(t)}ju_k\eta \\
&-\frac{1}{\delta_\Omega}\int_{\Gamma_0}z_0\eta(0)  - \frac{1}{\delta_\Omega}\int_0^T \intGt z_k \mdd \eta + \frac{\delta_{\Gamma'}}{\delta_\Omega}\int_0^T\intGt \tgrad z_k\cdot \tgrad \eta + \frac{1}{\delta_\Omega}\int_0^T\int_{\Gamma(t)} \sgrad \cdot (\JG z_k)\eta = 0,\label{eq:equation_uk_and_z_k}
\end{align}
and using the convergences in \eqref{eq:conv_u} and \eqref{eq:conv_w_z}, passing to the limit $\delta_k \to 0$ is immediate and leads to
\begin{align*}
&-\int_{\Omega_0}u_0\eta(0) -\int_0^T \int_{\Omega(t)} u \mdd\eta + \frac{1}{\delta_\Omega} \int_0^T\int_{\Omega(t)} \nabla u\cdot \nabla \eta + \int_0^T\int_{\Omega(t)} \grad \cdot (\JO u)\eta - \int_0^T\int_{\Gamma(t)}ju\eta  \\
&-\frac{1}{\delta_\Omega}\int_{\Gamma_0}z_0\eta(0)   - \frac{1}{\delta_\Omega}\int_0^T \intGt z \mdd \eta + \frac{\delta_{\Gamma'}}{\delta_\Omega}\int_0^T\intGt \tgrad z\cdot \tgrad \eta + \frac{1}{\delta_\Omega}\int_0^T\int_{\Gamma(t)} \sgrad \cdot (\JG z)\eta = 0,
\end{align*}
which is precisely \eqref{eq:vws_u_and_z}. We can in a similar way derive the equality \eqref{eq:vws_u_and_w} relating $u$ and $w$ if we use the equation for $w$ on the right-hand side of \eqref{eq:deltaKLimitStep}.
\subsection{Complementarity condition}
Note that we have $u_k\rightharpoonup u$ in $L^2_{H^1(\Omega)}$ and thus by continuity of the trace operator we also obtain $u_k \rightharpoonup u$ in $L^2_{H^{1\slash 2}(\Gamma)}$. Rewriting \eqref{eq:g_l1_l1} as 
\begin{align*}
    \int_0^T \intGt g(u_k, w_k) \leq C \delta_k
\end{align*}
and therefore letting $\delta_k \to 0$ gives $g(u_k, w_k) \to 0$ in $L^1_{L^1(\Gamma)}$. 
This along with \eqref{ass:new_on_g_1} immediately implies that $g(u,w)=0$ and we conclude the proof of \cref{thm:deltakLimit}.

\section{The $\delta_k = \delta_\Gamma = \delta_{\Gamma'} \to 0$ limit}\label{sec:deltakdeltaGammasLimit}

In this section we assume $\delta_\Gamma = \delta_{\Gamma'}=\delta_k$, which creates new dependencies on $k$. As such, the \emph{a priori} estimates we previously established need to be revisited. As before, denote the solutions of the problem as $(u_k, w_k, z_k)$ where $k$ is supposed to represent the parameter $\delta_k$.  
 \begin{enumerate}[label=({\roman*})]
\item First of all, note that we cannot use \cref{prop:bdOnvDuality} like we did in the previous section because the bounds of the cited lemma are not uniform in the diffusion coefficients. 

On the positive side, because $\delta_\Gamma = \delta_{\Gamma'}$, we can avail ourselves of the estimate in \cref{cor:nice_w_and_z_bound}, which gives boundedness of $z_k$ and $w_k$ in $L^\infty_{L^2(\Gamma)}$. This we can feed back into \cref{lem:l1boundOng} and \cref{lem:u_energy} to obtain
\begin{align*}
 \norm{u}{L^\infty_{L^2(\Omega)}}^2 + \norm{\grad u}{L^2_{L^2(\Omega)}}^2 + \frac{1}{\delta_k}\int_0^T\int_{\Gamma(t)} g(u,w)u + \frac{1}{\delta_k}\int_0^T \intGt g(u,w) \leq C.
\end{align*}
\item From \cref{lem:w_energy} and \cref{lem:z_energy_estimate}  we obtain the estimate
    \begin{align}\label{eq:h1_secondcase}
        \delta_k \| \tgrad z_k \|_{L^2_{L^2(\Gamma)}}^2 + \delta_k \|\tgrad w_k \|_{L^2_{L^2(\Gamma)}}^2 \leq C
    \end{align}
(note that we no longer get uniform gradient estimates on $w_k$ or $z_k$), which will come in use later. 

\item We no longer get a uniform $L^\infty$ bound for $w_k$ from \cref{lem:wLinfty} but we can get one easily (thanks to $\delta_\Gamma = \delta_{\Gamma'}$) from \cref{lem:initial_bound_w_z_2} below. 
\item Since we do not have a bound for $\sgrad w_k$, we also do not get uniformly bounded difference quotients for $w_k$ from \cref{lem:difference_quotients_w} (unless we assume additional conditions, see \cref{rem:ifKZerowDiffQuotientsEstimate}; we will not need this in this section). This means that we no longer can get a strong convergence for $w_k$.
\end{enumerate}

We now establish new estimates to overcome the obstacles described above. What we will do is to get a strong convergence now for $u_k$ like we did for $w_k$ in the previous section. This will require us to get an $L^\infty$-estimate for $z$, which is dealt with in the next lemma. In it, note that there is no need for the De Giorgi arguments of \cref{lem:deGiorgi} since $v$ solves the homogeneous problem.
\begin{lem}\label{lem:initial_bound_w_z_2}
If $w_0, z_0\in L^\infty(\Gamma_0)$ and $\delta_\Gamma = \delta_{\Gamma'}$, then 
\begin{align*}
    \|w\|_{L^\infty_{L^\infty(\Gamma)}} + \|z\|_{L^\infty_{L^\infty(\Gamma)}} \leq C,
\end{align*}
where $C$ is independent of all parameters.
\end{lem}

\begin{proof}
With $v=w+z$ as before, the fact that $\delta_\Gamma=\delta_{\Gamma'}$ implies that $v$ is a solution to the heat equation
\begin{align*}
    v_t - \delta_\Gamma \Delta_\Gamma v  + v\dVp + \sgrad \cdot (\JG v) = 0,
\end{align*}
and therefore boundedness of $v$ is a consequence of boundedness of $v_0$ using the standard argument. Since $w,z\geq 0$, the conclusion follows. 
\end{proof}

%

Boundedness for $u$ is more complicated because we are dealing with two sets ($\Omega(t)$ and its boundary) due to the Robin boundary condition.
Consider the following 
\begin{equation}\label{eq:3}
\begin{aligned}
\mdd{a} + a \grad \cdot \mathbf V_p  - \delta_\Omega^{-1} \Delta a + \grad \cdot (\JO a)  &= 0 &\text{in $\Omega(t)$},\\
\delta_\Omega^{-1}\nabla a \cdot \nu &= \delta_\Omega^{-1}\delta_{k'}^{-1} z  + ja&\text{on $\Gamma(t)$},\\
\grad a \cdot \nu &= 0&\text{in $\partial D(t)$},\\
a(0) &= u_0.
\end{aligned}
\end{equation}
This is almost the same equation for $u$, but the non-positive term $-g(u,w)$ (which formally speaking can be thought of as entering as a source term) is omitted, thus we expect that $a \geq u$. The next lemma shows that this is indeed the case.
\begin{lem}\label{lem:u_leq_a}
We have $u \leq a$.
\end{lem}
\begin{proof}
Define $U := ue^{-\lambda_u t}$ so that $\mdd u = e^{\lambda_u t}\md U + \lambda_u u.$ The equation for $U$ is
\begin{equation*}\label{eq:foru}
\begin{aligned}
\md{ U} + U(\grad \cdot \mathbf V_p  + \lambda_u) - \delta_\Omega^{-1} \Delta U  + \grad \cdot (\JO U) &=0&\text{in $\Omega(t)$},\\
\delta_\Omega^{-1}\nabla U \cdot \nu &= \delta_\Omega^{-1}(\delta_{k'}^{-1}e^{- \lambda_u t} z -\delta_k^{-1} e^{- \lambda_u t} g( u, w)) +  jU&\text{on $\Gamma(t)$},\\
\grad U \cdot \nu &=0 &\text{on $\partial D(t)$},\\
U(0) &= u_0.
\end{aligned}
\end{equation*}
In a similar way, the equation for $a$ can be transformed via $A:= ae^{-\lambda_u t}$ to
\begin{equation*}
\begin{aligned}
\mdd{A} + A(\grad \cdot \mathbf V_p  + \lambda_u) - \delta_\Omega^{-1} \Delta A + \grad \cdot (\JO A)  &=0&\text{in $\Omega(t)$},\\
\delta_\Omega^{-1}\nabla A \cdot \nu &= \delta_\Omega^{-1}\delta_{k'}^{-1}e^{- \lambda_u t} z  + jA&\text{on $\Gamma(t)$},\\
\grad A \cdot \nu &= 0&\text{on $\partial D(t)$},\\
A(0) &= u_0.
\end{aligned}
\end{equation*}
Define $Y := U - A$ which satisfies
\begin{equation*}
\begin{aligned}
\mdd{Y} + Y(\grad \cdot \mathbf V_p  + \lambda_u) - \delta_\Omega^{-1} \Delta Y + \grad \cdot (\JO Y)  &=0 &\text{in $\Omega(t)$},\\
\delta_\Omega^{-1}\nabla Y \cdot \nu &= -\delta_\Omega^{-1}\delta_k^{-1} e^{- \lambda_u t} g( u, w)   + jY&\text{on $\Gamma(t)$},\\
\grad Y \cdot \nu &= 0 &\text{on $\partial D(t)$},\\
Y(0) &= 0.
\end{aligned}
\end{equation*}
Test this with $Y^+$ to obtain
\begin{align*}
&\frac 12 \frac{d}{dt}\int_{\Omega(t)}|Y^+|^2 - \frac 12 \int_{\Omega(t)} |Y^+|^2\grad \cdot \Vp + \int_{\Omega(t)} (\grad \cdot \mathbf V_p  + \lambda_u)|Y^+|^2 + \delta_\Omega^{-1}|\grad Y^+|^2 + \frac 12 |Y^+|^2 \grad \cdot \JO\\
&\quad + \frac 12 \int_{\Gamma(t)}j|Y^+|^2\\
&\quad\quad= \int_{\Gamma(t)}-\delta_\Omega^{-1}\delta_k^{-1} e^{- \lambda_u t} g( u, w) Y^+  + j|Y^+|^2\\
&\quad\quad\leq \int_{\Gamma(t)}j|Y^+|^2.
\end{align*}
Simplifying, 
\begin{align*}
\frac 12 \frac{d}{dt}\int_{\Omega(t)}|Y^+|^2  + \int_{\Omega(t)} (\grad \cdot \VO  + \lambda_u)|Y^+|^2 + \delta_\Omega^{-1}|\grad Y^+|^2  
&\leq \frac{\norm{j}{\infty}}{2}\int_{\Gamma(t)}|Y^+|^2\\
&\leq \frac{\norm{j}{\infty}}{2}\int_{\Omega(t)}C_\epsilon|Y^+|^2 + \epsilon|\grad Y^+|^2
\end{align*}
where we used the interpolated trace inequality. Rearranging terms, choosing $\epsilon$ carefully and applying Gronwall's lemma (and using $Y^+(0) = 0$), we get the result.
\end{proof}
\begin{lem}\label{lem:linftyBdu}We have 
\begin{align*}
&\norm{u}{L^\infty_{L^\infty(\Omega)}}\\
&\quad \leq  C_1e^{\frac 12 \norm{\grad \cdot \VO}{\infty}T +  \delta_\Omega C_2\left(\frac 12\delta_\Omega^{-1}\delta_{k'}^{-1} \norm{z}{L^\infty_{L^\infty(\Gamma)}}  + 2\norm{j}{\infty}\right)^2T}\max\left(1, \norm{u_0}{\infty}, \frac 12 C_3\delta_\Omega^{-1}\delta_{k'}^{-1} \norm{z}{L^\infty_{L^\infty(\Gamma)}}\right)\\
&\quad\quad\times (\min(1,\delta_\Omega^{-1})^{-\frac{1}{2\kappa} - \frac 12} + C_4).
\end{align*}
\end{lem}
\begin{proof}
We can use \cref{lem:L-infinity-Neumann} to obtain an $L^\infty$ bound on the solution $a$ of \eqref{eq:3}. Using the previous lemma, we get the result via
\[0 \leq u \leq a \leq \norm{a}{L^\infty_{L^\infty(\Omega)}}.\]
\end{proof}
Using the extra results above we can argue as in the proof of \cref{lem:difference_quotients_w} to establish an estimate on the difference quotients now for $u$.
\subsection{Bound on the difference quotients of $u$}
In a similar way to \cref{sec:bound_diff_quotients_w}, let us first write down the weak formulation that the pullback of $u$ satisfies. Let $\mathbf D\Phi_t$ be the Jacobian matrix of $\Phi_t$  and let $J_t:=\det{\mathbf D\Phi_t}$ be its determinant. Define 
$\omega_t := |(\mathbf D\Phi_t)^{-T} \nu_0|$, where $\nu_0$ is the outward  normal to $\Gamma_0$.
\begin{lem}
Set $V_0(t):=\phi_{-t}(\nabla\cdot\mathbf{V}_\Omega)$, $\mathbf{J}_0 := \phi_{-t}(\JO)$, $j_0:=\phi_{-t}((\mathbf{V}_\Omega-\mathbf{V}_\Gamma)\cdot\tilde{\nu})$, and
\[\mathbf{B}(t):=(\mathbf D\Phi_t)^{-1}(\mathbf D\Phi_t)^{-T}.\]
The pullback $\tilde{u} := \phi_{-(\cdot)}u$ of $u$ satisfies 
\begin{align}
\nonumber \langle \tilde{u}_t, \psi \rangle  + \int_{\Omega_0}\delta_\Omega^{-1}\mathbf{B}(t)\grad \tilde{u} \grad \psi &+\delta_\Omega^{-1} \psi J_t \mathbf{B}(t)\grad  \tilde{u} \grad (1\slash J_t)  +  \tilde{u} \psi V_0 + 
\mathbf{J}_0 \cdot (\mathbf D\Phi_t)^{-T} \grad \tilde{u} \psi\\
&= \int_{\Gamma_0}\frac{1}{\delta_\Omega}\left( \frac{\tilde z}{\delta_{k'}} - \frac{g(\tilde{u}, \tilde{w})}{\delta_k}\right)\psi \omega_t+ j_0\tilde{u}\psi\omega_t \qquad \forall \psi \in H^1(\Omega_0).\label{eq:pulledBackWeakForm}
\end{align}

\end{lem}
\begin{proof}
Pulling back the weak form of the $u$ equation gives (see \cite[Chapter 9, \S 4.2]{Delfour} for the boundary integral)
\[\langle \tilde u_t, J_t\tilde \varphi \rangle + \int_{\Omega_0}\delta_\Omega^{-1}J_t \mathbf{B}(t)\grad \tilde u \grad \tilde \varphi +  \tilde u \tilde \varphi V_0J_t + 
J_t\mathbf{J}_0 \cdot (\mathbf D\Phi_t)^{-T} \grad \tilde u \tilde \varphi = \int_{\Gamma_0}\frac{1}{\delta_\Omega}\left( \frac{\tilde z}{\delta_{k'}} - \frac{g(\tilde u, \tilde w)}{\delta_k}\right)\tilde \varphi J_t \omega_t + j_0\tilde{u}\tilde{\varphi}J_t\omega_t.\]
Now, again using that multiplication by $J_t$ is an isomorphism, this becomes
\[\langle \tilde u_t, \psi \rangle + \int_{\Omega_0}\delta_\Omega^{-1}J_t\mathbf{B}(t)\grad \tilde u \grad ((J_t^0)^{-1}\psi) + \tilde u  \psi V_0 + 
\mathbf{J}_0 \cdot (\mathbf D\Phi_t)^{-T} \grad \tilde u \psi = \int_{\Gamma_0}\frac{1}{\delta_\Omega}\left( \frac{\tilde z}{\delta_{k'}} - \frac{g(\tilde u, \tilde w)}{\delta_k}\right)\psi \omega_t+  j_0\tilde{u}\psi\omega_t.\]
Expanding $J_t\mathbf{B}(t)\grad \tilde u \grad ((1\slash J_t)\psi)  = \mathbf{B}(t)\grad \tilde u \grad \psi + \psi J_t \mathbf{B}(t)\grad \tilde u \grad (1\slash J_t),$ 
integrating by parts in time and relabelling, we obtain \eqref{eq:pulledBackWeakForm}.
\end{proof}


\begin{lem}\label{lem:difference_quotients_u}
We have
\begin{align*}\label{eq:u_diffquot}
    &\frac1h \int_0^{T-h}\int_{\Omega_0}|\tilde{u}(t+h)- \tilde{u}(t)|^2 \leq C\left(\frac{1}{\delta_\Omega}+1 \right)\norm{\nabla \tilde{u}}{L^2(0,T;\Omega_0)}^2 + C\norm{\tilde{u}}{L^2(0,T;\Omega_0)}^2\\
    &\quad + 2\delta_\Omega^{-1}\delta_{k'}^{-1}\norm{\tilde{z}}{L^2(0,T; L^2(\Gamma_0))}\norm{\tilde{u}}{L^2(0,T;L^2(\Gamma_0))}\\
&\quad +  \frac{2}{\delta_\Omega\delta_k}\norm{ g(\tilde{u},\tilde{w})}{L^1(0,T;L^1(\Gamma_0))}
\norm{\tilde{u}}{L^\infty(0,T;L^\infty(\Gamma_0))} +  2\norm{\tilde{u}}{L^2(0,T;L^2(\Gamma_0))}^2.
\end{align*}
\end{lem}

\begin{proof}
Beginning as in the proof of \cref{lem:difference_quotients_w}, we have
\begin{align*}
&\int_{\Omega_0}|\tilde{u}(t+h)- \tilde{u}(t)|^2\\
&=\int_0^h\langle \tilde{u}_t(t+s), \tilde{u}(t+h)-\tilde{u}(t)\rangle\\
&= -\frac{1}{\delta_\Omega}\int_0^h\int_{\Omega_0}\mathbf{B}(t+s)\nabla \tilde{u}(t+s)\cdot \nabla (\tilde{u}(t+h)-\tilde{u}(t)) + J_{t+s}(\tilde{u}(t+h)-\tilde{u}(t))\mathbf{B}(t+s)\grad \tilde{u}(t+s)\grad (1\slash J_{t+s})\\
&\quad - \int_0^h\int_{\Omega_0} \tilde{u}(t+s)(\tilde{u}(t+h)-\tilde{u}(t))V_0(t+s) + \mathbf{J}_0(t+s) \cdot (\mathbf D\Phi_{t+s})^{-T} \grad \tilde{u}(t+s)(\tilde{u}(t+h)-\tilde{u}(t))\\
&\quad +  \int_0^h\int_{\Gamma_0}\frac{1}{\delta_\Omega}\left( \frac{ \tilde{z}(t+s)}{\delta_{k'}} - \frac{g(\tilde{u}(t+s), \tilde{w}(t+s))}{\delta_k}\right)(\tilde{u}(t+h)-\tilde{u}(t))\omega_t+ j_0(t+s)\tilde{u}(t+s)(\tilde{u}(t+h)-\tilde{u}(t))\omega_t.
\end{align*}
We now integrate over $(0, T-h)$ and switch integrals like in \cref{lem:difference_quotients_w} and use \eqref{eq:u_energy} to obtain for the inner integrals 
\begin{align*}
&\frac{1}{\delta_\Omega}\int_0^{T-h}\int_{\Omega_0}\mathbf{B}(t+s)\nabla \tilde{u}(t+s)\cdot \nabla (\tilde{u}(t+h)-\tilde{u}(t)) + J_{t+s}(\tilde{u}(t+h)-\tilde{u}(t))\mathbf{B}(t+s)\grad \tilde{u}(t+s)\grad (1\slash J_{t+s})\\
&\quad + \int_0^{T-h}\int_{\Omega_0} \tilde{u}(t+s)(\tilde{u}(t+h)-\tilde{u}(t))V_0(t+s) + \mathbf{J}_0(t+s) \cdot (\mathbf D\Phi_{t+s})^{-T} \grad \tilde{u}(t+s)(\tilde{u}(t+h)-\tilde{u}(t))\\
&\leq \frac{C}{\delta_\Omega}\norm{\nabla \tilde{u}}{L^2(0,T;L^2(\Omega_0))}^2 + \norm{\tilde{u}}{L^2(0,T;L^2(\Omega_0))}\norm{\nabla \tilde{u}}{L^2(0,T;L^2(\Omega_0))} + \norm{\nabla \tilde{u}}{L^2(0,T;L^2(\Omega_0))}^2  + C\norm{\tilde{u}}{L^2(0,T; L^2(\Omega_0))}^2 \\
&\leq C\left(\frac{1}{\delta_\Omega}+1 \right)\norm{\nabla \tilde{u}}{L^2(0,T;\Omega_0)}^2 + C\norm{\tilde{u}}{L^2(0,T;L^2(\Omega_0))}^2 
\end{align*}
%
and
\begin{align*}
&\int_0^{T-h}\int_{\Gamma_0}\frac{1}{\delta_\Omega}\left( \frac{ \tilde{z}(t+s)}{\delta_{k'}} - \frac{g(\tilde{u}(t+s), \tilde{w}(t+s))}{\delta_k}\right)(\tilde{u}(t+h)-\tilde{u}(t)) \omega_t+ j_0\tilde{u}(t+s)(\tilde{u}(t+h)-\tilde{u}(t))\omega_t\\
&\leq \delta_\Omega^{-1}\delta_{k'}^{-1}\norm{\tilde{z}}{L^2(0,T-h; L^2(\Gamma_0))}\norm{ \tilde{u}(\cdot + h)-\tilde{u}}{L^2(0,T-h;L^2(\Gamma_0))}\\
&\quad +  \frac{1}{\delta_\Omega\delta_k}\norm{ g(\tilde{u}(t+\cdot), \tilde{w}(t+\cdot))}{L^1(0,T-h;L^1(\Gamma_0))}
\norm{ \tilde{u}(\cdot + h)- \tilde{u}}{L^\infty(0,T-h;L^\infty(\Gamma_0))}\\
&\quad +  \norm{\tilde{u}(t+\cdot)}{L^2(0,T-h;L^2(\Gamma_0))}\norm{ \tilde{u}(\cdot + h)- \tilde{u}}{L^2(0,T-h;L^2(\Gamma_0))}\\
&\leq 2\delta_\Omega^{-1}\delta_{k'}^{-1}\norm{\tilde{z}}{L^2(0,T; L^2(\Gamma_0))}\norm{\tilde{u}}{L^2(0,T;L^2(\Gamma_0))}\\
&\quad +  \frac{2}{\delta_\Omega\delta_k}\norm{ g(\tilde{u},\tilde{w})}{L^1(0,T;L^1(\Gamma_0))}
\norm{\tilde{u}}{L^\infty(0,T;L^\infty(\Gamma_0))} +  2\norm{\tilde{u}}{L^2(0,T;L^2(\Gamma_0))}^2.
\end{align*}
\end{proof}
Using the previously established uniform bounds (including the $L^1_{L^1(\Gamma)}$ bound for $\delta_k^{-1} g(u,w)$ and the $L^\infty_{L^\infty(\Omega)}$ bound for $u$ from \cref{lem:linftyBdu}),  
we are led to
\begin{align*}
    \int_0^{T-h}\int_{\Omega_0}|\tilde{u}(t+h)- \tilde{u}(t)|^2 
    &\leq C\left(\frac{1}{\delta_\Omega} + 1 + \dfrac{1}{\delta_\Omega\delta_{k'}}\right)h
\end{align*}
where $C$ is independent of $\delta_k$.

\subsection{Passage to the limit}
Using the estimates in the previous section we have limit functions $$u\in L^\infty_{L^\infty(\Omega)}\cap L^2_{H^1(\Omega)}, \quad\quad w\in L^\infty_{L^\infty(\Gamma)}, \quad\quad z\in L^\infty_{L^\infty(\Gamma)},$$
such that 
\begin{equation*}
\begin{aligned}
u_k &\weakstar u &&\text{in $L^\infty_{L^\infty(\Omega)}$}\quad\text{and}\quad u_k \weaklyto u &&\text{in $L^2_{H^1(\Omega)}$,}\\
w_k &\weakstar w &&\text{in $L^\infty_{L^\infty(\Gamma)}$}\quad \text{and} \quad z_k \weakstar z &&\text{in $L^\infty_{L^\infty(\Gamma)}$}.
\end{aligned}
\end{equation*}
As a result we also get
\[u_k \weakstar u \quad \text{in $L^\infty_{L^\infty(\Gamma)}$}.\]
The estimate on the difference quotients for $u_k$ allows us by \cref{thm:aubin_lions} to obtain the stronger convergence 
\begin{equation*}
    u_k \to u \,\, \text{ in } L^2_{L^2(\Omega)},
\end{equation*}
and by \cite[Lemma 3.9]{ERV} we deduce the same for the trace of $u$:
\begin{equation*}
    u_k \to u \,\, \text{ in } L^2_{L^2(\Gamma)}.
\end{equation*}
Taking again $\eta \in \mathcal{V}$ with $\eta(T)=0$ and recalling the weak formulation relating $u_k$ and $w_k$ in \eqref{eq:equation_uk_and_z_k}, bearing in mind now that $\delta_k = \delta_\Gamma = \delta_{\Gamma'}$, we have
\begin{align*}
&-\int_{\Omega_0}u_0\eta(0)   -\int_0^T \int_{\Omega(t)} u_k \mdd\eta + \frac{1}{\delta_\Omega}\int_0^T\int_{\Omega(t)} \nabla u_k\cdot \nabla \eta + \int_0^T\int_{\Omega(t)} \grad \cdot (\JO u_k)\eta - \int_0^T\int_{\Gamma(t)}ju_k\eta \\
&-\frac{1}{\delta_\Omega}\int_{\Gamma_0}z_0\eta(0)  - \frac{1}{\delta_\Omega}\int_0^T \intGt z_k \mdd \eta + \frac{\delta_{k}}{\delta_\Omega}\int_0^T\intGt \tgrad z_k\cdot \tgrad \eta + \frac{1}{\delta_\Omega}\int_0^T\int_{\Gamma(t)} \sgrad \cdot (\JG z_k)\eta = 0.
\end{align*}
Now, we must use \eqref{eq:identity1} and integrate by parts on the divergence of the jump term since we cannot control the gradient of $z_k$; doing so yields
\begin{align}
\nonumber &-\int_{\Omega_0}u_0\eta(0)   -\int_0^T \int_{\Omega(t)} u_k \mdd\eta + \frac{1}{\delta_\Omega}\int_0^T\int_{\Omega(t)} \nabla u_k\cdot \nabla \eta + \int_0^T\int_{\Omega(t)} \grad \cdot (\JO u_k)\eta - \int_0^T\int_{\Gamma(t)}ju_k\eta \\
&-\frac{1}{\delta_\Omega}\int_{\Gamma_0}z_0\eta(0)  - \frac{1}{\delta_\Omega}\int_0^T \intGt z_k \mdd \eta + \frac{\delta_{k}}{\delta_\Omega}\int_0^T\intGt \tgrad z_k\cdot \tgrad \eta - \frac{1}{\delta_\Omega}\int_0^T\int_{\Gamma(t)} z_k \JG \cdot \sgrad \eta = 0.\label{eq:pickUpFrom}
\end{align}
The estimate \eqref{eq:h1_secondcase} implies
\begin{align*}
    \delta_k \| \tgrad w_k \|_{L^2_{L^2(\Gamma)}} \leq C\sqrt{\delta_k} \quad \text{ and } \quad     \delta_k \| \tgrad z_k \|_{L^2_{L^2(\Gamma)}} \leq C\sqrt{\delta_k},
\end{align*}
so that $\delta_k\sgrad z_k$ (and $\delta_k \sgrad w_k$) converges to $0$ as $\delta_k\to 0$. 
Passing to the limit using the convergences obtained above yields
\begin{align*}
    &-\int_{\Omega_0}u_0\eta(0) -\int_0^T \int_{\Omega(t)} u \mdd\eta + \frac{1}{\delta_\Omega}\int_0^T\int_{\Omega(t)} \nabla u\cdot \nabla \eta + \int_0^T\int_{\Omega(t)}   \grad \cdot (\JO u)\eta- \int_0^T\int_{\Gamma(t)} ju \eta\\ &-\frac{1}{\delta_\Omega}\int_{\Gamma_0}z_0\eta(0)  - \frac{1}{\delta_\Omega}\int_0^T \int_{\Gamma(t)} z \mdd \eta -\frac{1}{\delta_\Omega} \int_0^T \int_{\Gamma(t)} z\JG \cdot \sgrad \eta = 0.
\end{align*}
The $w$ case follows similarly.

\subsection{Complementarity condition}

The complementarity condition in this case can be obtained in the same way as before, first by obtaining that $g(u_n,w_n) \to 0$ as previously and now by using \eqref{ass:new_on_g_2}.

\section{The $ \delta_\Omega = \delta_k = \delta_{k'}^{-1} = \delta_\Gamma = \delta_\Gamma' \to 0$ limit}\label{sec:thirdLimit}
Now, we have the following:
 \begin{enumerate}[label=({\roman*})]

    \item Since $\delta_\Gamma = \delta_\Gamma'$,   \cref{lem:initial_bound_w_z_2} is still useful as it gives uniform bounds for $w_k$ and $z_k$ in $L^\infty_{L^\infty(\Gamma)}$.
    \item  For the gradients, we have \eqref{eq:h1_secondcase} still:
    \begin{align}\label{eq:grad_third}
        \delta_k \|\tgrad w\|_{L^2(L^2)}^2 + \delta_k \|\tgrad z\|_{L^2(L^2)}^2 \leq C\tag{\ref{eq:h1_secondcase}},
    \end{align}
    and for $g(u,w)$, note that the estimate \eqref{eq:g_l1_l1} implies
\begin{equation}
\dfrac{1}{\delta_\Omega\delta_k} \int_0^T \intGt g(u,w) \leq \|z\|_{L^1_{L^1(\Gamma)}}  + \|u_0\|_{L^2(\Omega)}.\label{eq:doubleDeltasGEstimate}
\end{equation}

     \item Examining \cref{lem:u_energy}, we obtain bounds on $u$ in $L^\infty_{L^2(\Omega)}$ and its gradient in $L^2_{L^2(\Omega)}$ because $\delta_{k'}^{-1}=\delta_k$ and due to the boundedness of $z$. In fact, the gradient estimate in \cref{lem:u_energy} implies
\begin{align}
\frac{1}{\delta_\Omega}\norm{\grad u}{L^2_{L^2(\Omega)}}^2   &\leq \dfrac{\norm{z}{L^2_{L^2(\Gamma)}}^2}{\delta_\Omega\delta_{k'}} + \left(\frac{C}{\delta_\Omega}\left(\dfrac{1}{\delta_{k'}} + \delta_\Omega\norm{j}{\infty}\right)^2 + \linfdVO\right)\norm{u}{L^2_{L^2(\Omega)}}^2 + \norm{u_0}{L^2(\Omega_0)}^2\label{eq:gradUVanishesEstimate}
\end{align}
and the right-hand side is bounded uniformly. 

 Looking at \cref{lem:linftyBdu}, we also get a uniform bound for $u$ in $L^\infty$.

 \item  \cref{lem:difference_quotients_u}, after making use of the estimates \eqref{eq:doubleDeltasGEstimate} and \eqref{eq:gradUVanishesEstimate} in combination with the $L^\infty$-estimate, provides a bound on the difference quotients of $u$.



\end{enumerate}
In summary, we obtain 
\begin{align*}
u_k \rightharpoonup u \,\, \text{ in } L^2_{H^1(\Omega)} \cap L^\infty_{L^\infty(\Omega)}, \quad w_k \weakstar w \,\, \text{ in } L^\infty_{L^\infty(\Gamma)}, \quad z_k \weakstar z \,\, \text{ in } L^\infty_{L^\infty(\Gamma)}.
\end{align*}
By \eqref{eq:gradUVanishesEstimate}, we have
\[\grad u_k \to \grad u \equiv 0\quad\text{in $L^2_{L^2(\Omega)}$}.\]
By continuity of the trace operator we also have 
\begin{align*}
    u_k \rightharpoonup u \,\, \text{ in } L^2_{H^{1\slash 2}(\Gamma)},
\end{align*}
and by the compactness result of \cref{thm:aubin_lions} (with  $X=H^1(\Omega_0)$ and $B=L^2(\Omega_0)$), we conclude the stronger convergence 
\begin{align*}
        u_k &\to u \,\, \text{ in } L^2_{L^{2}(\Omega)},
\end{align*}
and we again have
\[    u_k \to u \,\, \text{ in } L^2_{L^2(\Gamma)}.
\]
Note that these are exactly the same convergences we obtained in \cref{sec:deltakdeltaGammasLimit} but here we additionally have $\grad u_k \to 0$. 
\begin{remark}\label{rem:ifKZerowDiffQuotientsEstimate}
If $\mathbf{J}_\Gamma \equiv 0$, the difference quotient estimate for $w_k$ remains uniform: in the estimate given in \cref{lem:difference_quotients_w}, we see that the term $\mathbf{J}_0$ would vanish.  Then the estimate \eqref{eq:grad_third} is sufficient to bound the difference quotient uniformly:
    \begin{align*}
\int_0^{T-h}\|w(t+h)- w(t)\|_{L^2(\Gamma)}^2 \leq Ch.
\end{align*}
By  \cref{thm:aubin_lions} with $X=L^2(\Gamma_0)$, $B=H^{-1/2}(\Gamma_0)$ we would get 
    $w_k \to w$  in  $L^2_{H^{-1/2}(\Gamma)}.$
\end{remark}

\subsection{Passage to the limit}
%
Multiplying through \eqref{eq:pickUpFrom} by $\delta_\Omega$, we get
\begin{align*}
-\delta_\Omega\int_{\Omega_0}u_0\eta(0)   &-\delta_\Omega\int_0^T \int_{\Omega(t)} u_k \mdd\eta + \int_0^T\int_{\Omega(t)} \nabla u_k\cdot \nabla \eta + \delta_\Omega\int_0^T\int_{\Omega(t)} \grad \cdot (\JO u_k)\eta - \delta_\Omega\int_0^T\int_{\Gamma(t)}ju_k\eta \\
&-\int_{\Gamma_0}z_0\eta(0)  - \int_0^T \intGt z_k \mdd \eta + \delta_\Omega\int_0^T\intGt \tgrad z_k\cdot \tgrad \eta - \int_0^T\int_{\Gamma(t)} z_k \JG \cdot \sgrad \eta = 0.
\end{align*}
As in the previous limit, due to \eqref{eq:h1_secondcase}, the integral involving $\sgrad z_k$ converges to $0$ as $\delta_k\to 0$. In the limit, we obtain
\begin{align*}
\nonumber &  -\int_{\Gamma_0}z_0\eta(0) - \int_0^T\intGt \mdd \eta z - \int_0^T\intGt z \JG\cdot \sgrad \eta = 0.
\end{align*}
A similar argument gives
\begin{align*}
\nonumber &  -\int_{\Gamma_0}w_0\eta(0) - \int_0^T\intGt \mdd \eta w - \int_0^T\intGt w \JG\cdot \sgrad \eta = 0.
\end{align*}
\subsection{Complementarity condition}
Using \eqref{eq:doubleDeltasGEstimate}, we have $g(u_k, w_k) \to 0$ in $L^1_{L^1(\Gamma)}$ and 
assumption \eqref{ass:new_on_g_2} implies that $g(u,w)=0$. 

\subsection{The Dirichlet limit}\label{sec:thirdLimitDirichlet}
Let us finally address \cref{thm:fourDeltasLimitDirichlet}. We still obtain the same estimates on $w_k$ and $z_k$ as in \cref{sec:thirdLimit}. We now look for an $L^\infty$-estimate on $u$, which we recall satisfies
\begin{align*}
 \mdd{u}+u\nabla\cdot\mathbf V_p-\frac{1}{\delta_\Omega}\Delta u +\nabla \cdot\left(\mathbf{J}_\Omega u\right) &=0&\text{ in }\Omega(t),\\
\nabla u \cdot \nu &=  \frac{1}{\delta_{k'}} z-\frac{1}{\delta_k}g(u,w) + \delta_\Omega ju&\text{ on }\Gamma(t),\\
u &= u_D &\text{on $\partial D(t)$},\\
u(0) &= u_0.
\end{align*}
We need a preliminary result for a related equation. The following lemma essentially modifies the $L^\infty$-estimate of \cref{lem:L-infinity-Neumann} to allow for the Dirichlet boundary condition.
\begin{lem}\label{lem:L-infinity-Neumannd}
Let $y \in L^\infty_{L^\infty(\Gamma)}$, $h \in L^\infty_{L^\infty(\partial D)} \cap L^2_{H^{1\slash 2}(\Gamma)}$ with $h \geq 0$, $a_0\in L^\infty(\Omega_0)$ and let $a$ be the nonnegative solution to
		\begin{equation*}
			\begin{aligned}
				\mdd{a} + a\grad \cdot\Vp - D\Delta a + \grad\cdot(\JO a) &= 0 &&\text{ in } \Omega(t),\\
				D\grad a \cdot \nu - aj &= y &&\text{ on } \Gamma(t),\\
                a &= h &&\text{ on } \partial D(t),\\
				a(0) &= a_0 && \text{ in } \Omega_0.
			\end{aligned}
		\end{equation*}
Then $a \in L^\infty_{L^\infty(\Omega)}$ and 
		\begin{align*}
			a(t) &\leq C_1e^{ \norm{\grad \cdot \VO}{\infty}T +  D^{-1}C_2\left(\frac 12\norm{y}{L^\infty_{L^\infty(\Gamma)}}  + 2\norm{j}{\infty}\right)^2T}\max\left(1, \norm{h}{L^\infty_{L^\infty(\partial D)}}, \norm{a_0}{\infty}, \frac 12 C_3 \norm{y}{L^\infty_{L^\infty(\Gamma)}}\right)\\
			&\quad \times (\min(1,D)^{-\frac{1}{2\kappa} - \frac 12} + C_4)
 		\end{align*}
 where $C_1, C_2, C_3$ and $C_4$ are constants independent of all relevant parameters.
	\end{lem}
\begin{proof}
With the transformation $A:=ae^{-\lambda t}$ for a $\lambda$ to be fixed later, we have
\begin{align*}
\mdd A + A (\grad \cdot \Vp + \lambda) - D \Delta A + \grad \cdot (\JO A) &= 0 &&\text{ in } \Omega(t),\\
D\grad A \cdot \nu - Aj &= Y &&\text{ on } \Gamma(t),\\
 A &= H &&\text{ on } \partial D(t),\\
A(0) &= a_0&&\text{ in } \Omega_0,
\end{align*}
where $Y(t) := y(t)e^{-\lambda t}$ and $H(t) := he^{-\lambda t}$. Take a constant $k \geq \norm{h}{L^\infty_{L^\infty(\partial D)}}$, then defining (like in \cref{lem:L-infinity-Neumann}) $A_k(t):= (A(t)-k)^+$, we see that
\[A_k(t)|_{\partial D(t)} = (A(t)-k)^+|_{\partial D(t)} \leq \left(H-\norm{h}{L^\infty_{L^\infty(\partial D)}}\right)^+ = 0\]
and hence $A_k$ is a valid test function because $A_k(t) \in H^1_e(\Omega(t))$. Testing the equation for $A$ with $A_k := (A-  k)^+$, we obtain that $A_k$ satisfies exactly the same equality  as in the proof of \cref{lem:L-infinity-Neumann} and the proof carries through if we adjust \eqref{eq:largenessOnk} in the obvious way.
\end{proof}
\begin{lem}\label{lem:linftyBdud}We have 
\begin{align*}
&\norm{u}{L^\infty_{L^\infty(\Omega)}}\\
&\quad \leq  C_1e^{\frac 12 \norm{\grad \cdot \VO}{\infty}T +  \delta_\Omega C_2\left(\frac 12\delta_\Omega^{-1}\delta_{k'}^{-1} \norm{z}{L^\infty_{L^\infty(\Gamma)}}  + 2\norm{j}{\infty}\right)^2T}\max\left(1, \norm{u_D}{L^\infty_{L^\infty(\partial D)}}, \norm{u_0}{\infty}, \frac 12 C_3\delta_\Omega^{-1}\delta_{k'}^{-1} \norm{z}{L^\infty_{L^\infty(\Gamma)}}\right)\\
&\quad\quad\times (\min(1,\delta_\Omega^{-1})^{-\frac{1}{2\kappa} - \frac 12} + C_4).
\end{align*}
\end{lem}
\begin{proof}
Consider the following 
\begin{equation*}
\begin{aligned}
\mdd{a} + a \grad \cdot \mathbf V_p  - \delta_\Omega^{-1} \Delta a + \grad \cdot (\JO a)  &= 0 &\text{in $\Omega(t)$},\\
\delta_\Omega^{-1}\nabla a \cdot \nu &= \delta_\Omega^{-1}\delta_{k'}^{-1} z  + ja&\text{on $\Gamma(t)$},\\
a &= u_D &\text{on $\partial D(t)$},\\
a(0) &= u_0
\end{aligned}
\end{equation*}
By arguing almost identically as in \cref{lem:u_leq_a}, we get that $a \geq u$. 
Using this and the non-negativity of $u$, the claim follows from \cref{lem:L-infinity-Neumannd}.
\end{proof}
We need the following version of \cref{lem:l1boundOng} to bound $\delta_k^{-1}g(u,w)$ in $L^1_{L^1(\Gamma)}$. There, we used the equation for $u$; here we replace it by using the $w$ equation.
\begin{lem}\label{lem:l1boundOngd}
We have
\begin{align*}
\dfrac{1}{\delta_k} \int_0^T \intGt g(u,w) &\leq \dfrac{1}{\delta_{k'}} \|z\|_{L^1_{L^1(\Gamma)}}  + \|w_0\|_{L^2(\Omega)}.
\end{align*}
\end{lem}
\begin{proof}
This follows from
\begin{align*}
\frac{d}{dt}\int_{\Gamma(t)} w  &= \int_{\Gamma(t)} \mdd w + w\dVp \\
&= - \int_{\Gamma(t)}\grad \cdot (\JG w) +  \delta_\Gamma\int_{\Gamma(t)} \slap w + \frac{1}{\delta_{k'}}\int_{\Gamma(t)}z - \frac{1}{\delta_k}\int_{\Gamma(t)}g(u,w)\\
&= \frac{1}{\delta_{k'}}\int_{\Gamma(t)}z - \frac{1}{\delta_k}\int_{\Gamma(t)}g(u,w).
\end{align*}
\end{proof}

The energy estimate for $u$ given in \cref{lem:u_energy} no longer applies because we cannot test the $u$ equation with itself, as the space of test functions is different. We modify it as follows.
\begin{lem}
Let $u_D \in L^\infty_{L^\infty(\partial D)} \cap C^0_{H^{1\slash 2}(\partial D)}$. With
\begin{align*}
L_1 &:= \norm{\dVO}{\infty} + \norm{\dVp}{\infty} + \norm{\grad \cdot \JO}{\infty} + \norm{\JO}{\infty},\\
L_2 &:= (\norm{\dVp}{\infty}+ \norm{\grad \cdot \JO}{\infty})\norm{\hat u}{L^2_{L^2(\Omega)}}^2  +  \norm{\JO}{\infty} \norm{\grad \hat u}{L^2_{L^2(\Omega)}}^2,\\
K &:= \delta_\Omega L_2 + C(\delta_{k'}^{-1} + \delta_\Omega\norm{j}{\infty})^2 \norm{\hat u}{L^2_{L^2(\Omega)}}^2  + \norm{\grad \hat u}{L^2_{L^2(\Omega}}^2 + \delta_\Omega \norm{u_0-\tilde u}{L^2(\Omega_0)}^2, 
\end{align*}
we have
\begin{align*}
\norm{u-\hat u}{L^\infty_{L^2(\Omega)}}^2 &\leq \left(\dfrac{\norm{z}{L^2_{L^2(\Gamma)}}^2}{\delta_\Omega\delta_{k'}} + \frac{K}{\delta_\Omega}\right) e^{\frac{1}{\delta_\Omega}\left(\delta_\Omega L_1 +  2C(\delta_{k'}^{-1} + \delta_\Omega\norm{j}{\infty})^2 \right)T},\\
\frac 12\norm{\grad u}{L^2_{L^2(\Omega)}}^2 + \dfrac{2}{\delta_k}\int_0^t\int_{\Gamma(t)} g(u,w) u &\leq    \dfrac{1}{\delta_{k'}} \norm{z}{L^2_{L^2(\Gamma)}}^2 + \left(\delta_\Omega L_1 +  C(\delta_{k'}^{-1} + \delta_\Omega\norm{j}{\infty})^2 \right)  \norm{u- \hat u}{L^2_{L^2(\Omega)}}^2 \\
&\quad  +  \delta_\Omega L_2 +  C(\delta_{k'}^{-1} + \delta_\Omega\norm{j}{\infty})^2 \norm{\hat u}{L^2_{L^2(\Omega)}}^2 + \norm{\grad \hat u}{L^2_{L^2(\Omega)}}^2 \\
&\quad + \delta_\Omega \norm{u_0-\tilde u}{L^2(\Omega_0)}^2.
\end{align*}
\end{lem}
\begin{proof}

Define $\tilde u \in H^1(\Omega_0)$ as the solution of the harmonic extension problem
\begin{equation*}
\begin{aligned}
\Delta \tilde u &=0 &&\text{in }\Omega_0,\\
\tilde u &=  0&&\text{on }\Gamma_0,\\
\tilde u &=u_D(0) &&\text{on $\partial D_0$}.
\end{aligned}
\end{equation*}
From this, we construct 
\[\hat u(t) := \phi_t \tilde u\]
which satisfies
\begin{equation*}
\begin{aligned}
\md \hat u &= 0&&\text{ in }\Omega(t),\\
\hat u &=  0 &&\text{ on }\Gamma(t),\\
\hat u &=u_D &&\text{ on $\partial D(t)$},
\end{aligned}
\end{equation*}
with the first line because $\hat u$ is the pushforward of a constant-in-time object and the last because trace operators and the pushforward maps commute. 
It follows that $\hat u \in \mathbb{W}(H^1(\partial D), H^1(\partial D)^*)$ because $\phi_{-t} \hat u$ belongs to the associated reference space corresponding to $\partial D_0$.
Test the equation for $u$ with $u-\hat u$ (note that $u(t)-\hat u(t) \in H^1_e(\Omega(t))$):
\begin{align*}
\langle \md u, (u-\hat u) \rangle + \intOt u (u-\hat u)\dVp + \frac{1}{\delta_\Omega}\intOt\grad u \grad (u-\hat u) + \intOt\grad \cdot (\JO  u)(u-\hat u) = \frac{1}{\delta_\Omega}\intGt u\grad u \cdot \nu 
\end{align*}
using that $\hat u = 0$ on $\Gamma(t)$.  This is
\begin{align*}
&\langle \md (u - \hat u), (u-\hat u) \rangle + \intOt (u-\hat u)^2\dVp + \hat u(u-\hat u)\dVp  + \frac{1}{\delta_\Omega}\intOt \grad u \grad (u-\hat u)\\
& + \intOt \grad \cdot (\JO  (u-\hat u))(u-\hat u) + \grad \cdot (\JO  \hat u)(u-\hat u)= \frac{1}{\delta_\Omega}\intGt u\grad u \cdot \nu,
\end{align*}
giving, using the identity \eqref{eq:bulkVeryUsefulIdentity},
\begin{align*}
&\frac 12 \frac{d}{dt}\intOt (u - \hat u)^2 + \frac 12\intOt (u-\hat u)^2\dVO + \intOt \hat u(u-\hat u)\dVp  + \frac{1}{\delta_\Omega}\intOt \grad u \grad (u-\hat u) \\
& + \intOt \grad \cdot (\JO  \hat u)(u-\hat u)= \frac{1}{\delta_\Omega}\intGt u\grad u \cdot \nu   - \frac 12 \intGt u^2 j.
\end{align*}
Multiplying by $2\delta_\Omega$ and 
with the same manipulations as in the proof of \cref{lem:u_energy} to deal with the boundary terms, we can write this as
\begin{align*}
&\delta_\Omega \frac{d}{dt}\intOt (u - \hat u)^2 + \delta_\Omega \intOt (u-\hat u)^2\dVO + 2\delta_\Omega\intOt \hat u(u-\hat u)\dVp  + 2\intOt\grad u \grad (u-\hat u) \\
& + 2\delta_\Omega\intOt \grad \cdot (\JO  \hat u)(u-\hat u)
\leq  \intGt   \dfrac{1}{\delta_{k'}} z^2 +   \theta u^2 - \dfrac{2}{\delta_k} g(u,w) u
\end{align*}
where \[\theta:=\dfrac{1}{\delta_{k'}} + \delta_\Omega\norm{j}{\infty}.\]
Using Young's inequality we can bound
\begin{align*}
2\delta_\Omega\intOt \grad \cdot (\JO  \hat u)(u-\hat u) &= 2\delta_\Omega\intOt \grad \cdot \JO \hat u(u-\hat u) + \JO \cdot \grad \hat u (u-\hat u)\\
&\leq \delta_\Omega \norm{\grad \cdot \JO}{\infty}\intOt |\hat u|^2 + (u-\hat u)^2 + \delta_\Omega \norm{\JO}{\infty}  \intOt |\grad \hat u|^2 + (u-\hat u)^2
\end{align*}
where we have taken $\norm{\JO}{\infty} = \sup_i \norm{[\JO]_i}{\infty}$. Thus the lower order terms can be estimated by
\begin{align*}
&\delta_\Omega(\norm{\dVO}{\infty} + \norm{\dVp}{\infty} + \norm{\grad \cdot \JO}{\infty} + \norm{\JO}{\infty})\intOt (u-\hat u)^2 \\
& + \delta_\Omega (\norm{\dVp}{\infty}+ \norm{\grad \cdot \JO}{\infty})\norm{\hat u}{L^2(\Omega(t))}^2  + \delta_\Omega \norm{\JO}{\infty} \norm{\grad \hat u}{L^2(\Omega(t))}^2.
\end{align*}
Let us define 
\[L_1 :=  \norm{\dVO}{\infty} + \norm{\dVp}{\infty} + \norm{\grad \cdot \JO}{\infty} + \norm{\JO}{\infty} \]
(as above) and
\[\hat L_2(t) := (\norm{\dVp}{\infty}+ \norm{\grad \cdot \JO}{\infty})\norm{\hat u}{L^2(\Omega(t))}^2  +  \norm{\JO}{\infty} \norm{\grad \hat u}{L^2(\Omega(t))}^2.\]
For the gradient term, we can estimate simply by Young's inequality:
\[\intOt \grad u \grad (u- \hat u) = \intOt |\grad u|^2 - \grad u \grad \hat u \geq  \frac 12\intOt|\grad u|^2 -  |\grad \hat u|^2.\]
We have then
\begin{align*}
\delta_\Omega \frac{d}{dt}\intOt (u - \hat u)^2 +   \intOt |\grad u|^2  &\leq   \intGt \dfrac{1}{\delta_{k'}}  z^2 +   \theta u^2 - \dfrac{2}{\delta_k}  g(u,w) u  + \delta_\Omega L_1\intOt (u-\hat u)^2 \\
&\quad + \delta_\Omega \hat L_2(t) + \frac 12\norm{\grad \hat u}{L^2(\Omega(t))}^2.
\end{align*}
Now by the interpolated trace theorem we have, for $\eps >0$,  
\begin{align*}
\theta \int_{\Gamma(s)}  u^2 &\leq    \int_{\Omega(s)}\frac C2\theta^2 u^2 +   \frac 12|\nabla u|^2 \leq   C\theta^2  \int_{\Omega(s)} (u-\hat u)^2 + |\hat u|^2 +   \int_{\Omega(s)} \frac 12 |\nabla u|^2.
\end{align*}
Inserting this above, defining $ L_2(t) := \int_0^t  \hat L_2(s)$
\begin{align*}
&\delta_\Omega \intOt (u - \hat u)^2 +  \frac 12\int_0^t\intOt |\grad u|^2 + \dfrac{2}{\delta_k}\int_0^t\int_{\Gamma(t)} g(u,w) u   \leq    \dfrac{1}{\delta_{k'}} \norm{z}{L^2_{L^2(\Gamma)}}^2 + \left(\delta_\Omega L_1 +  C\theta^2 \right)\int_0^t\intOt (u-\hat u)^2\\
&\quad   + \hat \delta_\Omega L_2(t) +  C\theta^2 \norm{\hat u}{L^2_{L^2(\Omega)}}^2  + \norm{\grad \hat u}{L^2_{L^2(\Omega}}^2 + \delta_\Omega \norm{u_0-\tilde u}{L^2(\Omega_0)}^2.
\end{align*}
Thus applying Gronwall's inequality, we end up with, denoting $L_2 = L_2(T)$ and setting
\[K:= \delta_\Omega  L_2 + C\theta^2 \norm{\hat u}{L^2_{L^2(\Omega)}}^2  +  \norm{\grad \hat u}{L^2_{L^2(\Omega}}^2 + \delta_\Omega \norm{u_0-\tilde u}{L^2(\Omega_0)}^2,\]
the estimate
\[\norm{u-\hat u}{L^\infty_{L^2(\Omega)}}^2 \leq \left(\dfrac{\norm{z}{L^2_{L^2(\Gamma)}}^2}{\delta_\Omega\delta_{k'}} + \frac{K}{\delta_\Omega}\right) e^{\frac{1}{\delta_\Omega}\left(\delta_\Omega L_1 +  2C\theta^2 \right)T}.\]
\end{proof}
Note that we no longer get a uniform bound on $u$ in $L^\infty_{L^2(\Omega)}$ (fortunately \cref{lem:linftyBdud} takes care of this) because of the presence of the $\grad \hat u$ term in the definition of $K$, but we do get one on its gradient. 
In summary, we obtain 
\begin{align*}
u_k \rightharpoonup u \,\, \text{ in } L^2_{H^1(\Omega)} \cap L^\infty_{L^\infty(\Omega)}, \quad w_k \weakstar w \,\, \text{ in } L^\infty_{L^\infty(\Gamma)}, \quad z_k \weakstar z \,\, \text{ in } L^\infty_{L^\infty(\Gamma)}.
\end{align*}
By continuity of the trace operator we have 
\begin{align*}
    u_k \rightharpoonup u \,\, \text{ in } L^2_{H^{1\slash 2}(\Gamma)} \quad\text{and}\quad     u_k \rightharpoonup u \,\, \text{ in } L^2_{H^{1\slash 2}(\partial D)},
\end{align*}
the latter of which immediately gives $u|_{\partial D} = u_D$. 

We lack a uniform estimate on $(1\slash \delta_k)\grad u_k$ (akin to \eqref{eq:gradUVanishesEstimate}) and therefore we do not obtain uniform estimates on the difference quotients of $u_k$ via \cref{lem:difference_quotients_u}. However, if we assume that $\JG \equiv 0$, as explained in \cref{rem:ifKZerowDiffQuotientsEstimate}, we do get an estimate on the difference quotients of $w_k$, and so 
\begin{align*}
    w_k &\to w \,\, \text{ in } L^2_{H^{-1/2}(\Gamma)}.
\end{align*}
We also retain the estimate \eqref{eq:doubleDeltasGEstimate} on the gradients of $w_k$ and $z_k$. To pass to the limit, we begin with taking $\eta \in \mathcal{V}_e$ with $\eta(T)=0$ and recalling the weak formulation relating $u_k$ and $z_k$ from above:
\begin{align*}
-\delta_\Omega\int_{\Omega_0}u_0\eta(0)   &-\delta_\Omega\int_0^T \int_{\Omega(t)} u_k \mdd\eta + \int_0^T\int_{\Omega(t)} \nabla u_k\cdot \nabla \eta + \delta_\Omega\int_0^T\int_{\Omega(t)} \grad \cdot (\JO u_k)\eta - \delta_\Omega\int_0^T\int_{\Gamma(t)}ju_k\eta \\
&-\int_{\Gamma_0}z_0\eta(0)  - \int_0^T \intGt z_k \mdd \eta + \delta_\Omega\int_0^T\intGt \tgrad z_k\cdot \tgrad \eta - \int_0^T\int_{\Gamma(t)} z_k \JG \cdot \sgrad \eta = 0.
\end{align*}
Sending $\delta_\Omega \to 0$ yields
\begin{align*}
\int_0^T\int_{\Omega(t)} \nabla u \cdot \nabla \eta &-\int_{\Gamma_0}z_0\eta(0)  - \int_0^T \intGt z \mdd \eta  - \int_0^T\int_{\Gamma(t)} z \JG \cdot \sgrad \eta = 0.
\end{align*}
Recall that we actually assumed that $\JG \equiv 0$; we now argue to remove this assumption. Take an arbitrary parametrisation velocity $\widehat{\Vp}$ satisfying the same properties as $\Vp$ in \cref{sec:notation_problem_setting}, with corresponding parametrised material derivative $\widehat{\mdd}$. We have, by adding and subtracting the same terms to the above equality, 
\begin{align*}
\int_0^T\int_{\Omega(t)} \nabla u \cdot \nabla \eta &-\int_{\Gamma_0}z_0\eta(0)  -\int_0^T\intGt z\widehat{\mdd}\eta - \int_0^T\intGt z(\VG-\widehat{\Vp})\cdot \sgrad \eta\\
&+ \int_0^T\intGt z\widehat{\mdd}\eta + \int_0^T\intGt z(\VG-\widehat{\Vp})\cdot \sgrad \eta   - \int_0^T \intGt z \mdd \eta= 0.
\end{align*}
Note that we can write $\widehat{\mdd}\eta = \partial^\circ \eta + \sgrad \eta \cdot \widehat{\Vp}$ and similarly $\mdd \eta = \partial^\circ \eta + \sgrad \eta \cdot \VG$, so the terms on the left-hand side of the last line above become
\begin{align*}
&\int_0^T\intGt z\widehat{\mdd}\eta + \int_0^T\intGt z(\VG-\widehat{\Vp})\cdot \sgrad \eta   - \int_0^T \intGt z \mdd \eta  \\
&=\int_0^T\intGt z(\widehat{\mdd}\eta- \mdd \eta) + \int_0^T\intGt z(\VG-\widehat{\Vp})\cdot \sgrad \eta\\
&=\int_0^T\intGt z\sgrad\eta \cdot (\widehat{\Vp}-\VG) + \int_0^T\intGt z(\VG-\widehat{\Vp})\cdot \sgrad \eta\\
&= 0.
\end{align*}
Setting $\widehat{\JG} := \VG-\widehat{\Vp}$, we have shown that
\begin{align*}
\int_0^T\int_{\Omega(t)} \nabla u \cdot \nabla \eta &-\int_{\Gamma_0}z_0\eta(0)  -\int_0^T\intGt z\widehat{\mdd}\eta - \int_0^T\intGt z\widehat{\JG}\cdot \sgrad \eta = 0
\end{align*}
for an arbitrary parametrised velocity $\widehat{\Vp}$. It remains to show that $u, w, z$ belong to the corresponding function spaces defined using $\widehat{\Vp}$ and the associated maps $\widehat{\Phi}$ and $\hat \phi$. Writing $\hat\phi_{-t}z = z \circ \widehat{\Phi_t} = z \circ \Phi_t \circ \Phi^t \circ \widehat{\Phi_t} = (\phi_{-t} z)\circ \Phi^t \circ \widehat{\Phi_t}$ and using the fact that the composition of diffeomorphisms is a diffeomorphism,
\[\int_0^T\int_{\Gamma_0}|\hat\phi_{-t}z(x)|^2 = \int_0^T\int_{\Gamma_0}|(\phi_{-t} z) \circ \Phi^t \circ \widehat{\Phi_t}(x)|^2 = \int_0^T\int_{\Gamma_0}|(\phi_{-t} z)(y)|^2  \hat \phi_{-t}(J^t) \hat J_t
\leq C\norm{\phi_{-(\cdot)}z}{L^2(0,T;L^2(\Gamma_0))}\]
and hence $\hat \phi_{-(\cdot)}z \in L^2(0,T;L^2(\Gamma_0))$, giving, by compatibility, that $z \in \widehat{L^2_{L^2(\Gamma)}}.$ The other quantities can be tackled by a similar argument. \subsubsection{Complementarity condition}
The complementarity condition follows by the assumption \eqref{ass:new_on_g_3} and \cref{lem:l1boundOngd} (which implies that $g(u_k, w_k) \to 0$ in $L^1_{L^1(\Gamma)}$. Thus we have shown existence for \eqref{eq:systemFourDeltasLimitD}. 

\section{Numerical experiments}\label{sec:numerical}
We now present some numerical simulations of \eqref{eq:model_without_Vp} that support the theoretical results of the previous sections and illustrate a robust numerical method for the simulation of coupled bulk-surface systems of equations on evolving domains. Our approach is based on a parametric finite element method on moving triangulations  that approximate the evolving domain. We employ a piecewise linear coupled bulk-surface finite element method for the approximation. The method is based on the coupled bulk-surface finite element method with an evolving surface finite element method for the approximation of the surface PDEs and an ALE finite element method for the bulk equation c.f.,  \cite{Dziuk2013,elliott2013finite,elliott2014error,elliott2021unified} for details on the design and analysis of the numerical methods.

\subsection{Domain evolution and discretisation}\label{sec:evo_and_disc}
For all the simulations of this section, we assume the same continuous geometry and use the same discretisation of the geometry.  We assume $D(t)$ is a sphere of radius 2 centred at the origin and  that the outer boundary of $\O$ satisfies $\partial D(t)=\partial D(0)$, i.e., the outer boundary is stationary (the method employed below can be straightforwardly extended to the case of a moving outer boundary).
 Furthermore, for all the simulations we assume the interior surface $\Gamma(t)$ that  corresponds to the cell membrane is given by the zero level set of the function
\[
\phi(\vec x,t)= \left(x_1+\tanh(5t)\left(0.7-x_2^2\right)\right)^2+x_2^2+x_3^2-1,
\]
then 
\begin{equation}\label{eqn:normal_velocity_sim}
\nu(\vec x,t)=-\frac{\nabla \phi(\vec x,t)}{\vert\nabla\phi(\vec x,t)\vert}\quad\text{ and  }V(\vec x,t)=-\frac{\phi_t(\vec x,t)}{\vert\nabla\phi(\vec x,t)\vert}
\end{equation}
define the normal to $\G(t)$ (pointing out of $\O(t)$) and the normal velocity.
For simplicity we assume the material velocity of $\Gamma$, $\VG$=$\mathbf{V}$, i.e., it has no tangential component.  For the material velocity $\VO$ of $\O$, we will either assume $\VO =\vec 0$ or that $\VO=E(\vec V_\G)$ where $E(\VG)$ is the harmonic extension of the velocity of the boundary into the interior of $\O$, where we recall that the velocity of the outer boundary is zero.  

We partition the time interval $[0,T]$  as
$0=t_0<t_1<\dots<t_N=T$, with uniform (for simplicity) timestep $\tau=t_{n+1}-t_n$, $n=0,1,\dots,N-1$. We define initial computational domains $\O_h^0$ and $\G_h^0$ by requiring that $\O_h^0$ is a polyhedral approximation to $\O(0)$ and we set $\G_h^0=\partial\O_h^0 \setminus \partial D_h$, i.e., $\G_h^0$ is the interior boundary of the polyhedral domain $\O_h^0$. We assume that $\O_h^0$ is the union of tetrahedra and hence the faces of $\G_h^0$ are triangles.  We define $\T_h^0$ to be a triangulation of $\O_h^0$ consisting of  closed simplices. Furthermore, we assume the triangulation is such that for every $k\in\T_h^0$, $k\cap\G_h^0$ consists of at most one face of $k$. We construct an initial triangulation  $\O_h^0$ which consists of 19642 tetrahedra with $4135$ vertices with a higher resolution in the neighbourhood of $\G_h^0$; the induced  surface triangulation of $\G_h^0$  consists of $1948$  triangles with $N_\G=976$ vertices. As we shall consider Dirichlet conditions on the outer boundary, we denote by $N_\O$ the number of vertices of $\T_h^0$ that do not lie on the outer boundary $\partial D_h^0$.
We  evolve the nodes of the surface triangulation with the normal velocity defined in \eqref{eqn:normal_velocity_sim}. The nodes of the bulk triangulation are evolved with the harmonic extension of this velocity into the interior.  \cref{fig:surface_mesh} shows the resulting surface triangulation at a series of time points shaded by $V=\vec V_\G\cdot\nu$.
\begin{figure}[h]
\centering
\includegraphics[width=0.13\linewidth]{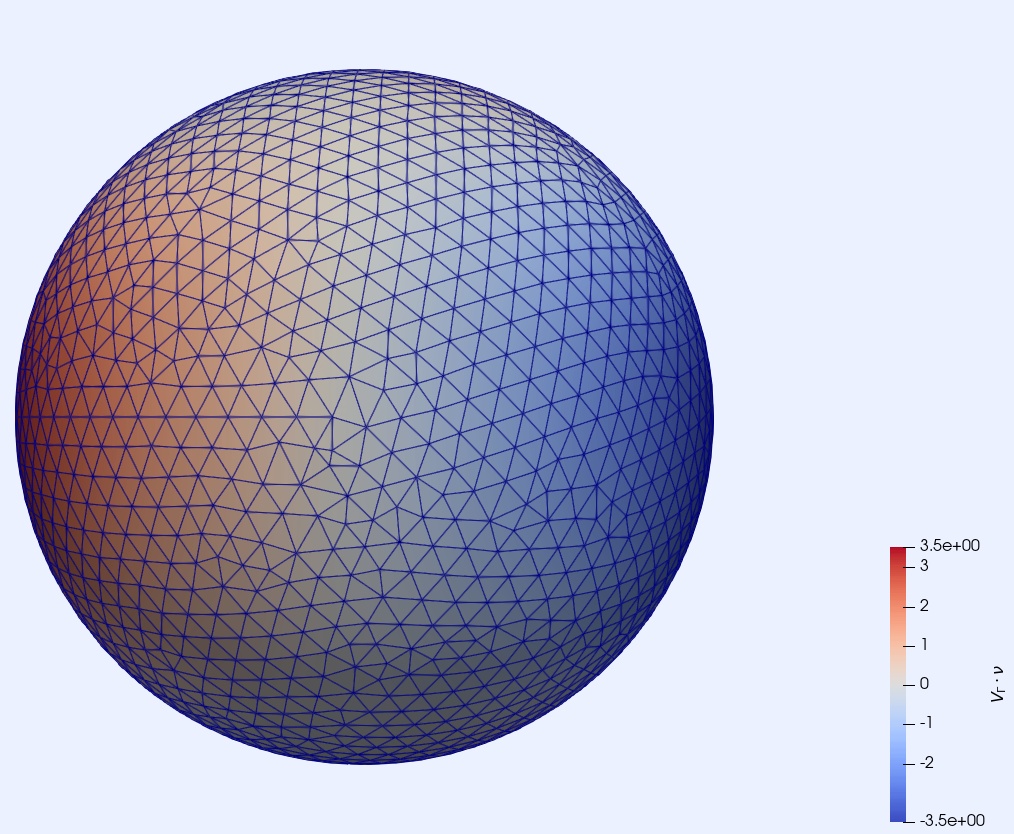}
\includegraphics[width=0.13\linewidth]{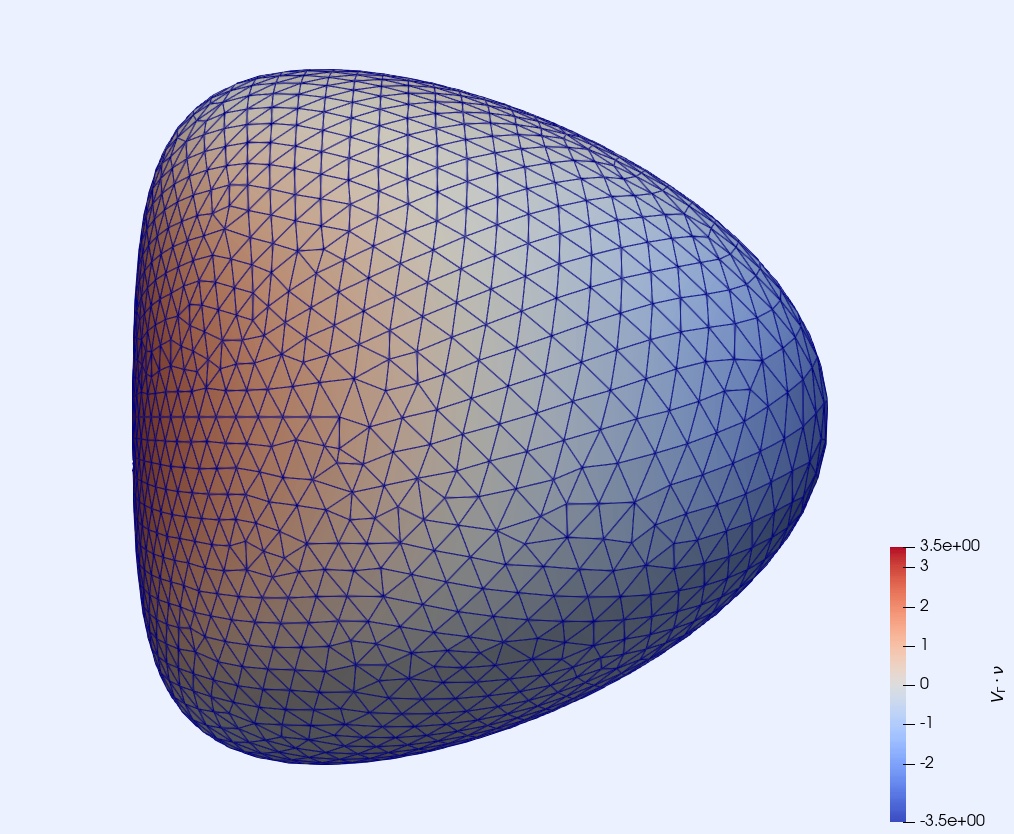}
\includegraphics[width=0.13\linewidth]{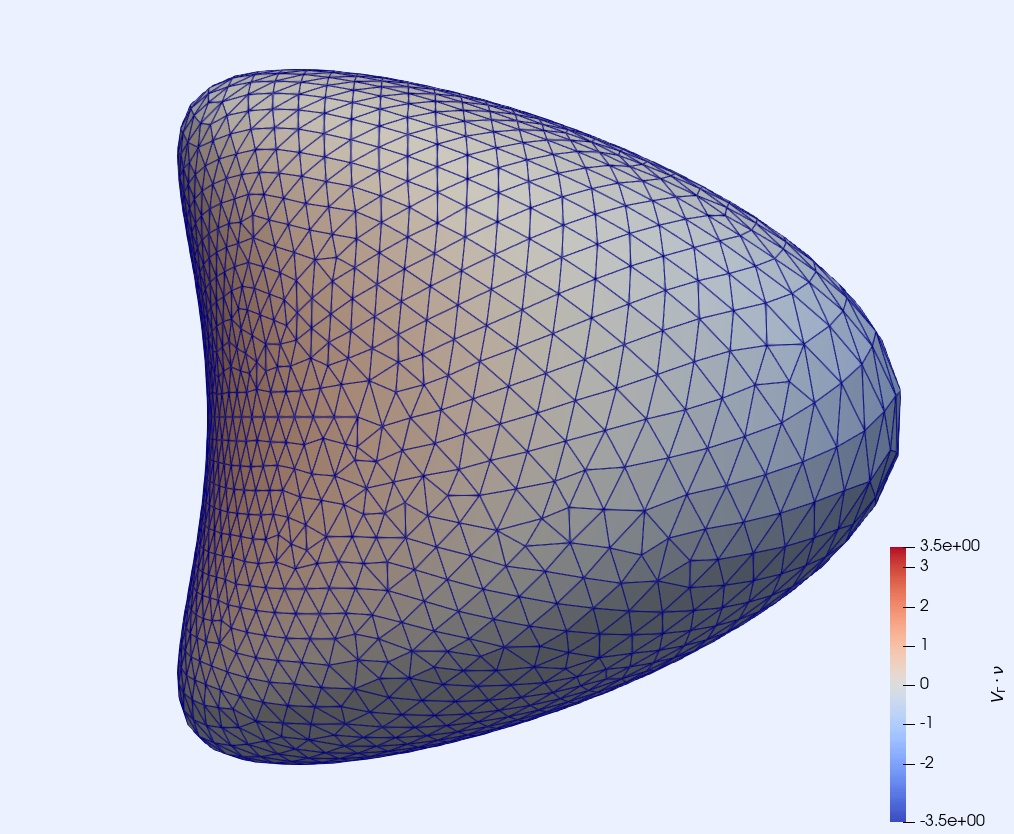}
\includegraphics[width=0.13\linewidth]{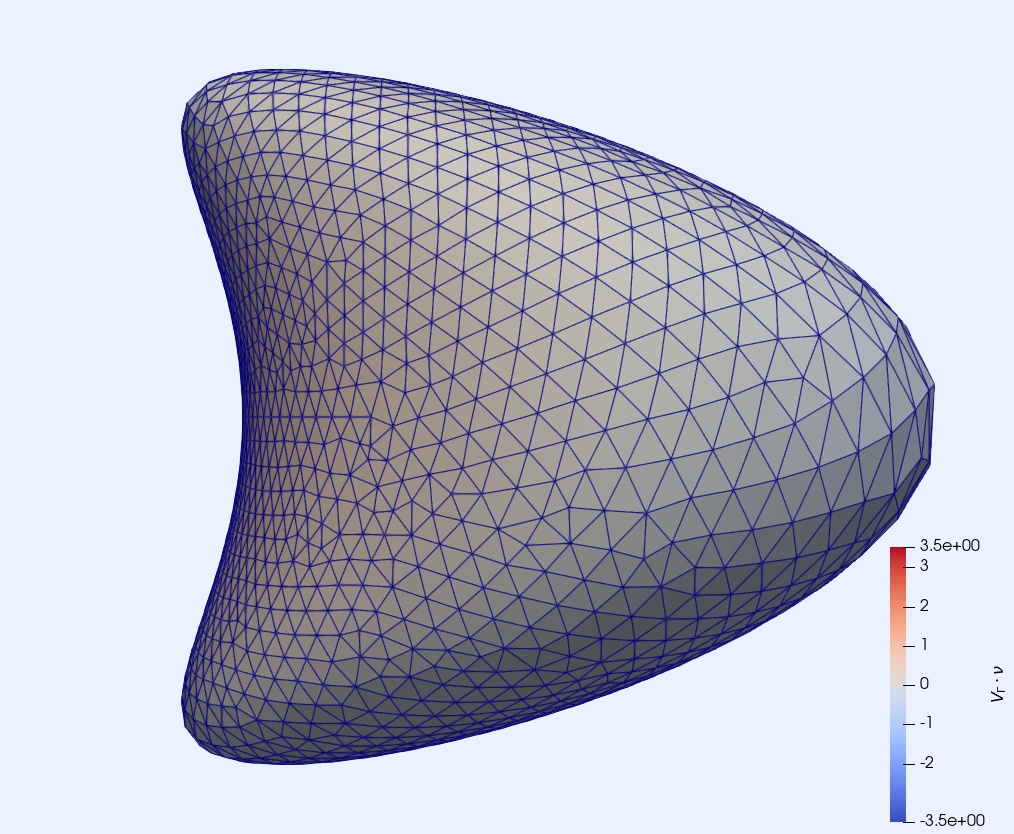}
\includegraphics[width=0.13\linewidth]{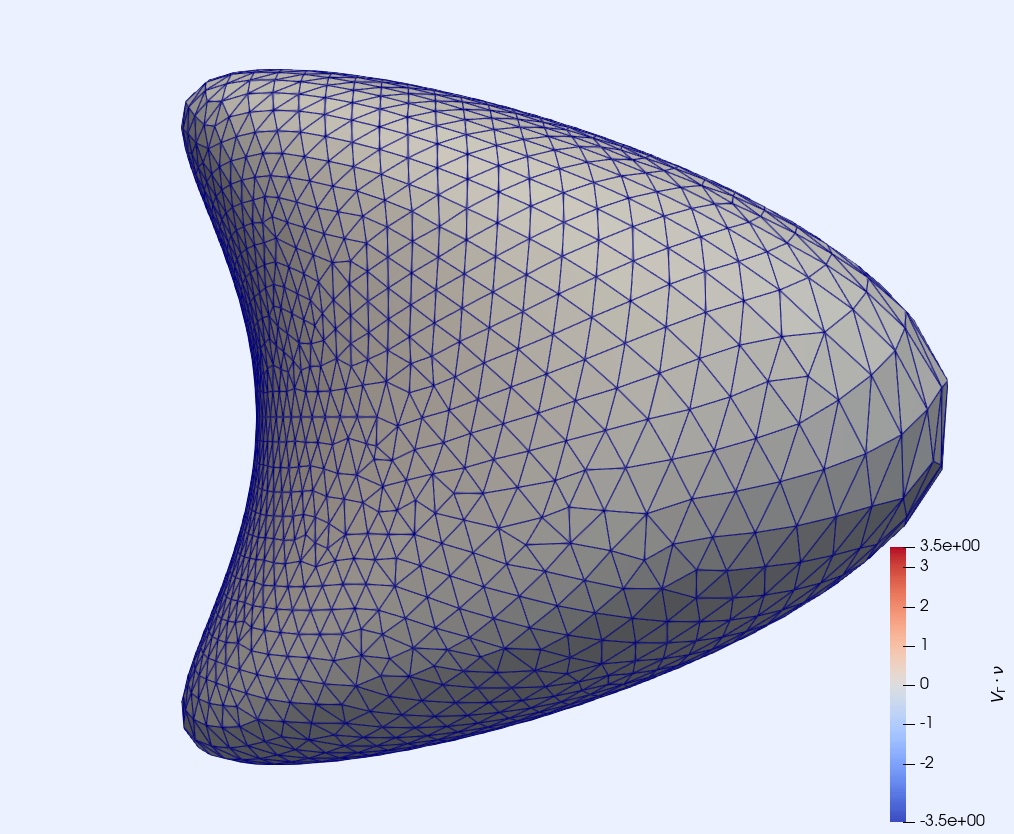}
\includegraphics[width=0.13\linewidth]{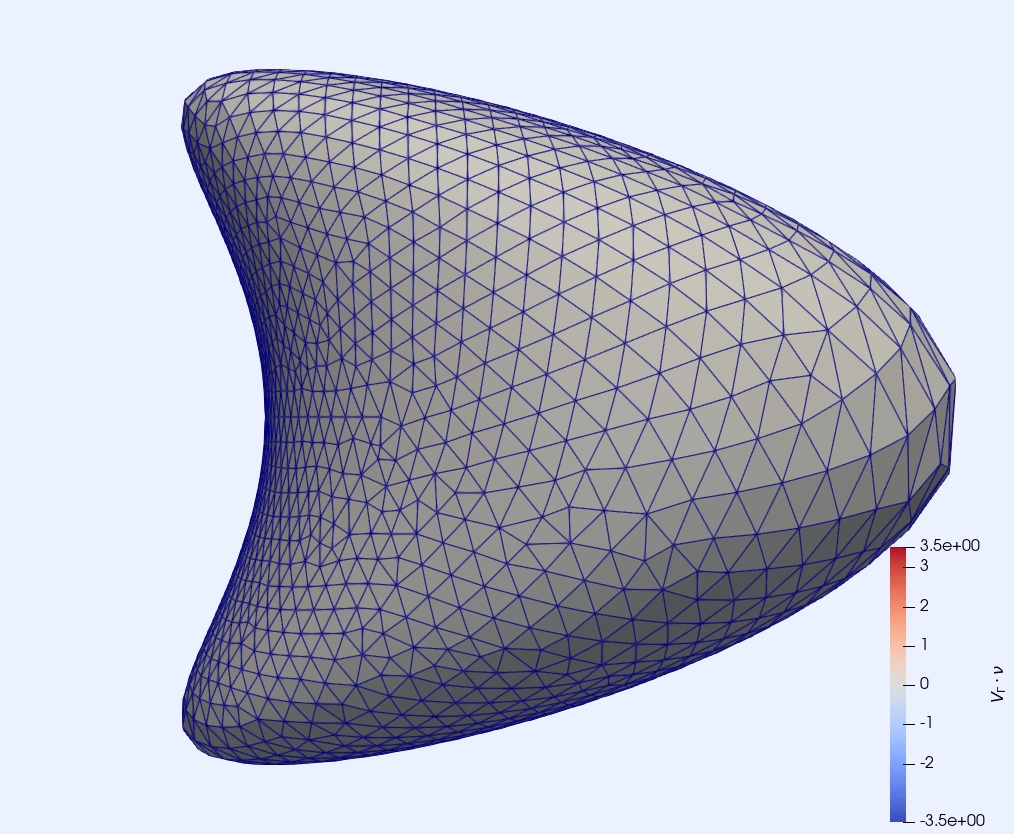}
\includegraphics[width=0.13\linewidth]{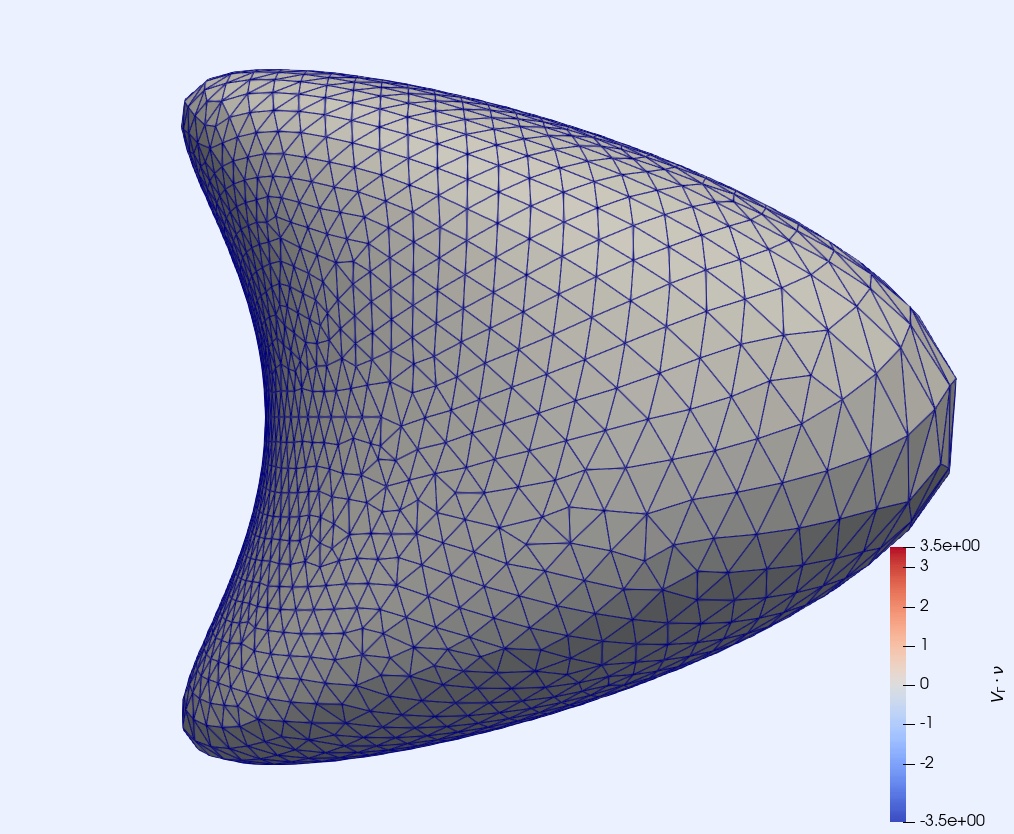}
\caption{Snapshots of the surface triangulation shaded by $V=\vec V_\G\cdot\nu$ at times $0,0.1,0.2,0.3,0.4, 0.8$ and $1.0$ reading from left to right.}
\label{fig:surface_mesh}
\end{figure}
To define the triangulations $\T_h^n$ of $\O$ at each time, $t_n, n\in[1,N]$,  we compute the harmonic extension of the velocity of $\G_h^n$ into $\O_h^n$ and use this velocity to evolve the nodes of the triangulation of $\O_h^n$. 

\subsection{Evolving coupled bulk-surface finite element method}
For the approximation we define bulk and surface finite element spaces
\begin{subequations}\label{eq:FEM_spaces}
\begin{align}
\V_h^n&:=\left\lbrace \Phi_h\in C(\O_h^n): \Phi_h|\partial_{D_h}=0, \Phi_h\vert_{k}\in \P_1(k)\text{ for each }k\in\T_h^n\right\rbrace,\\ 
\V_{h,U_D}^n&:=\left\lbrace \Phi_h\in C(\O_h^n): \Phi_h|\partial_{D_h}=U_D, \Phi_h\vert_{k}\in \P_1(k)\text{ for each }k\in\T_h^n\right\rbrace,\\ 
\V_{\G,h}^n&:=\left\lbrace \Phi_h\in C(\G_h^n): \Phi_h\vert_{s}\in \P_1(s)\text{ for all }k\in\T_h^n\text{ with }s=k\cup\G_h^n\neq \emptyset\right\rbrace.
\end{align}
\end{subequations}
For $n\in[0,N]$ we denote by $\chi^n_{j}, \chi^n_{\G,j_\G}$, $j=1,\dots,N_\O$,  $j_\G=1\dots,N_\G$ the nodal basis of $\V_h^n$ and $\V_{\G,h}^n$ respectively.

The numerical scheme we employ to approximate the solution of \eqref{eq:newModel} reads as follows:  for $n=1,\dots,N$,  given $(U_h^{n-1}, W_h^{n-1}, Z_h^{n-1})\in\left(\V_{h,U_D}^{n-1}\times \left(\V_{\G,h}^{n-1}\right)^2\right)$  find $(U_h^{n}, W_h^{n}, Z_h^{n})\in\left(\V_{h,U_D}^{n}\times \left(\V_{\G,h}^{n}\right)^2\right)$ such that for $j=1\dots,N_\O$ and for $ j_\G=1\dots,N_\G$,
\begin{subequations}\label{eq:FEM_scheme}
\begin{align}
\nonumber &\int_{\O_h^n}\delo\left( U_h^n\chi_j^n-\tau U_h^n\vec J_{\O,h}^n\cdot\nabla\chi_j^n\right) +\tau\nabla U_h^n\cdot\nabla\chi_j^ndx+\delo\tau\int_{\G_h^n} J_h^nU_h^n \chi_j^nds\\
&\quad= \delo\int_{\O_h^{n-1}}U_h^{n-1}\chi_j^{n-1}dx+\tau\int_{\G_h^{n-1}} \frac{1}{\delz}Z_h^{n-1} \chi_j^{n-1} -\frac{1}{\delk}g(U_h^{n-1},W_h^{n-1}) \chi_j^{n-1}ds\\
&\int_{\G_h^n} W_h^n\chi_{\G,j}^n+\tau\delg\nabla_{\G}W_h^n\cdot\nabla_\G\chi_{\G,j}^nds=\int_{\G_h^{n-1}}\left(W_h^{n-1}+\tau\left(\frac{1}{\delz}Z_h^{n-1} \chi_{\G,j}^{n-1} -\frac{1}{\delk}g(U_h^{n-1},W_h^{n-1})\right)\right)\chi_{\G,j}^{n-1}ds\\
&\int_{\G_h^n} Z_h^n\chi_{\G,j}^n+\tau\delgp\nabla_{\G}Z_h^n\cdot\nabla_\G\chi_{\G,j}^nds=\int_{\G_h^{n-1}}\left(Z_h^{n-1}+\tau\left(\frac{1}{\delk}g(U_h^{n-1},W_h^{n-1})-\frac{1}{\delz}Z_h^{n-1} \chi_{\G,j}^{n-1}\right)\right)\chi_{\G,j}^{n-1}ds
\end{align}
\end{subequations}
where $\vec J_{\O,h}^n$ is an approximation to $\JO$ and $J_h^n$ an approximation to $\JO\vert_{\G}\cdot\nu$. For the approximation of the initial condition we take the interpolant of the initial data into the respective finite element space.

\subsection{Simulations  approximating the $\delo=\delk=\delz^{-1}=\delg=\delgp\to 0$ limit.} \label{sec:evo_hetero}

For the simulation results reported on in this subsection, we take $\delo=\delk=\delz^{-1}=\delg=\delgp=0.01$ and hence the results can be interpreted  as an approximation of the limiting problem stated in \cref{sec:delo_delk_delkp_delg_limit}. We set $g(u,w)=uw$ to contrast the results we present here in the evolving domain setting with those of \cite{ERV} for fixed domains.  We take constant initial and Dirichlet boundary data for $u$ with $u_0=u_D=1$ and constant initial conditions for $w$ and $z$ with $w^0=1$ and $z^0=0$.  We consider the domain evolution described in \cref{sec:evo_and_disc}  with $\VG$ as defined therein and $\VO=\vec 0$. We take $\vec J_{\O,h}^n\in\V_h^n$ and define its nodal values such that $J_{\O,h}^n(\vec X^n_j)=-(\vec X^n-\vec X_j^{n-1})/\tau, j=1,\dots,N_\O$. For the windshield effect term we take $J_h^n=I_{\G,h}^n\left(-\frac{\phi_t(\vec x,t)}{\vert\nabla\phi(\vec x,t)\vert}\right)$ with $I_{\G,h}$ the linear Lagrange interpolant.  

\cref{fig:evo_hetero} shows the results of a simulation. We observe an initial rapid transition in which the trace of $u$ on $\G$ vanishes to satisfy the condition $g(u,w)=0$. The evolution of the domain  causes the concentration of $w$ to become spatially heterogenous and by $t=0.25$ we have a region on $\G$ where $w$ is zero which corresponds to the region where the most protrusion has occurred and hence the concentration of $w$ has been reduced. The region where $w$ is zero grows over time and in this region the trace of $u$ on $\G$ approaches 1 exhibiting the free boundary problem satisfied by the limiting equation. We note that the simulations reported on in this subsection can be interpreted as approximations of the free boundary problem given in \eqref{eqn:surface_stefan}. In \cref{fig:surface_FBP} we show corresponding approximations to the position of the free boundary at various times which we approximate as the level set given by $w=0.1$. 

 \begin{figure}[htb!]
\centering
\subfigure[][{$t=0$}]{
\includegraphics[width=.3\textwidth
]{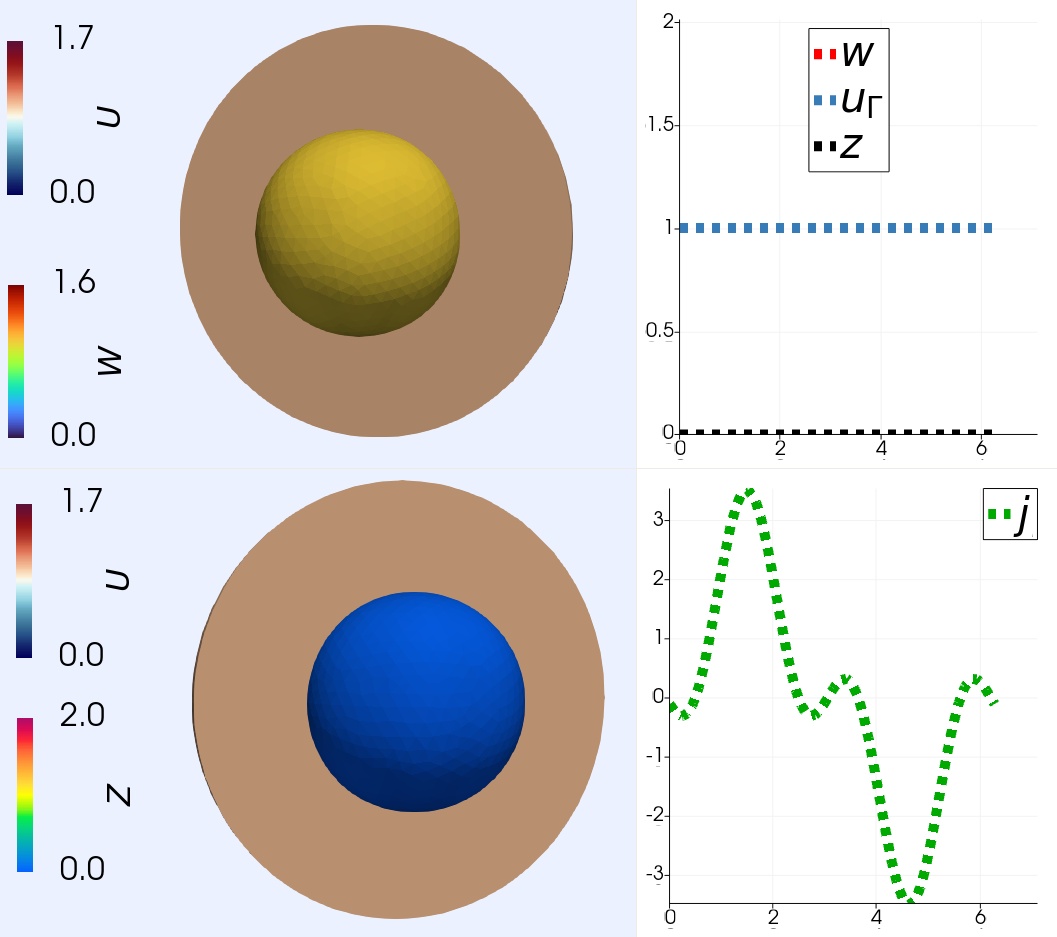}
}
\hskip0.1em
\subfigure[][{$t=0.002$}]{
\includegraphics[width=.3\textwidth
]{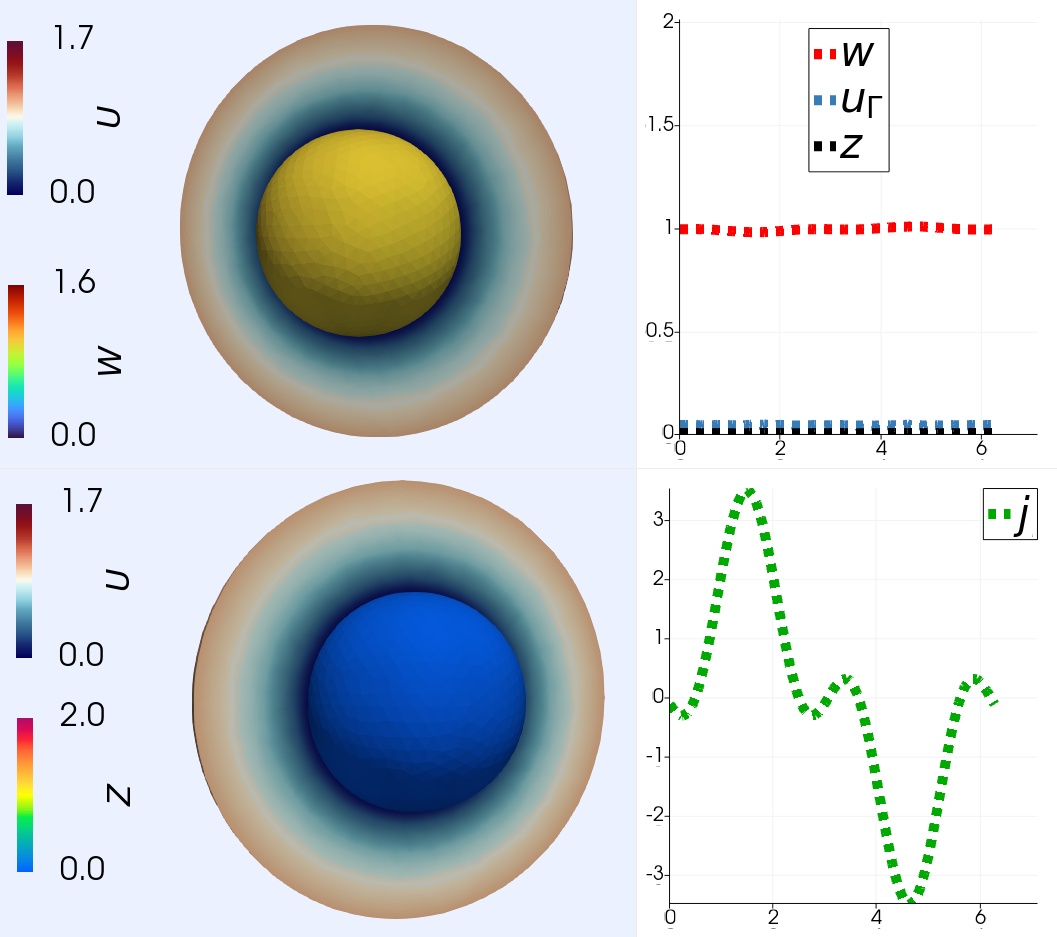}
}
\hskip0.1em
\subfigure[][{$t=0.05$}]{
\includegraphics[width=.3\textwidth
]{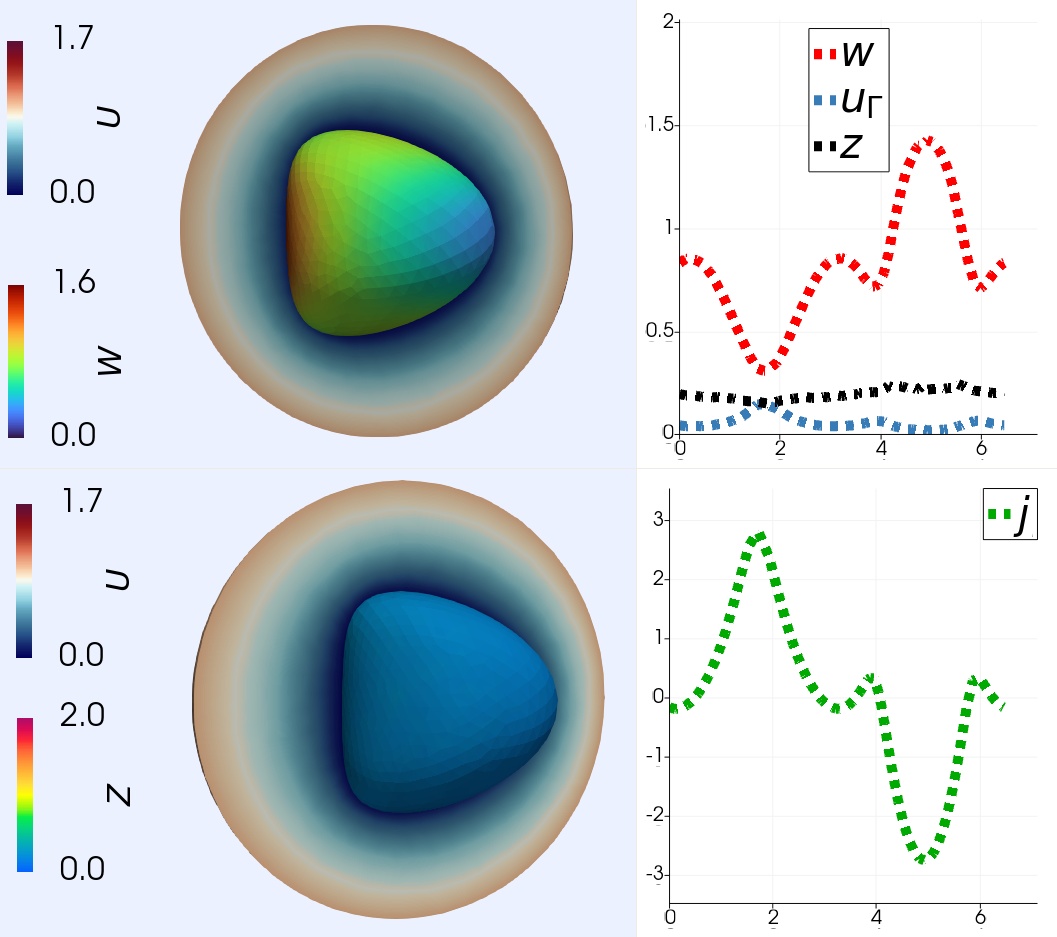}
}
\hskip0.1em
\subfigure[][{$t=0.15$}]{
\includegraphics[width=.3\textwidth
]{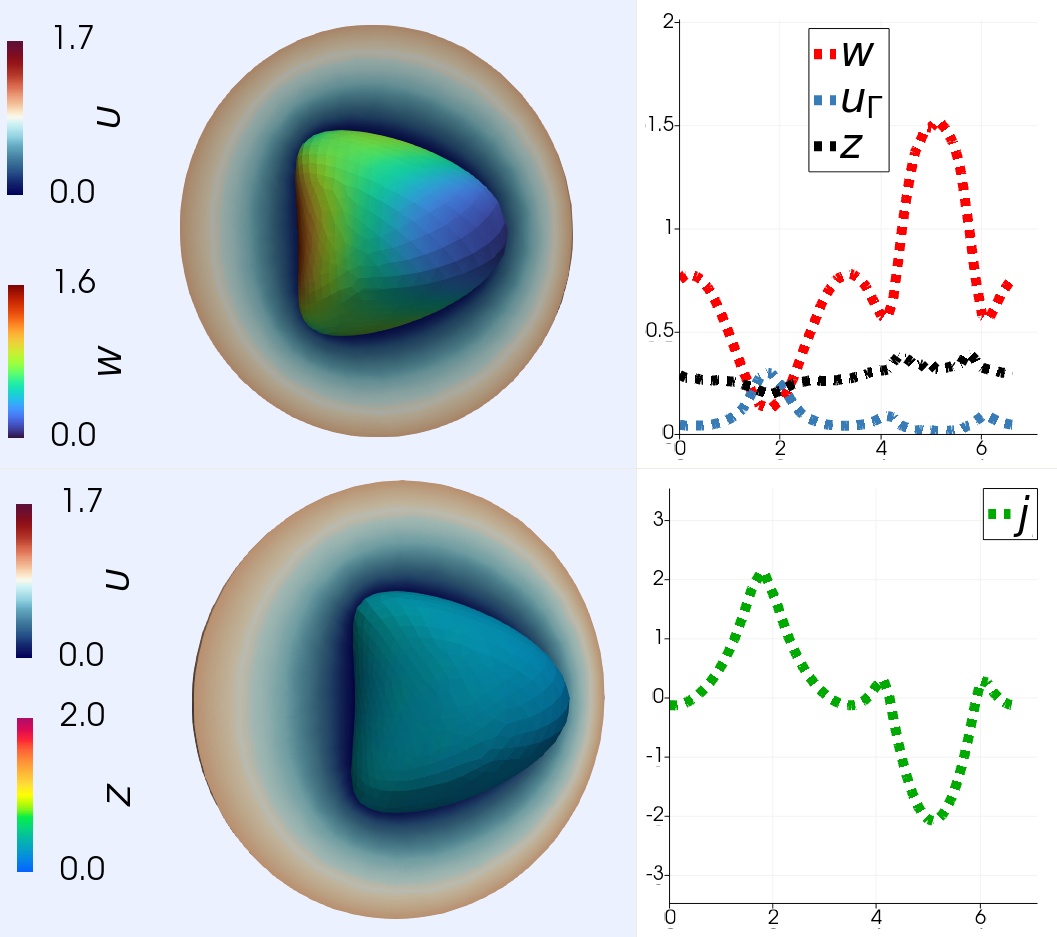}
}
\hskip0.1em
\subfigure[][{$t=0.25$}]{
\includegraphics[width=.3\textwidth
]{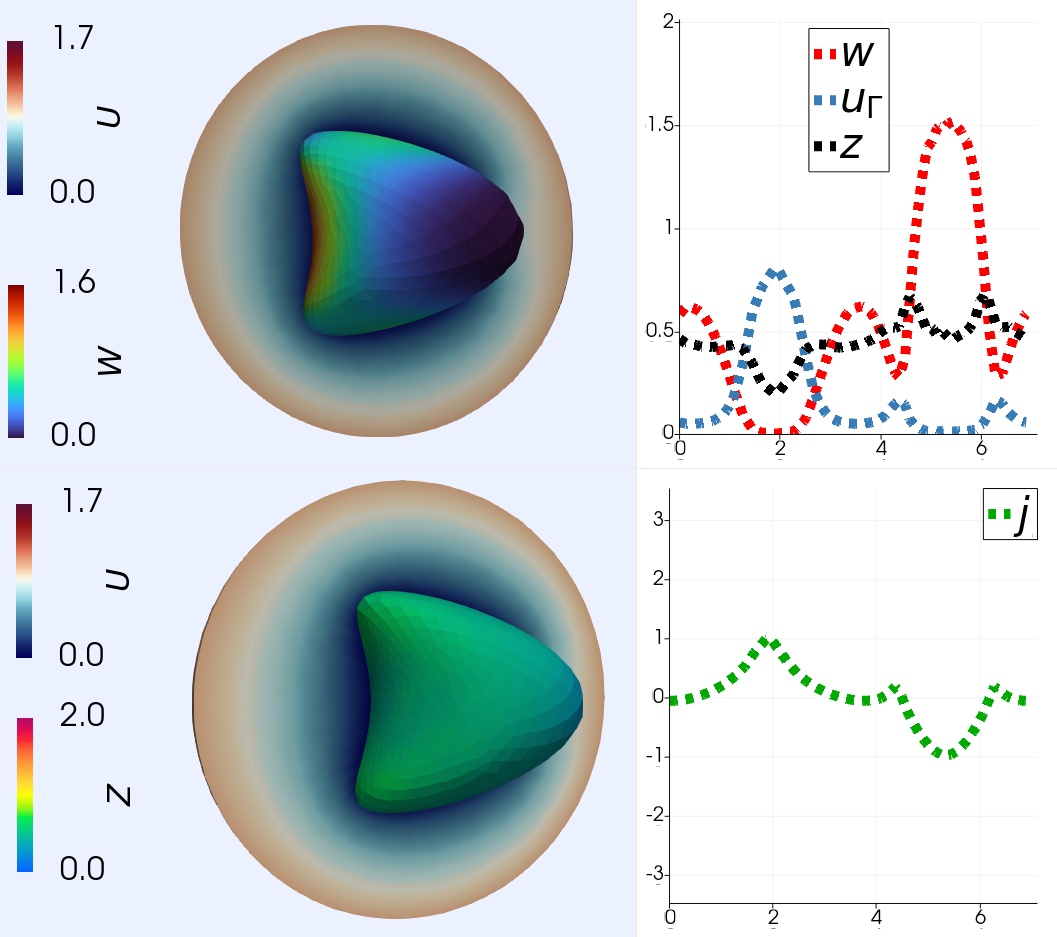}
}
\hskip0.1em
\subfigure[][{$t=0.4$}]{
\includegraphics[width=.3\textwidth
]{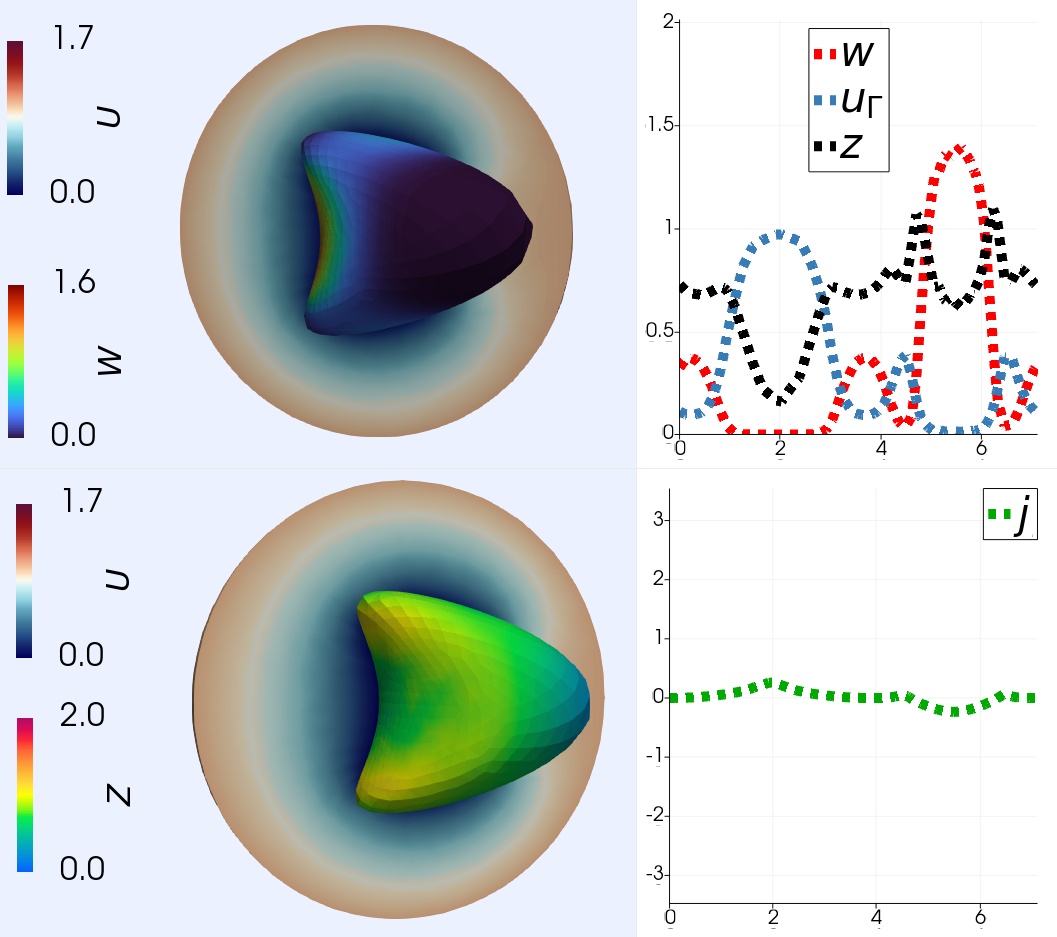}
}
\hskip0.1em
\subfigure[][{$t=0.6$}]{
\includegraphics[width=.3\textwidth
]{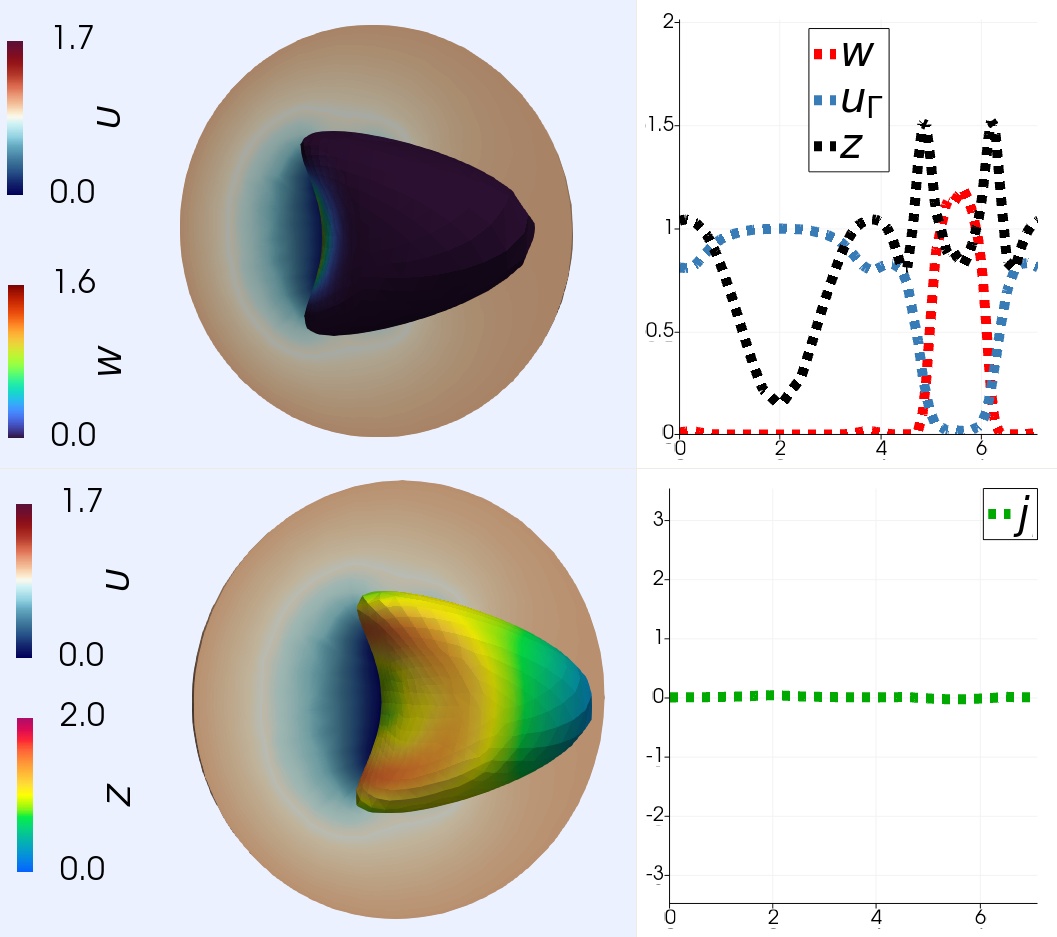}
}
\hskip0.1em
\subfigure[][{$t=0.8$}]{
\includegraphics[width=.3\textwidth
]{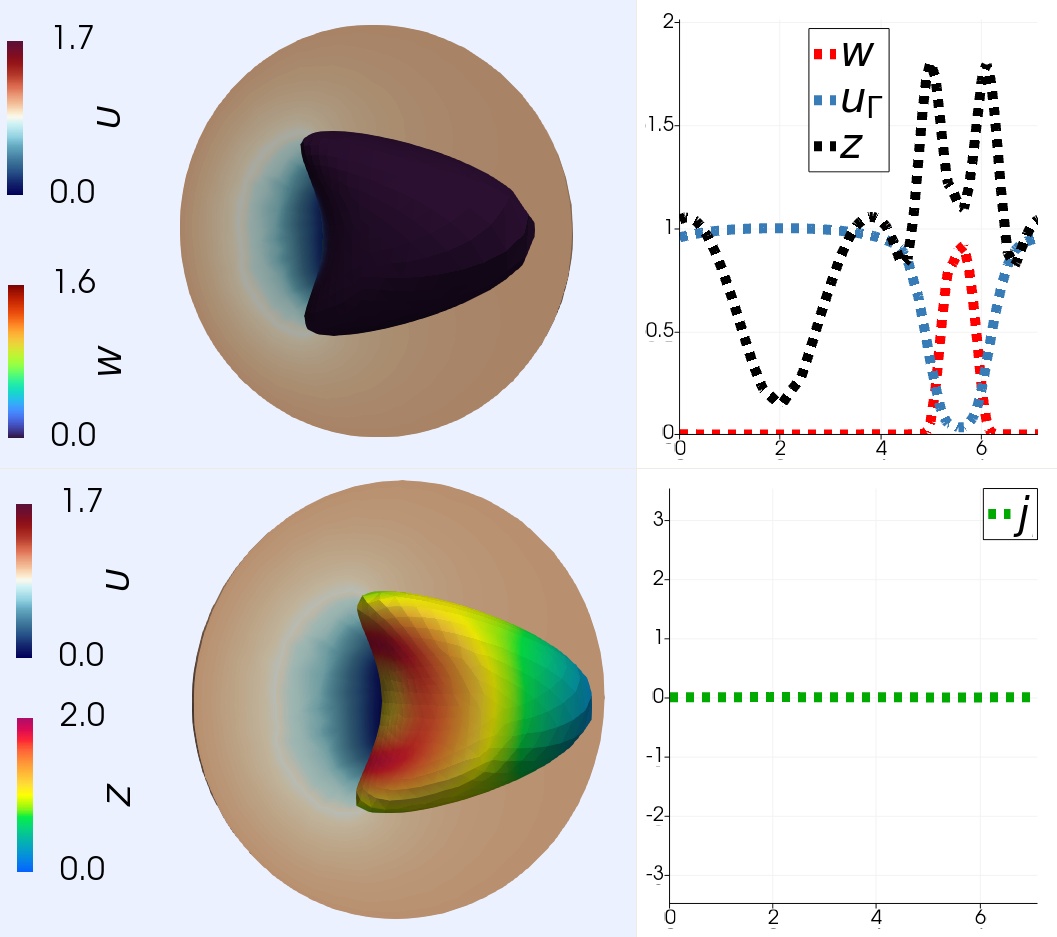}
}
\hskip0.1em
\subfigure[][{$t=1$}]{
\includegraphics[width=.3\textwidth
]{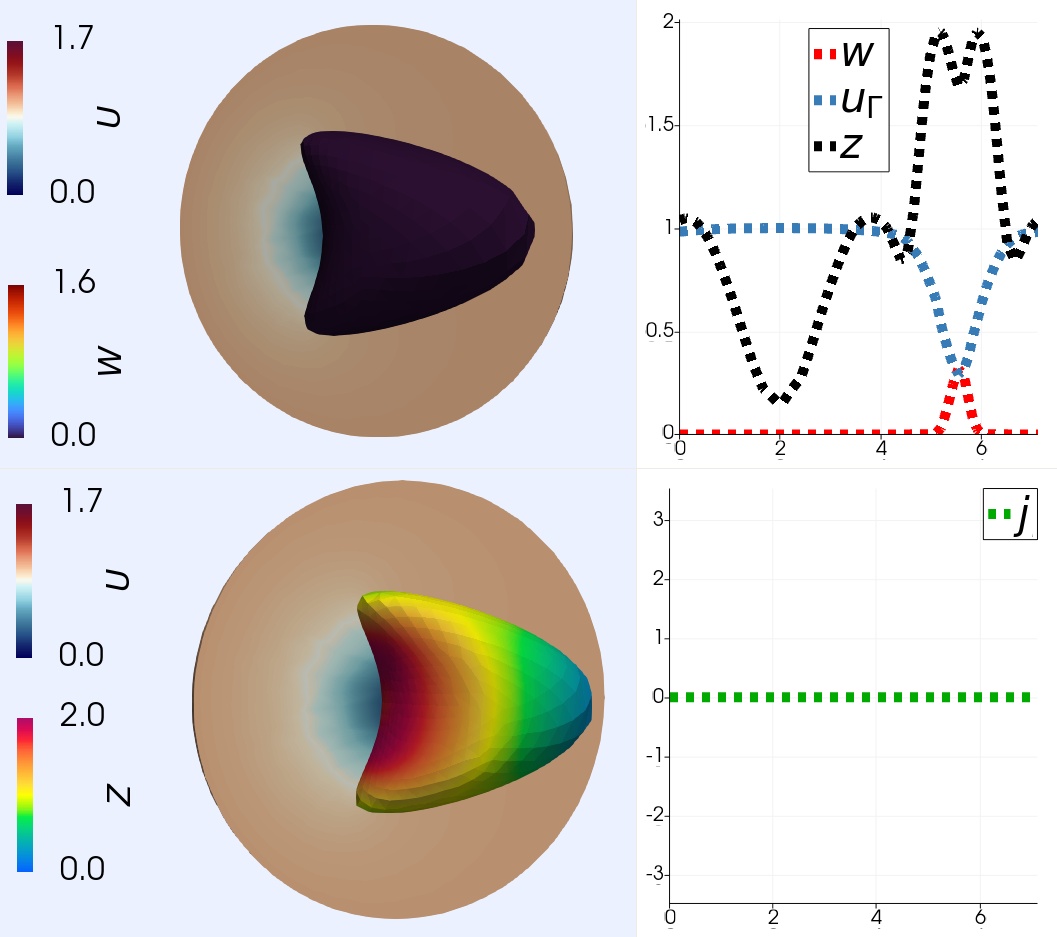}
}
\caption{Simulation results of  \cref{sec:evo_hetero} approximating  the $\delo=\delk=\delz^{-1}=\delg=\delgp\to0$ limit of \cref{sec:delo_delk_delkp_delg_limit} (see \cref{sec:evo_hetero} for details of the parameters).   In each subfigure the top left-hand panel indicates the  surface shaded by the concentration of $w$ and the bulk domain shaded by the concentration of $u$ with  half of the bulk domain made transparent. The bottom left-hand panel indicates the  surface shaded by the concentration of $z$ and the bulk domain shaded by the concentration of $u$ with  half of the bulk domain made transparent. The top right-hand panel indicates the values of the trace of $u$ (blue), $w$ (red) and $z$ (black) on the curve on $\G_h$ with $x_3=0$ and the bottom right-hand panel indicates the values of the approximation $j=-\VG\cdot\nu$ on the curve on $\G_h$ with $x_3=0$.}\label{fig:evo_hetero}
\end{figure}

\begin{figure}[h]
\centering
\includegraphics[width=0.15\linewidth]{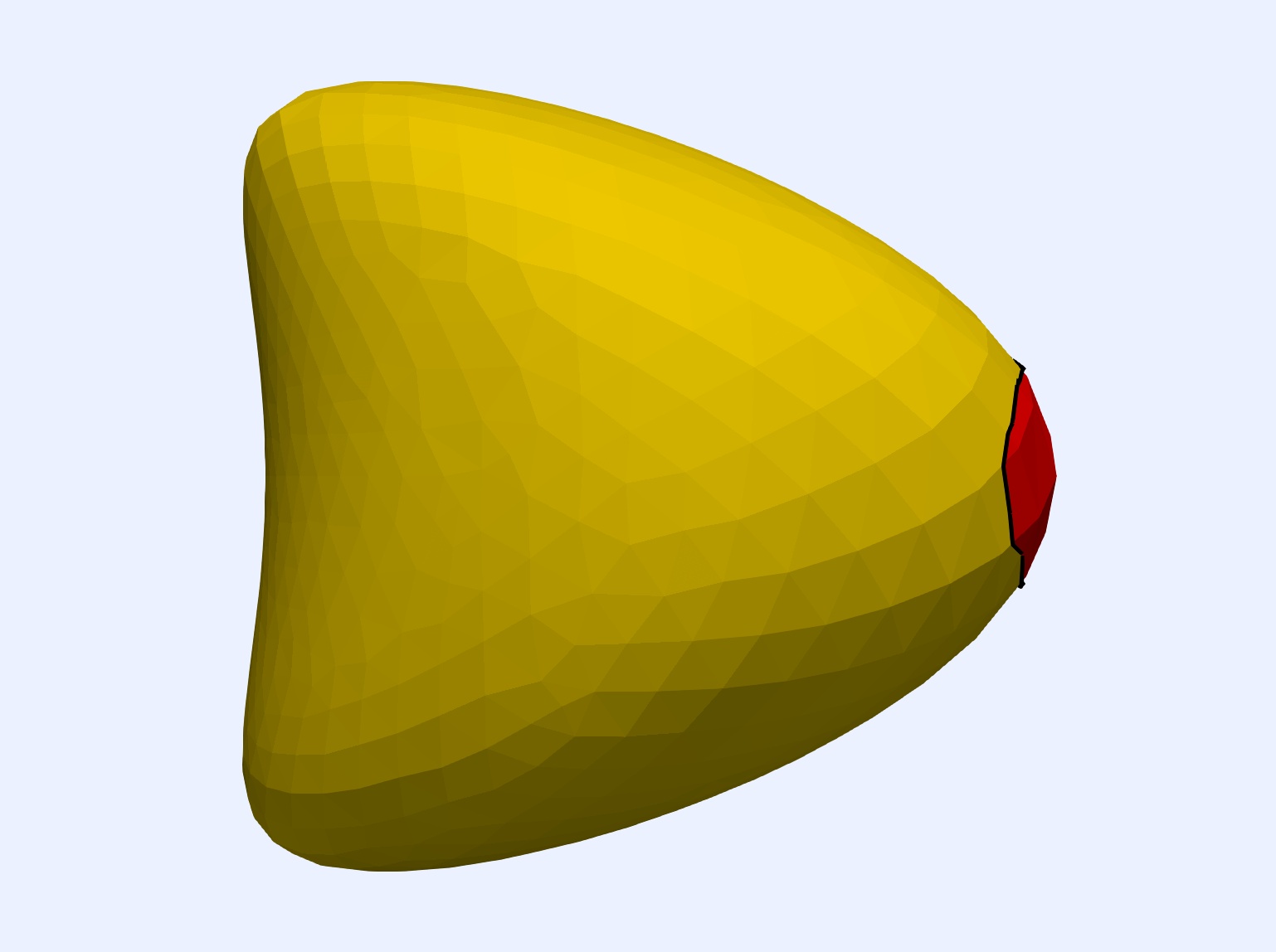}
\includegraphics[width=0.15\linewidth]{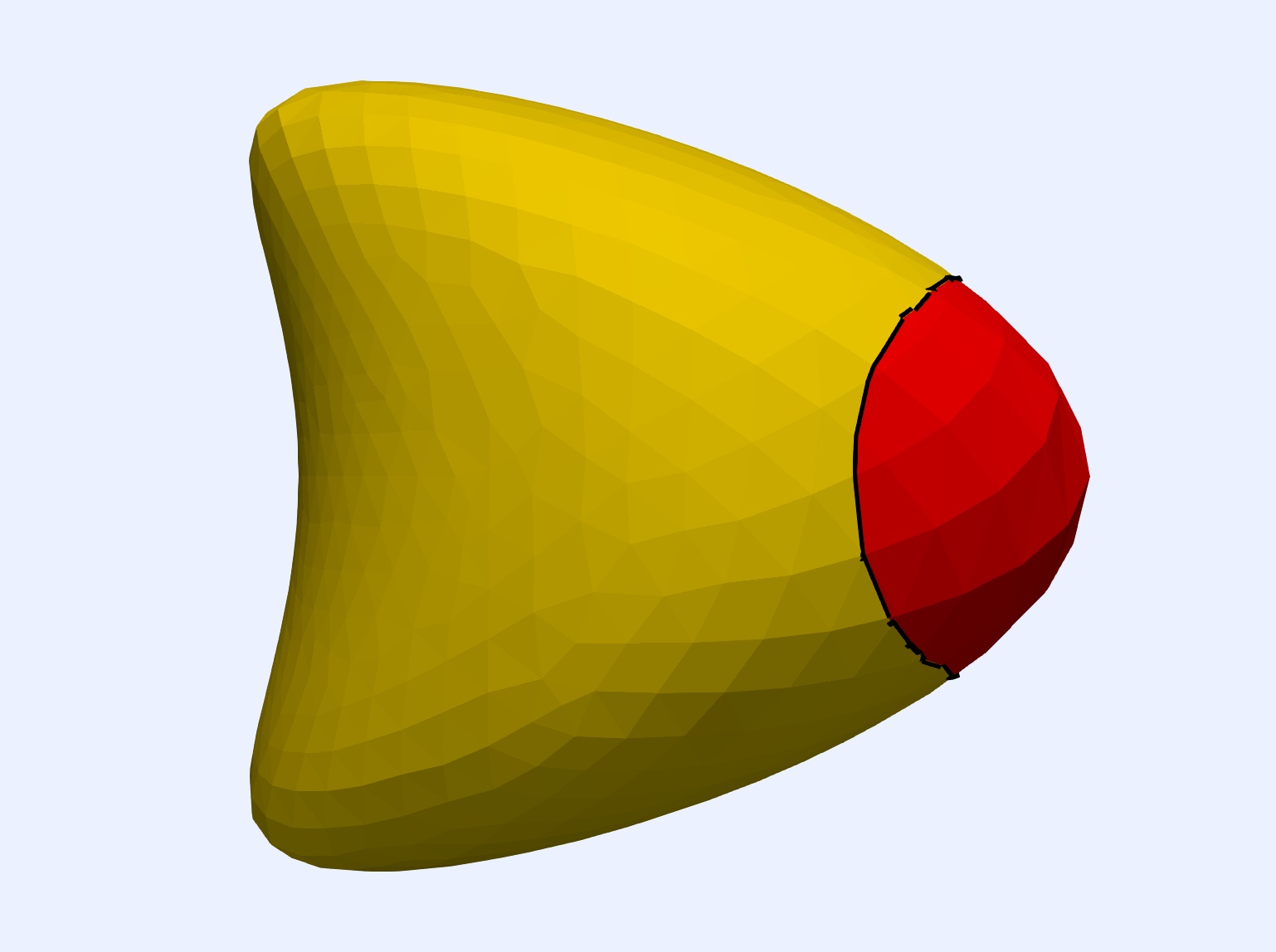}
\includegraphics[width=0.15\linewidth]{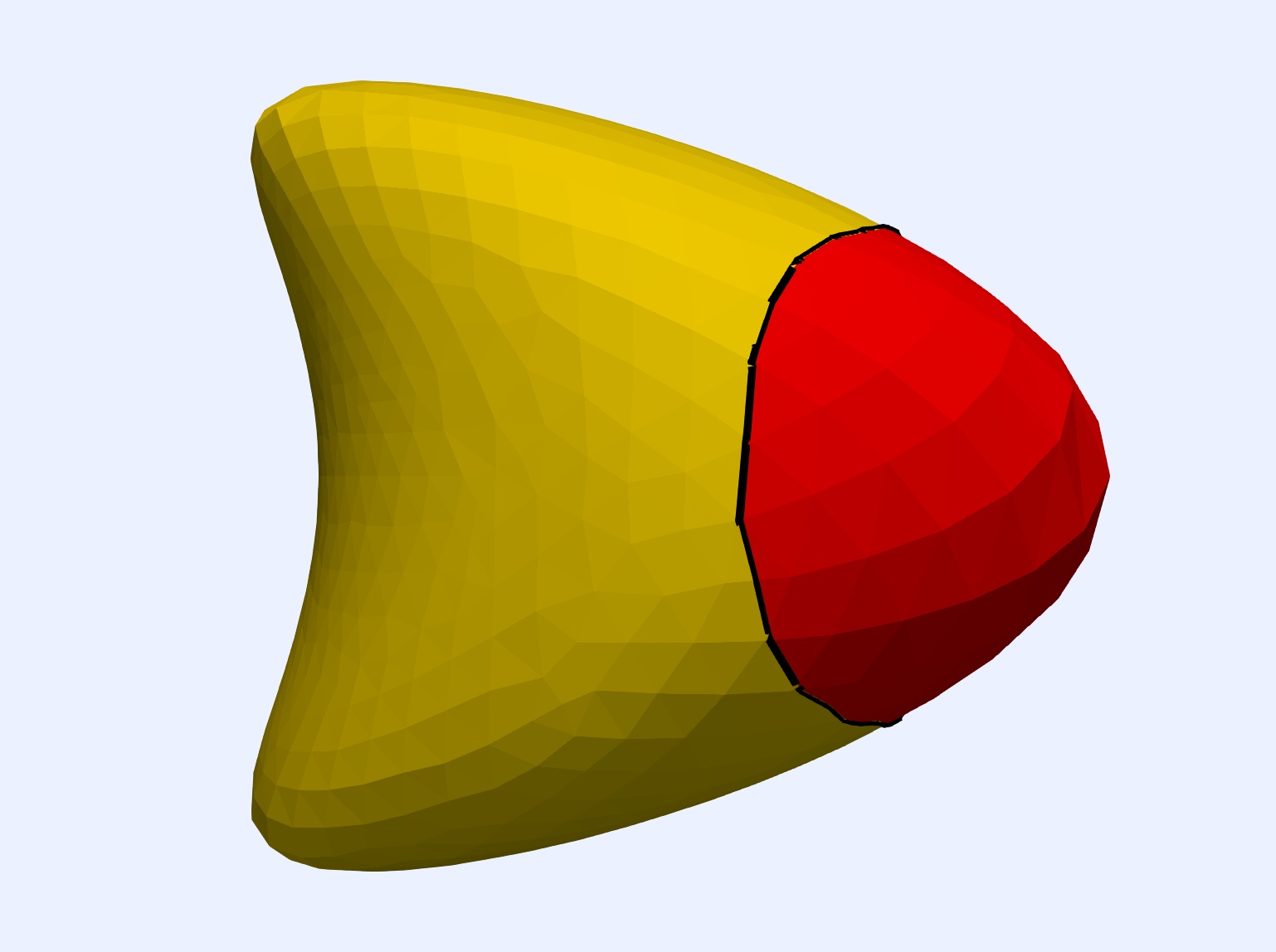}
\includegraphics[width=0.15\linewidth]{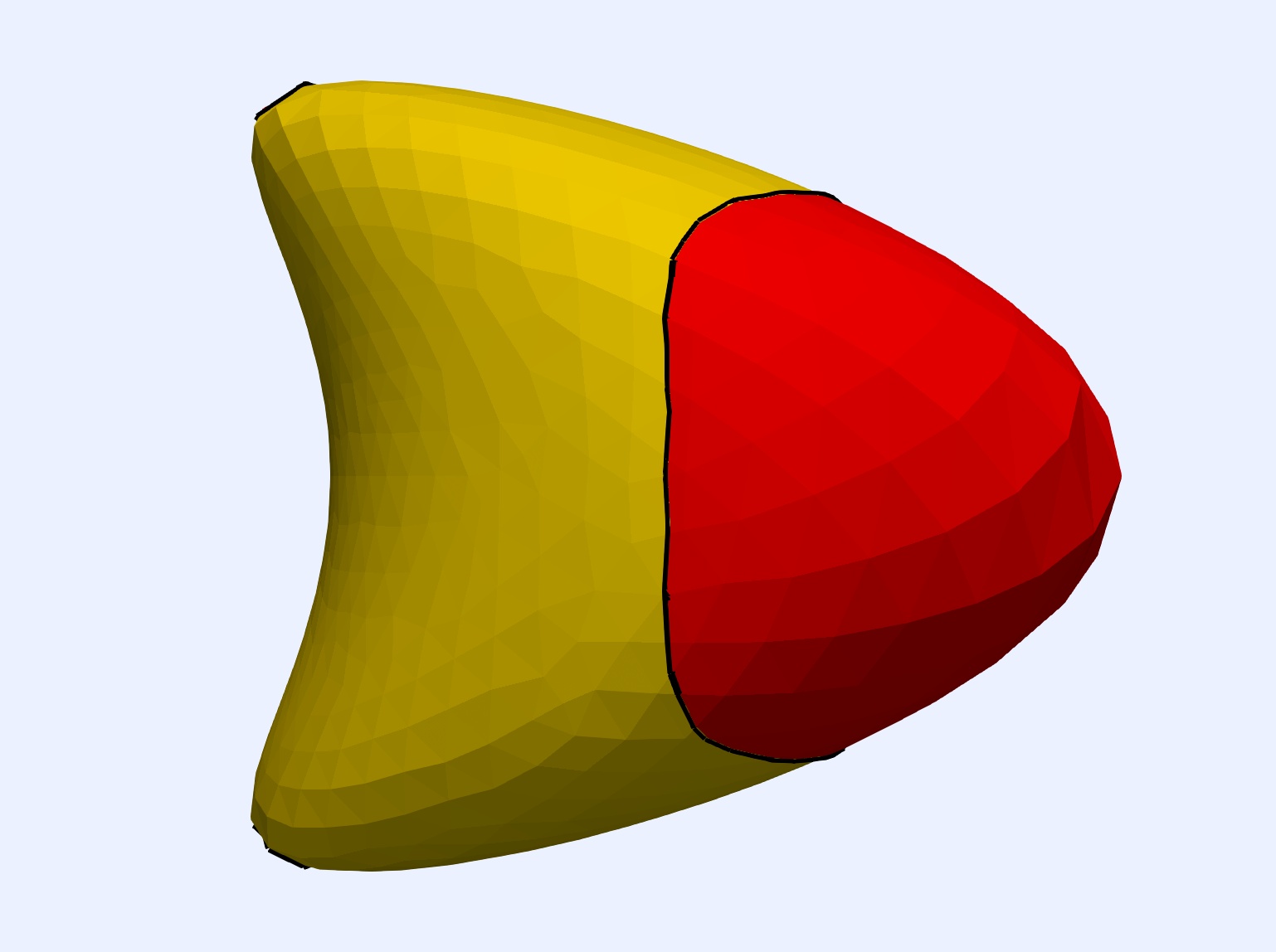}
\includegraphics[width=0.15\linewidth]{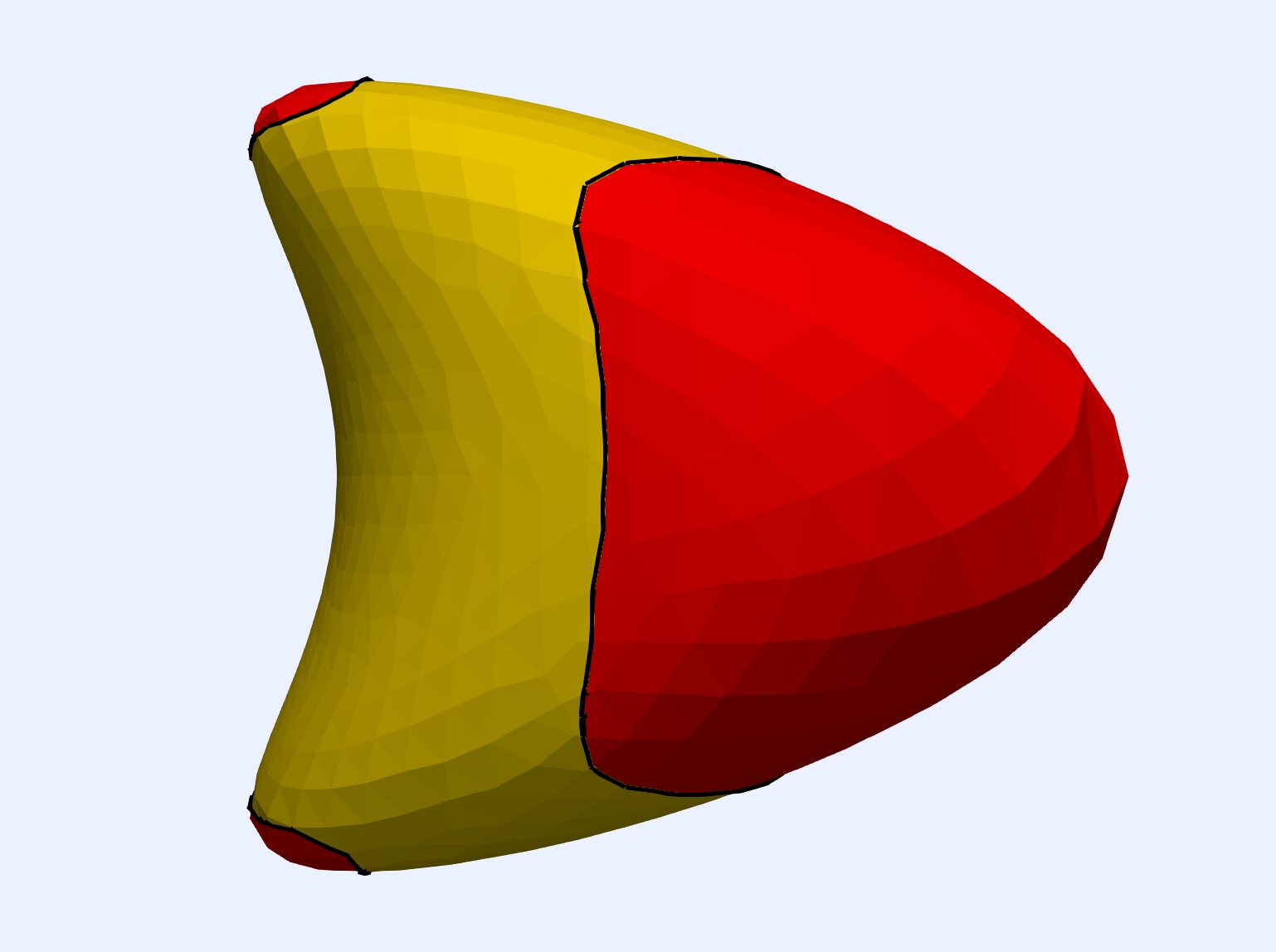}
\includegraphics[width=0.15\linewidth]{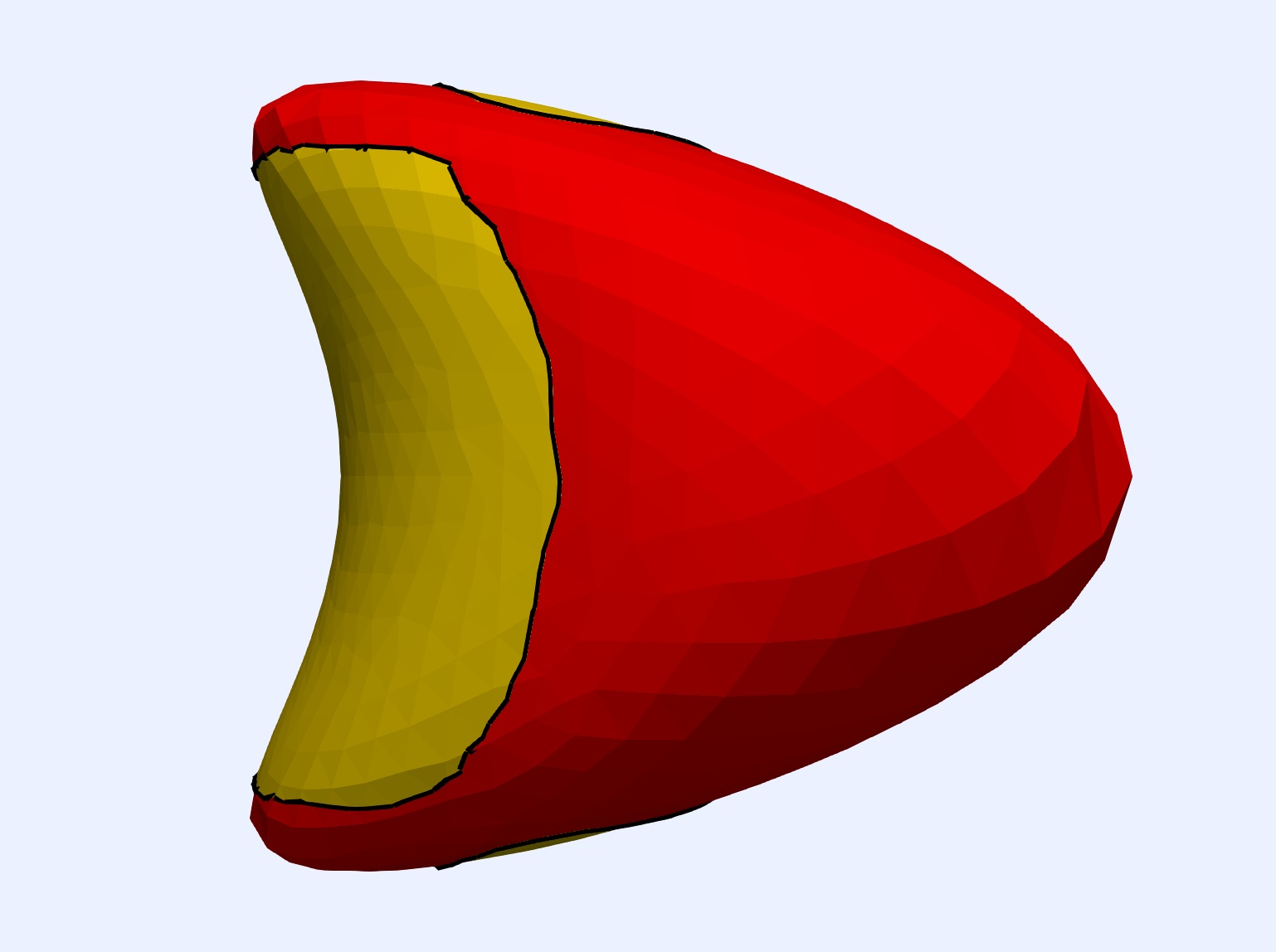}
\includegraphics[width=0.15\linewidth]{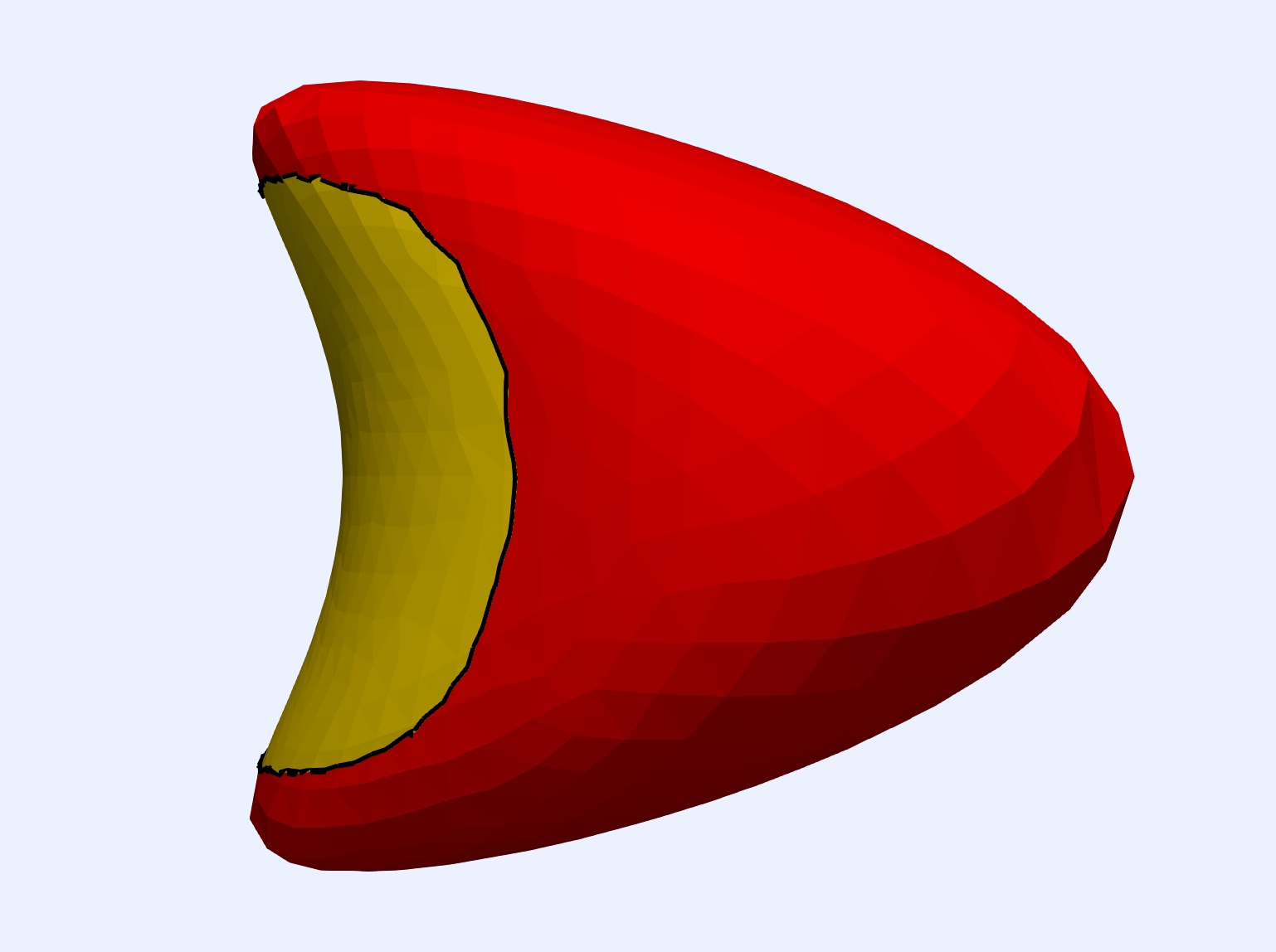}
\includegraphics[width=0.15\linewidth]{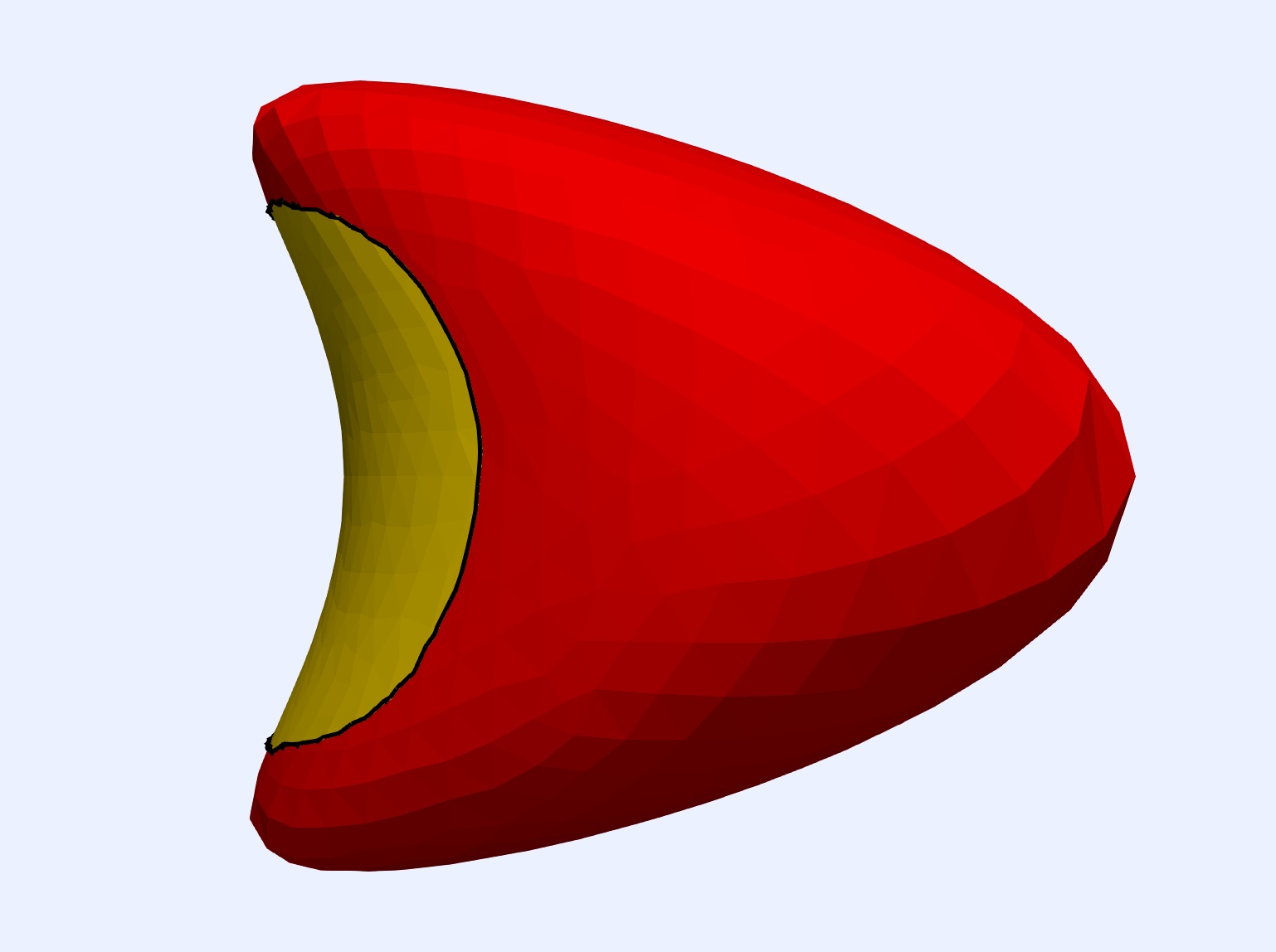}
\includegraphics[width=0.15\linewidth]{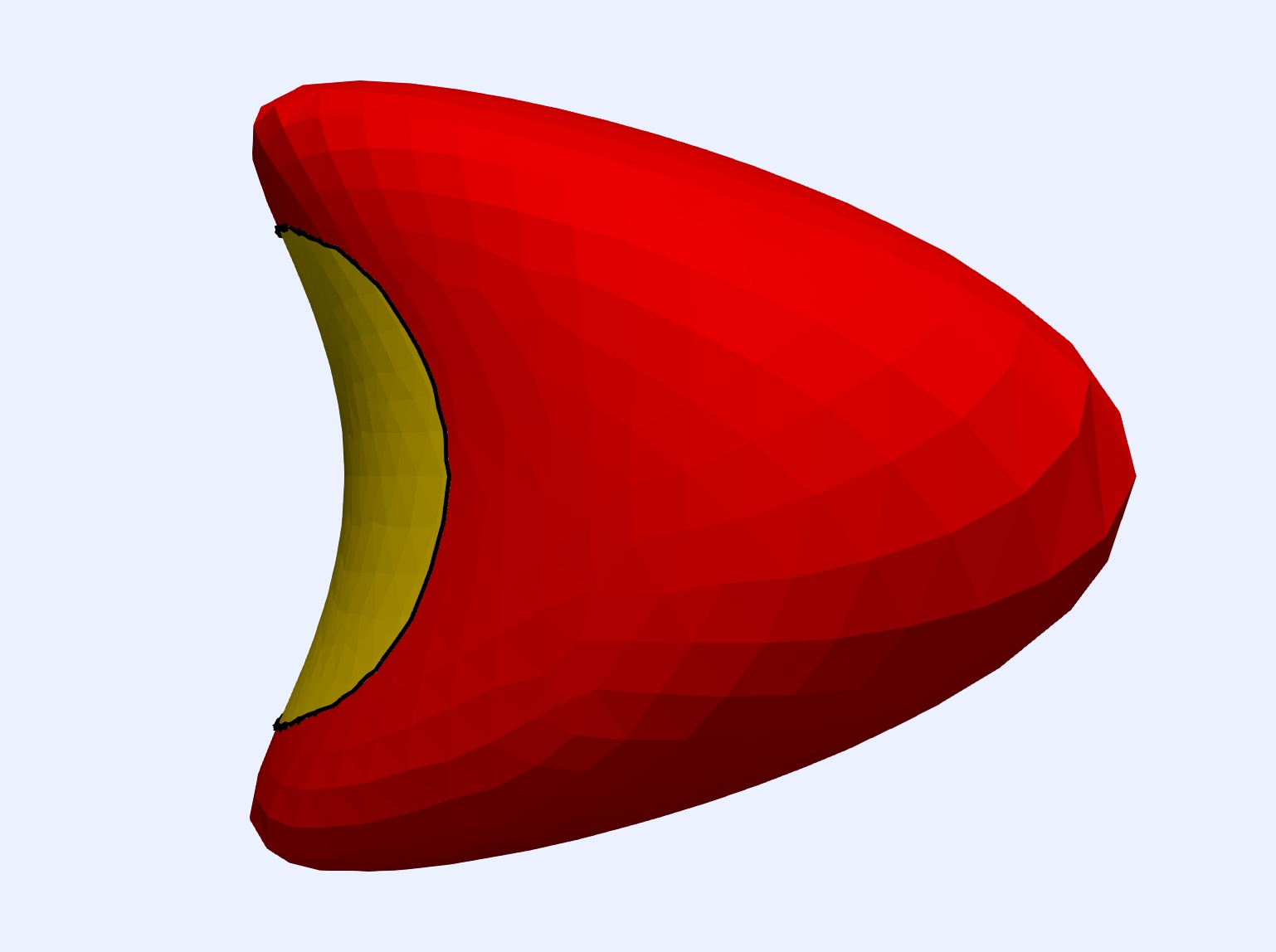}
\includegraphics[width=0.15\linewidth]{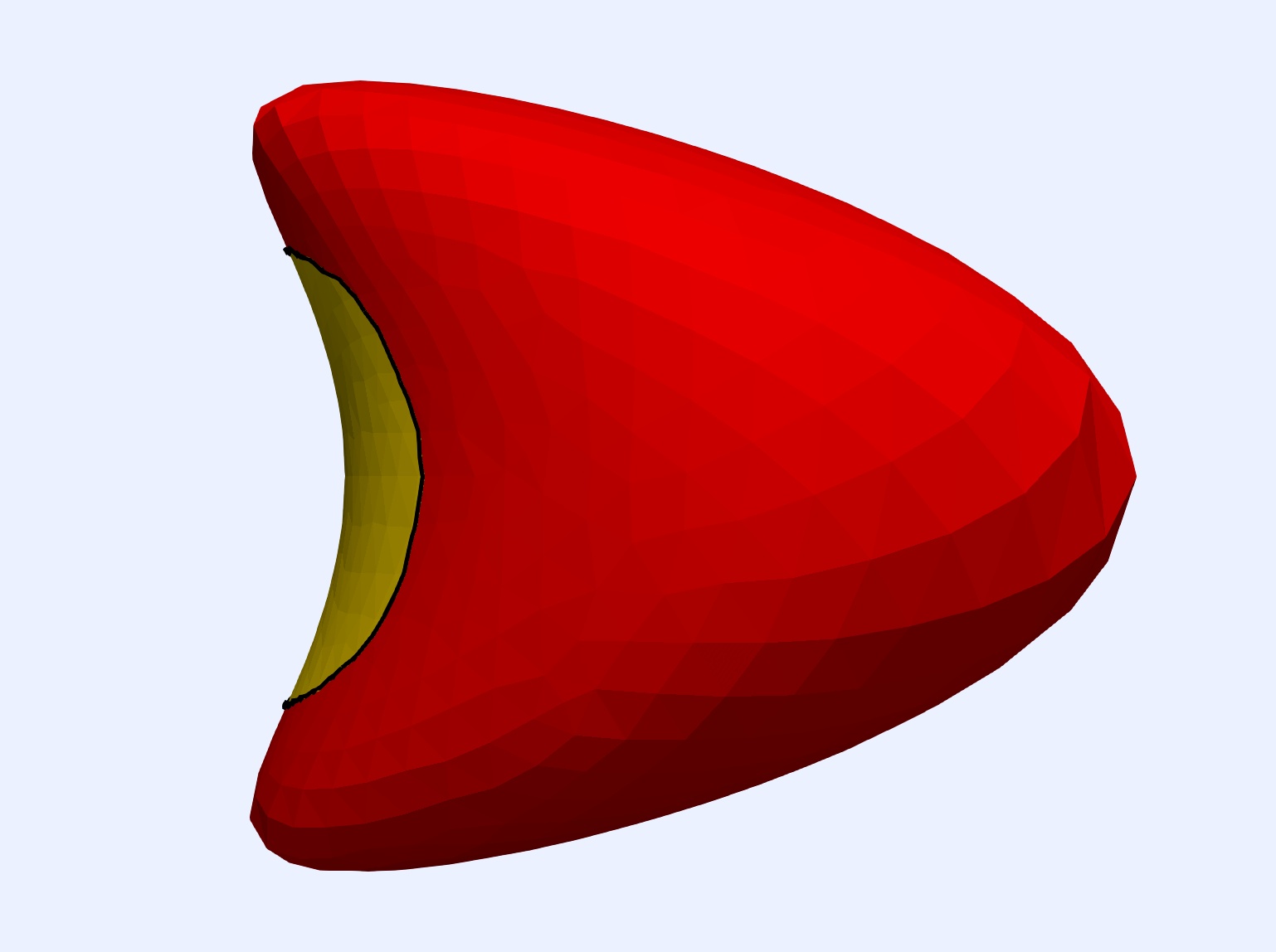}
\includegraphics[width=0.15\linewidth]{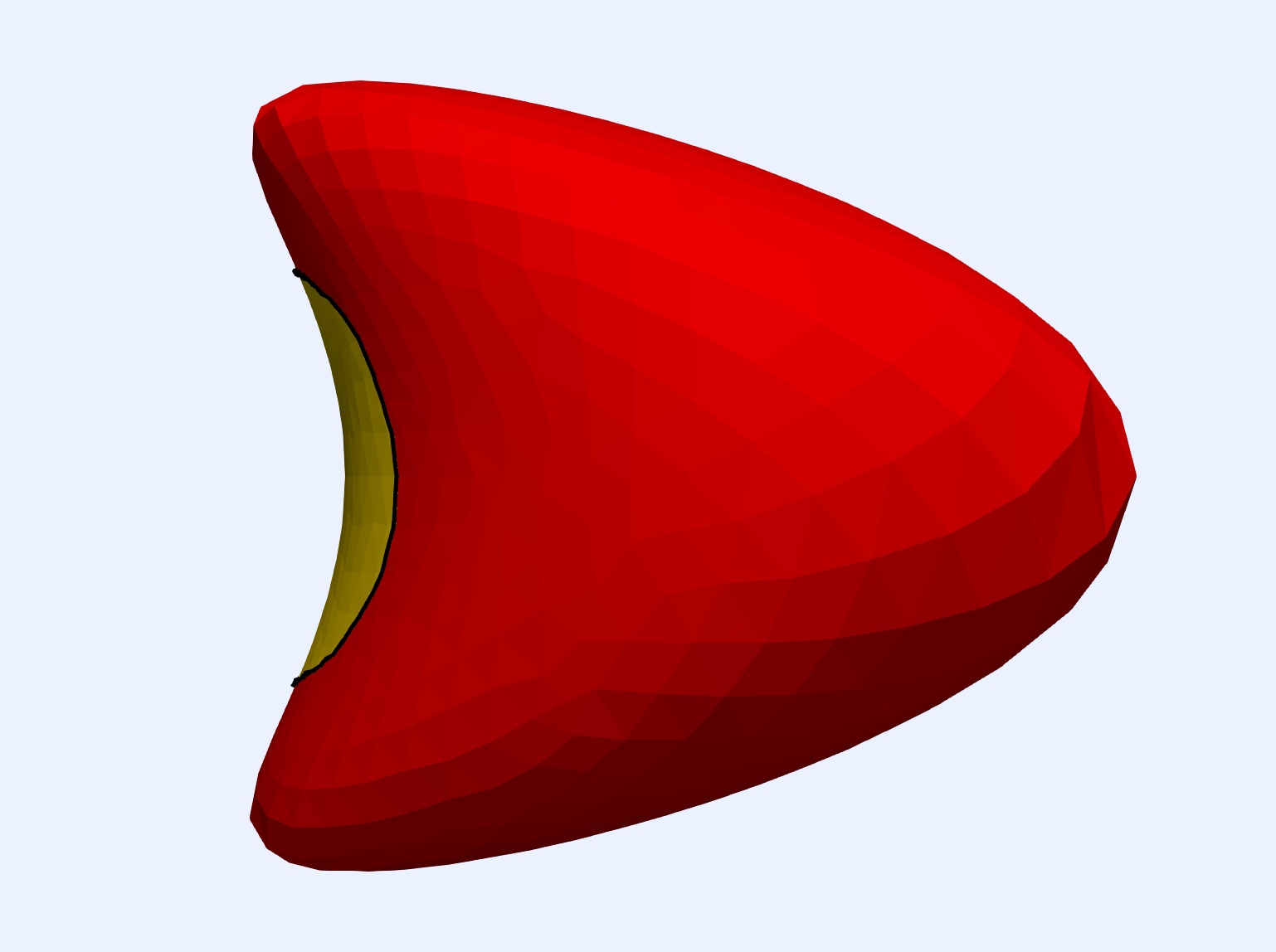}
\includegraphics[width=0.15\linewidth]{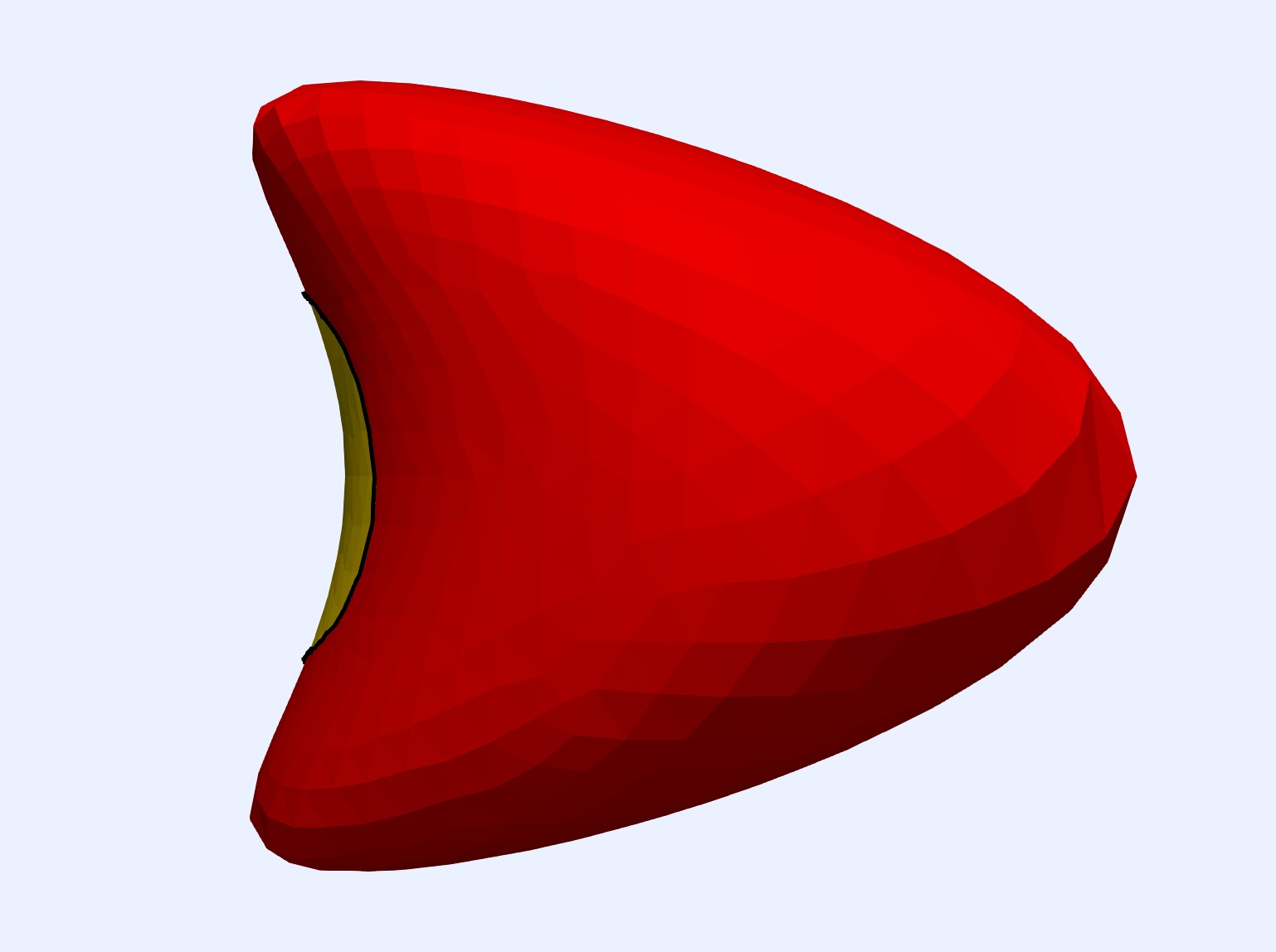}
\caption{Simulation results of  \cref{sec:evo_hetero} approximating  the evolution of the free boundary in \eqref{eqn:surface_stefan} (see \cref{sec:evo_hetero} for details). In each snapshot the level curve on which  $w=0.1$ (black line) is shown with the surface coloured yellow for $w>0.1$ and red for $w<0.1$. The snapshots are shown   from  $t=0.12$ to $t=0.84$ at   uniform steps of $0.06$ reading from left to right and top to bottom. }
\label{fig:surface_FBP}
\end{figure}

\subsection{Simulations approximating the $\delk=\delg=\delgp\to 0$ limit.}\label{sec:windshield}

For the simulation results reported on in this subsection, we take $\delk=\delg=\delgp=0.001$ and take $\delo=\delz=1$, hence the results can be interpreted  as an approximation of the limiting problem stated in \cref{sec:delk_delg_limit}. We set $g(u,w)= \frac{u^2w}{1+u^2}$. We take constant initial and Dirichlet boundary data for $u$ with $u_0=u_D=1$ and for the initial condition for the surface species we set $w^0(\vec x)=e^{-6\left(1-x_1^2\right)}$ and $z^0(\vec x)=0$.

To illustrate the influence of the so called windshield effect, we consider the domain evolution described in \cref{sec:evo_and_disc} with the windshield effect being present in the case $\VO=\vec 0$ and absent in the case $\VO=E(\VG)$ as in the latter case the (inner) boundary of $\O$ moves with the same velocity as the surface $\G$. We take the timestep $\tau=10^{-6}.$ 

\cref{fig:windshield} shows the results of the simulation with $\VO=\vec 0$. We take $\vec J_{\O,h}^n\in\V_h^n$ and define its nodal values such that $\vec J_{\O,h}^n(\vec X^n_j)=-(\vec X^n-\vec X_j^{n-1})/\tau, j=1,\dots,N_\O$. For the windshield effect term we take $J_h^n=I_{\G,h}^n\left(-\frac{\phi_t(\vec x,t)}{\vert\nabla\phi(\vec x,t)\vert}\right)$ with $I_{\G,h}$ the linear Lagrange interpolant.  We see a rapid initial evolution such that the supports of the trace of $u$ on $\G$ and $w$ become close to disjoint and then a slower evolution as $w$ is depleted an $u$ increases.    We see $u$ is larger near regions of the  surface where $j$ is largest with the maximum value of $u$ exceeding 1 and that $u$ is smaller near regions where $j$ is smallest, thus, demonstrating the windshield effect.

\cref{fig:no_windshield} shows the results of the simulation with $\VO=E(\VG)$, i.e., the harmonic extension of the velocity of the surface. In this case $\vec J_{\O,h}^n=\vec 0$ and and $J_h^n=0$, i.e., the scheme is Lagrangian.  We see a rapid initial evolution such that the supports of the trace of $u$ on $\G$ and $w$ become close to disjoint and then a slower evolution as $w$ is depleted and $u$ increases. The absence of the windshield effect results in $u$ not exceeding 1 throughout the evolution and the profiles for $w$ and $z$ remain broadly similar to the case considered in \cref{fig:windshield}. The fact that the surface species exhibit similar dynamics with or without the windshield effect may have implications to biological phenomena such as chemotaxis \cite{endres2008accuracy,macdonald2016computational}.

\begin{figure}[htb!]
\centering
\subfigure[][{$t=0$}]{
\includegraphics[width=.3\textwidth
]{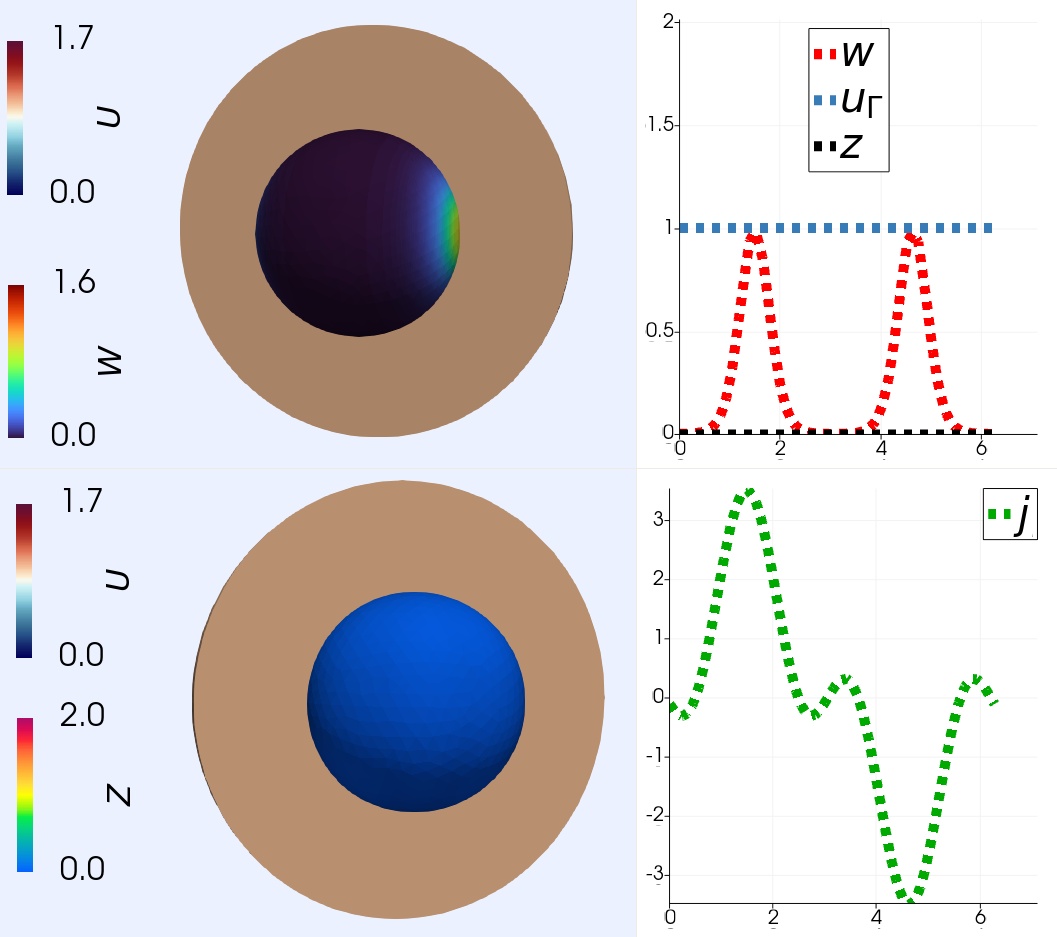}
}
\hskip0.1em
\subfigure[][{$t=0.004$}]{
\includegraphics[width=.3\textwidth
]{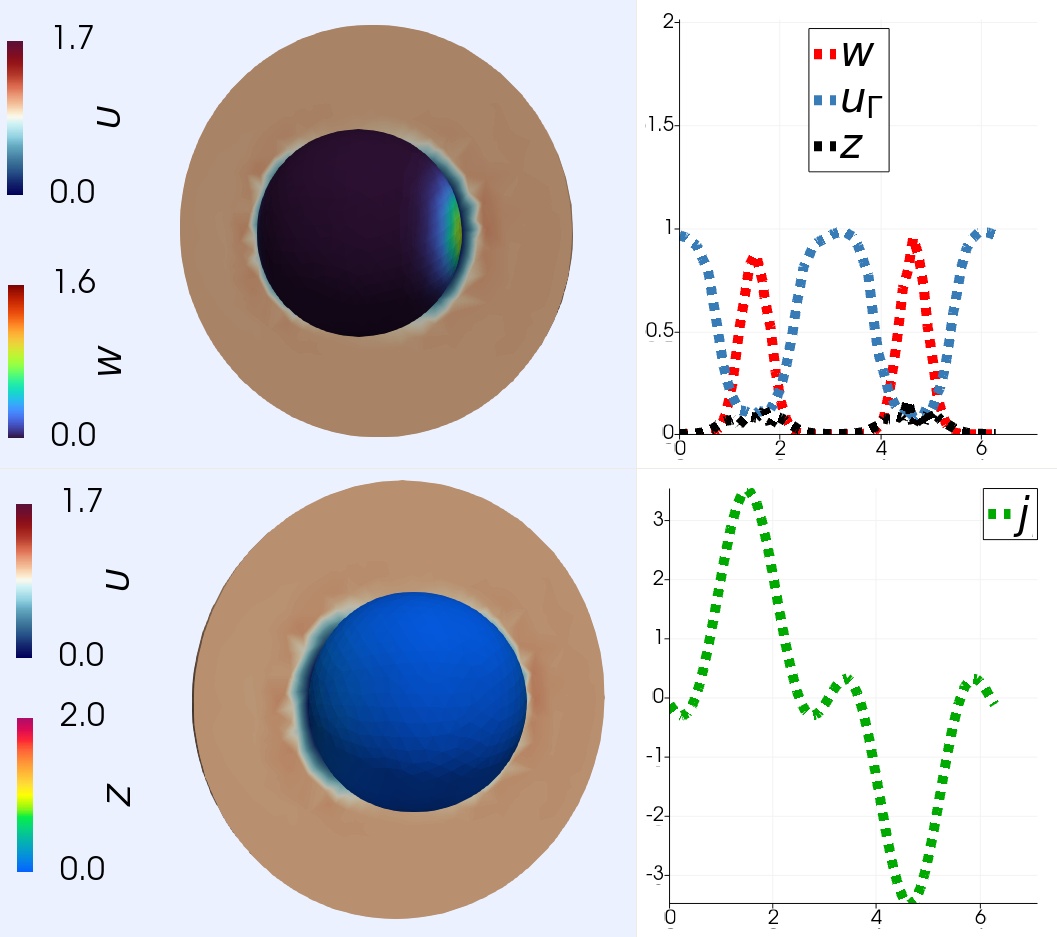}
}
\hskip0.1em
\subfigure[][{$t=0.024$}]{
\includegraphics[width=.3\textwidth
]{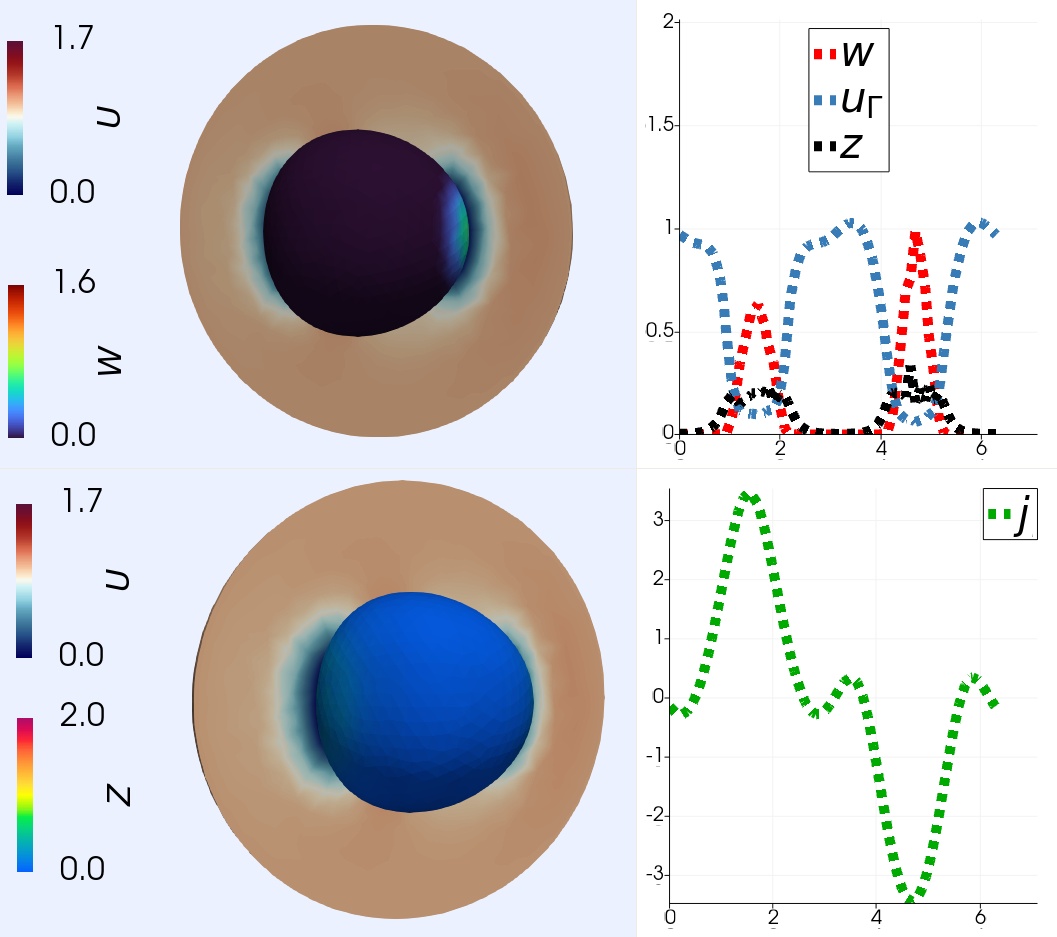}
}
\hskip0.1em
\subfigure[][{$t=0.06$}]{
\includegraphics[width=.3\textwidth
]{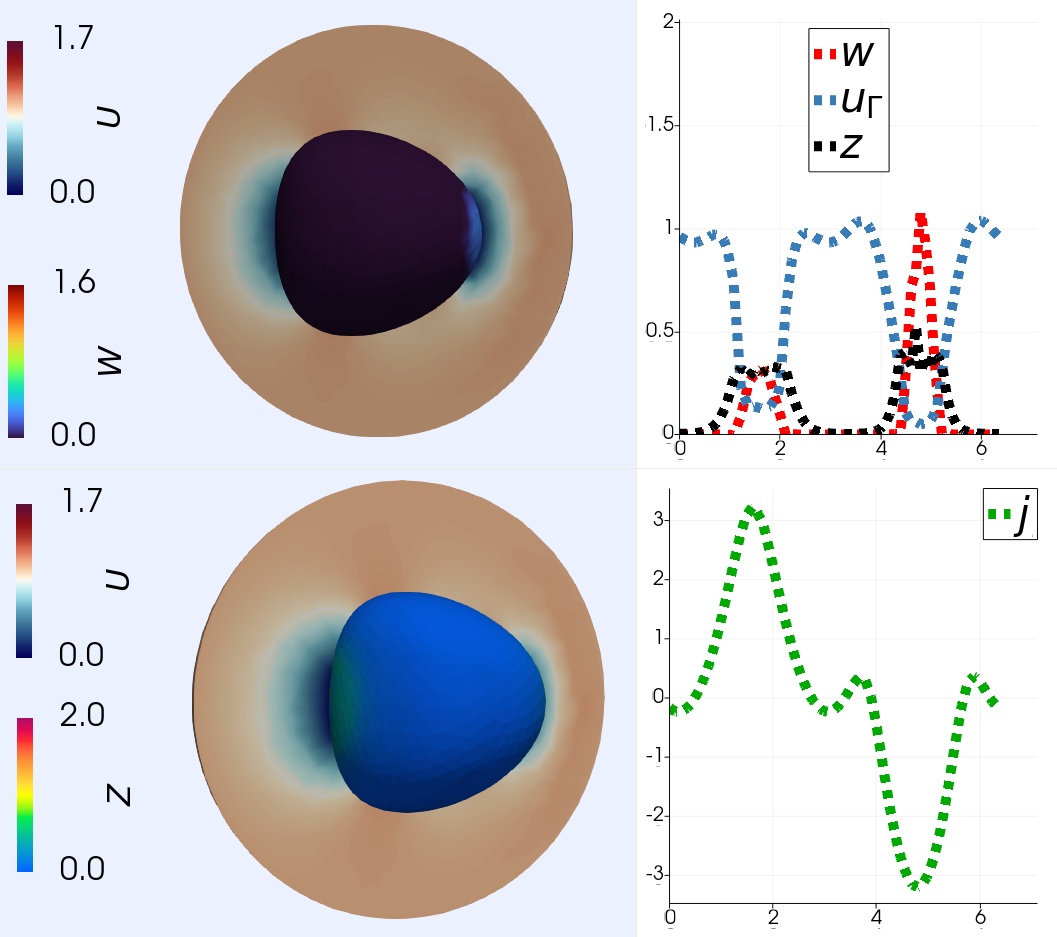}
}
\hskip0.1em
\subfigure[][{$t=0.2$}]{
\includegraphics[width=.3\textwidth
]{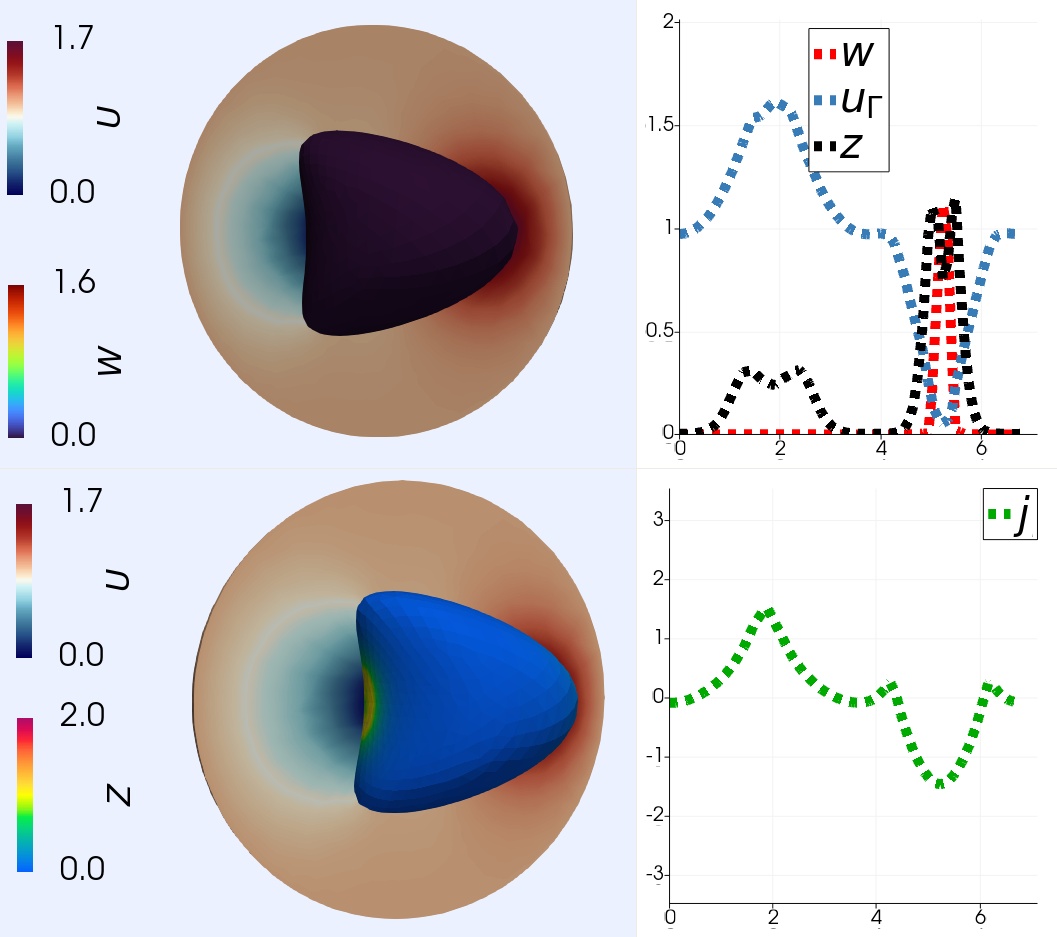}
}
\hskip0.1em
\subfigure[][{$t=0.4$}]{
\includegraphics[width=.3\textwidth
]{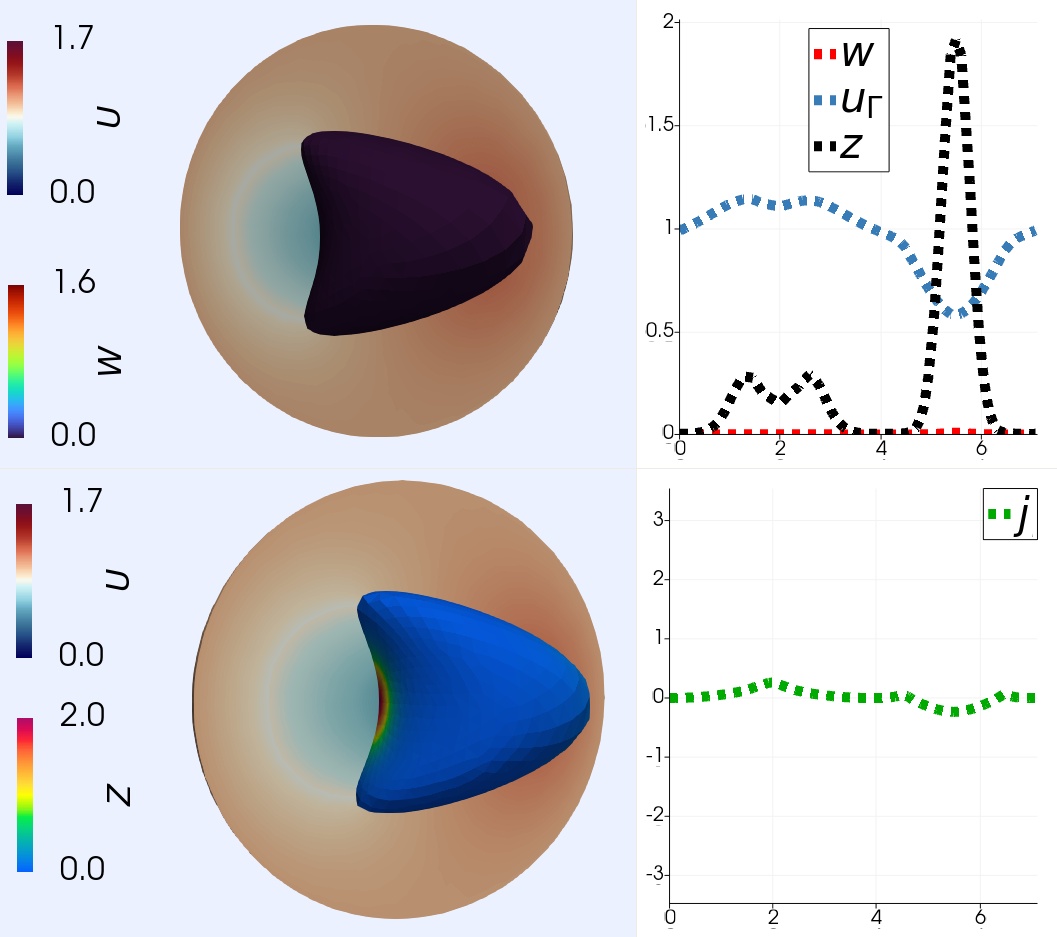}
}
\caption{Simulation results of  \cref{sec:windshield} approximating  the $\delk=\delg=\delgp\to0$ limit of \cref{sec:delk_delg_limit} with $\VO=\vec 0$ illustrating the windshield effect (see \cref{sec:windshield} for details of the parameters). In each subfigure the top left-hand panel indicates the  surface shaded by the concentration of $w$ and the bulk domain shaded by the concentration of $u$ with  half of the bulk domain made transparent. The bottom left-hand panel indicates the  surface shaded by the concentration of $z$ and the bulk domain shaded by the concentration of $u$ with  half of the bulk domain made transparent. The top right-hand panel indicates the values of the trace of $u$ (blue), $w$ (red) and $z$ (black) on the curve on $\G_h$ with $x_3=0$ and the bottom right-hand panel indicates the values of the approximation $j=-\VG\cdot\nu$ on the curve on $\G_h$ with $x_3=0$.}\label{fig:windshield}
\end{figure}

\begin{figure}[htb!]
\centering
\subfigure[][{$t=0$}]{
\includegraphics[width=.3\textwidth
]{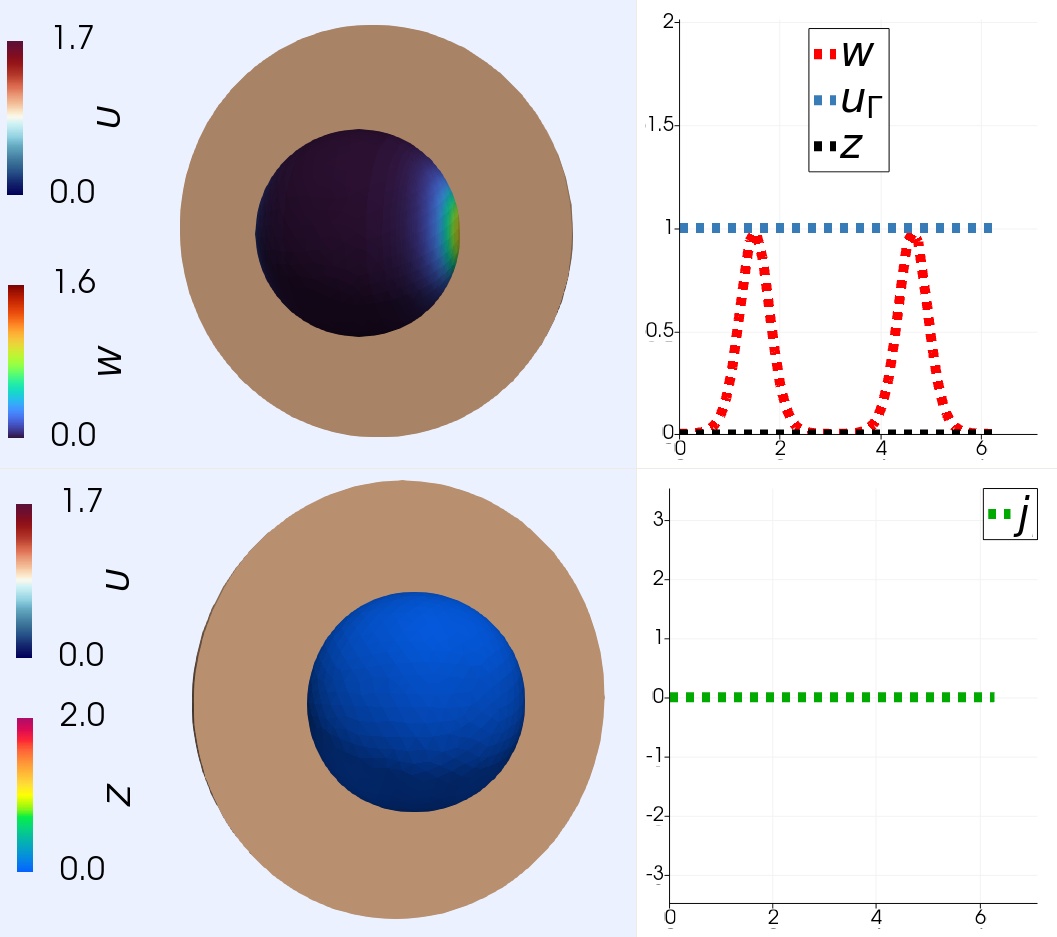}
}
\hskip0.1em
\subfigure[][{$t=0.004$}]{
\includegraphics[width=.3\textwidth
]{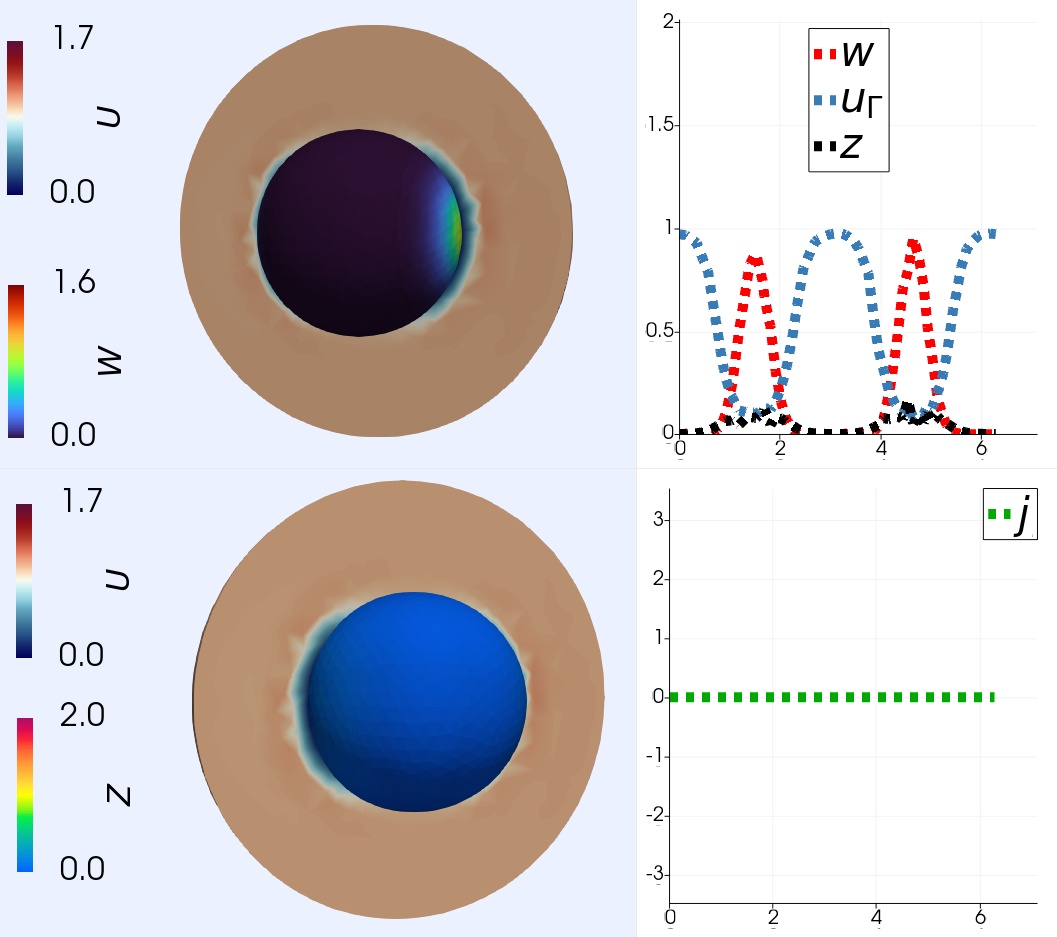}
}
\hskip0.1em
\subfigure[][{$t=0.024$}]{
\includegraphics[width=.3\textwidth
]{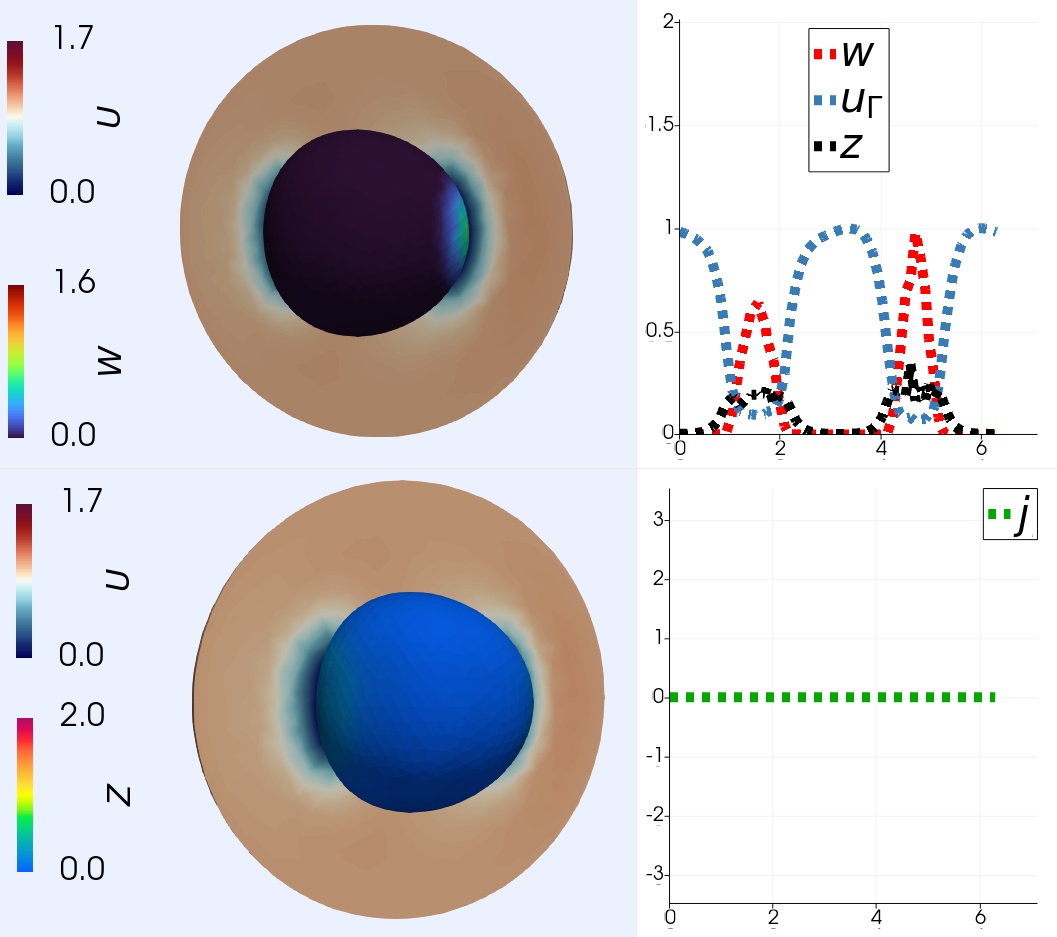}
}
\hskip0.1em
\subfigure[][{$t=0.06$}]{
\includegraphics[width=.3\textwidth
]{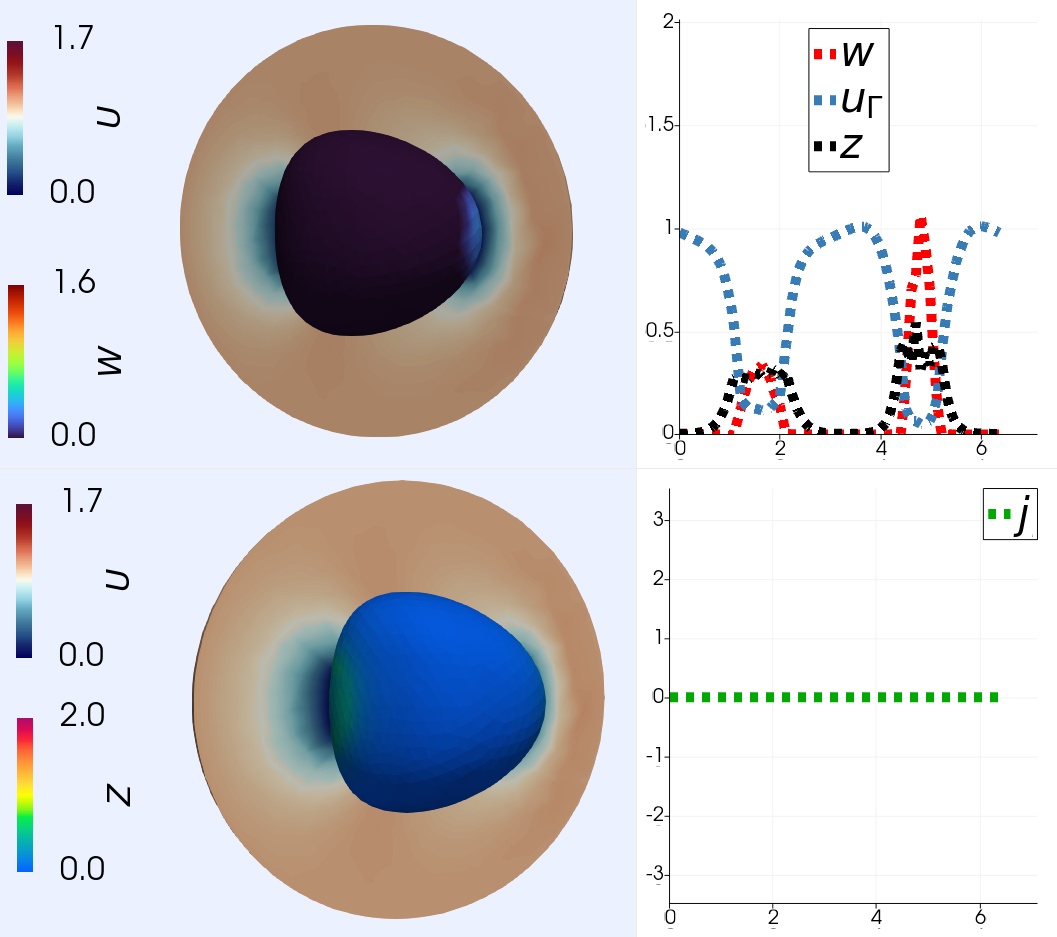}
}
\hskip0.1em
\subfigure[][{$t=0.2$}]{
\includegraphics[width=.3\textwidth
]{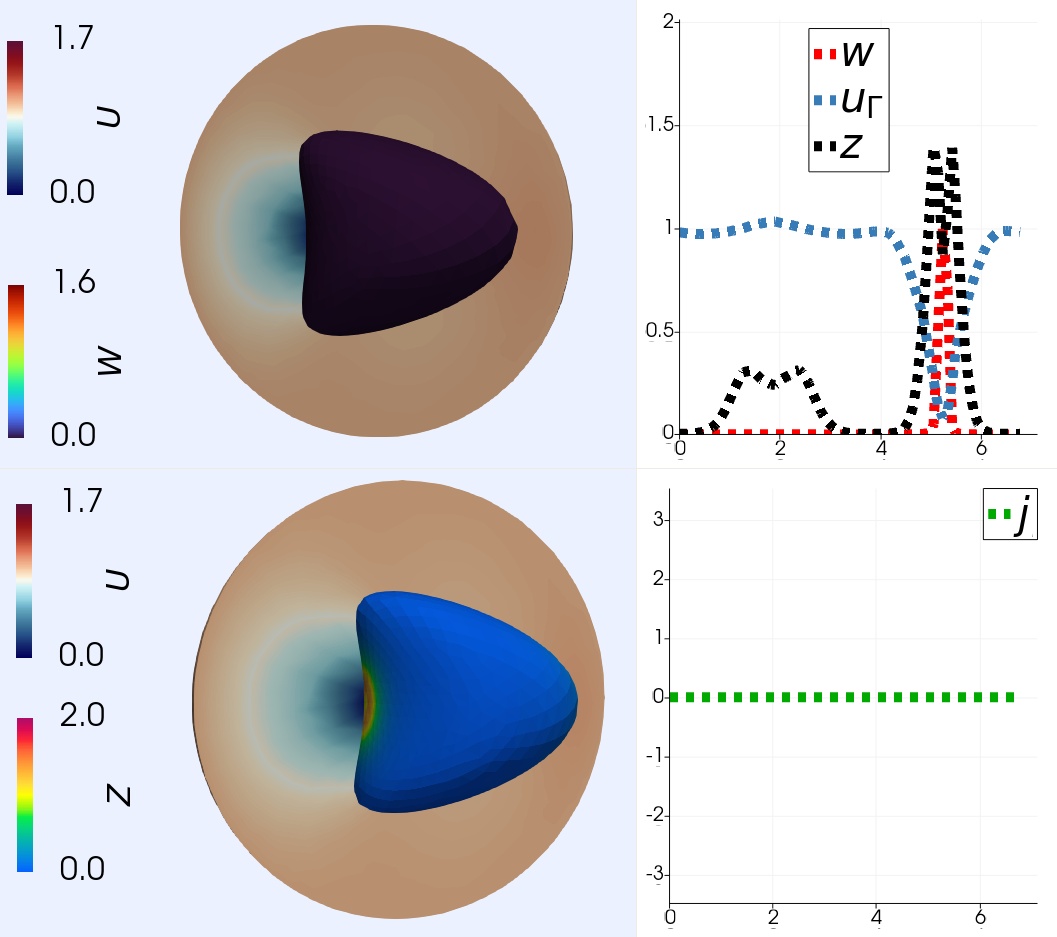}
}
\hskip0.1em
\subfigure[][{$t=0.4$}]{
\includegraphics[width=.3\textwidth
]{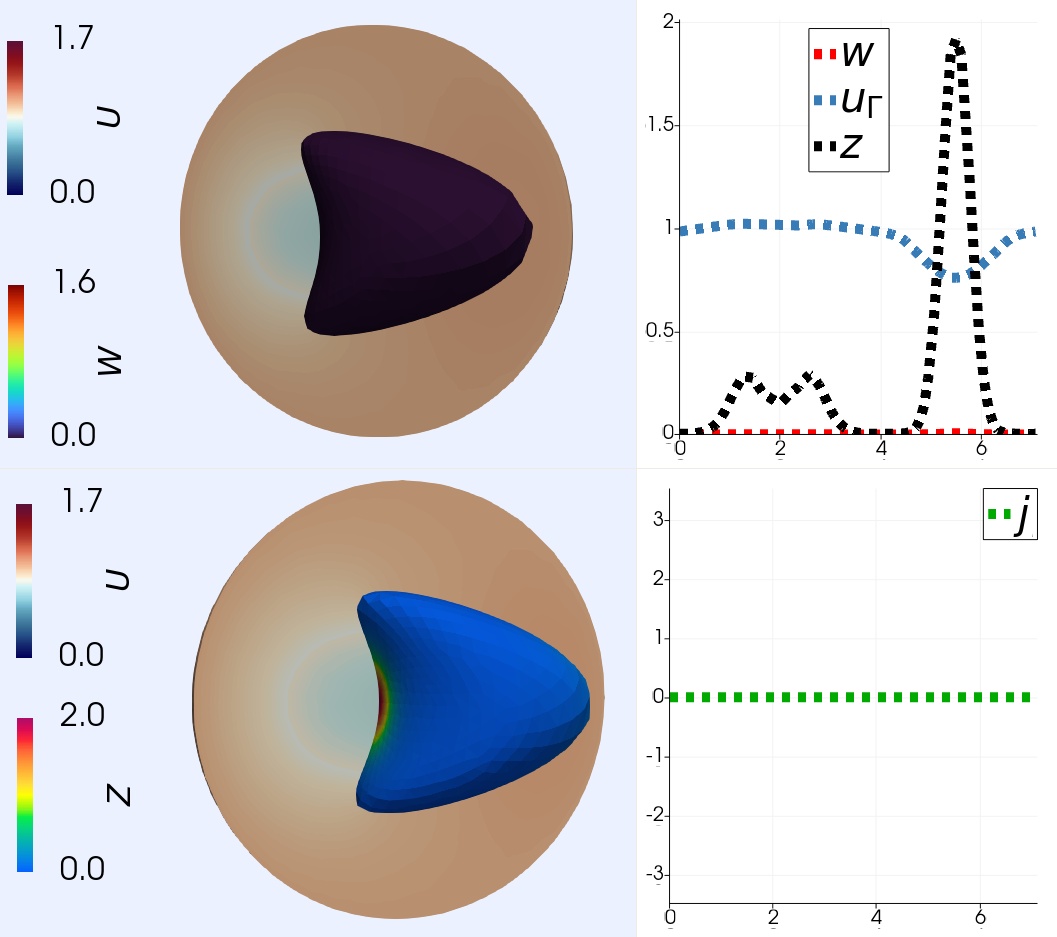}
}
\caption{Simulation results of  \cref{sec:windshield} approximating  the $\delk=\delg=\delgp\to0$ limit of \cref{sec:delk_delg_limit} with $\VO=E(\VG)$, resulting in no windshield effect, and everything else as in \cref{fig:windshield} (see \cref{sec:windshield} for details of the parameters). In each subfigure the top left-hand panel indicates the  surface shaded by the concentration of $w$ and the bulk domain shaded by the concentration of $u$ with  half of the bulk domain made transparent. The bottom left-hand panel indicates the  surface shaded by the concentration of $z$ and the bulk domain shaded by the concentration of $u$ with  half of the bulk domain made transparent. The top right-hand panel indicates the values of the trace of $u$ (blue), $w$ (red) and $z$ (black) on the curve on $\G_h$ with $x_3=0$ and the bottom right-hand panel indicates the values of the approximation $j=0$ on the curve on $\G_h$ with $x_3=0$ (included simply for comparison with \cref{fig:windshield}).}\label{fig:no_windshield}
\end{figure}

\section*{Acknowledgements}
AA thanks Alexander Mielke and Marcelo Bongarti for useful discussions. The research of DC was funded by an EPSRC grant under the MASDOC centre for doctoral training at the University of Warwick. CV was partially supported by an individual grant from the Dr Perry James (Jim) Browne Research Centre on Mathematics and its Applications (University of Sussex). 

\appendix
\section{Integration by parts identities}\label{sec:app_identities}
We collect some technical facts that are used in the paper. Here and below, $\Omega \subset \mathbb{R}^{d+1}$ is a sufficiently smooth bounded domain with $\partial \Omega = \Gamma \cup \partial D$. 

\subsection{Bulk identities}
Let $\mathbf{J}\colon \mathbb{R}^{d+1} \to \mathbb{R}^{d+1}$ be a vector field satisfying
\[\mathbf{J}\cdot \nu = 0 \quad\text{on $\partial D$}.\]
For sufficiently regular functions $a\colon \Omega \to \mathbb{R}$ and $b\colon \Omega \to \mathbb{R}$, we have from the product rule and the divergence theorem
\[\int_\Omega (\mathbf{J}\cdot\grad a)b + (\grad\cdot\mathbf{J})ab  = \int_\Omega \grad \cdot (a\mathbf{J})b = \int_\Omega\grad \cdot (a\mathbf{J}b) - a\mathbf{J}\cdot \grad b = \int_{\Gamma}ab \mathbf{J} \cdot \nu -\int_\Omega a(\mathbf{J}\cdot \grad b).\]
From this we can deduce several expressions that will be useful throughout the paper:
\begin{align}
\nonumber \int_\Omega \grad \cdot (a\mathbf{J})b &= \int_{\Gamma}ab \mathbf{J} \cdot \nu -\int_\Omega a\mathbf{J}\cdot \grad b,\\
     \int_\Omega  (\mathbf{J}\cdot \grad a)a &= \frac 12 \int_{\Gamma} a^2\mathbf{J}\cdot \nu - \frac 12\int_\Omega a^2  \grad \cdot \mathbf{J}\label{eq:bulkId1},\\
     \intertext{and with $a_k := (a-k)^+$ for a number $k \in \mathbb{R}$,}
    \int_\Omega \grad \cdot (\mathbf{J}a)a_k  &= \int_\Omega (\grad \cdot \mathbf{J}) aa_k +\frac 12 \int_\Gamma a_k^2\mathbf{J}\cdot\nu - \frac{1}{2}\int_\Omega a_k^2 \grad \cdot \mathbf{J}\label{eq:bulkPosId}.
    \end{align}    
    Here, we used that $\grad a=\grad a_k$  in $\operatorname{supp} \{a_k\}$ to write
$\grad\cdot (\mathbf{J}a)a_k= (\grad\cdot\mathbf{J})aa_k+ (\mathbf{J}\cdot\grad a_k)a_k$. Now, using \eqref{eq:bulkPosId}, we can derive
\begin{align}
\intOt \grad \cdot (\mathbf{J}_\Omega a)a^+  &= \frac 12\intOt |a^+|^2\grad \cdot \mathbf{J}_\Omega  +\frac 12 \intGt j|a^+|^2\label{eq:bulkIBPidentity}
\end{align}
if we recall that $j$ is the jump of the velocities as defined in \eqref{eq:defn_jumps}. 

Now if $a, b \in \mathbb{W}(H^1(\Omega), H^1(\Omega)^*)$, we have the transport formula
\begin{align*}
\frac{d}{dt}\intOt ab = \langle \mdd a, b \rangle_{H^1(\Omega(t))^*, H^1(\Omega(t))} + \langle \mdd b, a \rangle_{H^{1}(\Omega(t))^*, H^1(\Gamma(t))} + \intOt ab\dVpO
\end{align*}
and thus
\begin{align*}
\langle \mdd a, a \rangle = \frac 12 \frac{d}{dt}\intOt a^2 - \frac 12 \intOt a^2\dVpO.
\end{align*}
Using the fact that the identity \eqref{eq:bulkId1} implies
\begin{align*}
\intOt \grad \cdot (\JO a)a &=  \frac 12\intOt a^2 \grad \cdot \JO + \frac 12\intGt a^2 j,
\end{align*}
we may deduce
\begin{align}
\nonumber \langle \mdd a, a \rangle + \int_{\Omega(t)}a^2\grad \cdot \Vp + \int_{\Omega(t)} \grad \cdot (\JO a) a &= \frac 12 \frac{d}{dt}\int_{\Omega(t)}a^2 + \frac 12\int_{\Omega(t)}a^2\grad \cdot \Vp +  \frac 12\intOt a^2 \grad \cdot \JO + \frac 12\intGt  a^2 j  \\
&= \frac 12 \frac{d}{dt}\int_{\Omega(t)}a^2 + \frac 12\int_{\Omega(t)}a^2\grad \cdot \VO + \frac 12\intGt a^2 j.  \label{eq:bulkVeryUsefulIdentity}
\end{align}

\subsection{Surface identities}
Recall the divergence theorem\footnote{Bear in mind that we have defined $\nu(t)$ as the unit normal pointing \textit{into} $\Gamma(t)$, hence the minus sign in the formula.} 
\[\int_\Gamma \sgrad \cdot \mathbf J = -\int_\Gamma H\mathbf J\cdot \nu\] on closed surfaces \cite[Theorem 2.10]{Dziuk2013} where $H$ denotes the mean curvature. 
We derive for a sufficiently smooth function $a\colon \Gamma \to \mathbb{R}$ the identities
\begin{align}
\nonumber \int_\Gamma (\mathbf{J}\cdot \sgrad a)a &= -\frac 12 \int_{\Gamma} a^2H\mathbf{J}\cdot \nu - \frac 12\int_\Gamma (\sgrad \cdot \mathbf{J})a^2,\\
\int_\Gamma \sgrad \cdot (\mathbf{J}a)a_k  &= \int_\Gamma (\sgrad \cdot \mathbf{J}) aa_k  -\frac 12 \nonumber \int_\Gamma a_k^2H(\mathbf{J}\cdot\nu) - \frac{1}{2}\int_\Gamma a_k^2 \sgrad \cdot \mathbf{J}.
\end{align}
We  note the identity
\begin{equation}\label{eq:identity1}
\int_{\Gamma} \sgrad \cdot (\JG a)\eta = -\int_{\Gamma}a\JG \cdot \sgrad \eta,
\end{equation}
which follows by the divergence theorem and the expansion formula
\[\sgrad \cdot (\JG a\eta) = \sgrad \cdot (\JG a)\eta + (\JG \cdot \sgrad \eta)a\]
bearing in mind $\JG\cdot \nu = 0$ (since $\mathbf{V}_p$ and $\mathbf{V}_\Gamma$ have the same normal components).
Hence
\begin{align}
\intGt \sgrad \cdot (\mathbf{J}_\Gamma a)a^+  &= \frac 12\intGt |a^+|^2 \sgrad \cdot \mathbf{J}_\Gamma\label{eq:surfaceIBPidentity} .
\end{align}
\begin{remark}
The equalities \eqref{eq:bulkIBPidentity} and \eqref{eq:surfaceIBPidentity} 
also hold when $a^+$ is replaced with $a^-$. 
\end{remark}

Note that, using again the divergence theorem and remembering $\mathbf J_\Gamma \cdot \nu = 0$,
\begin{align*}
\int_\Gamma \sgrad \cdot (\JG a)a &= \int_\Gamma a^2 \sgrad \cdot \JG + a\JG \cdot \sgrad a\\
&= \int_\Gamma a^2 \sgrad \cdot \JG + \frac 12 \JG \cdot \sgrad (a^2)\\
&= \int_\Gamma a^2 \sgrad \cdot \JG - \frac 12 \sgrad \cdot \JG a^2\\
&= \frac 12\int_\Gamma a^2 \sgrad \cdot \JG .
\end{align*}
For $a, b \in \mathbb{W}(H^1(\Gamma), H^1(\Gamma)^*)$, we have the formula
\begin{align*}
\frac{d}{dt}\intGt ab = \langle \mdd a, b \rangle_{H^{-1}(\Gamma(t)), H^1(\Gamma(t))} + \langle \mdd b, a \rangle_{H^{-1}(\Gamma(t)), H^1(\Gamma(t))} + \intGt ab\dVp
\end{align*}
and thus
we can conclude like before that
\begin{align}
\langle \mdd a, a \rangle + \int_{\Gamma(t)}a^2\dVp + \int_{\Gamma(t)} \sgrad \cdot (\JG a)a  
 &= \frac 12 \frac{d}{dt}\int_{\Gamma(t)}a^2 + \frac 12 \int_{\Gamma(t)}a^2\dVG\label{eq:veryUsefulIdentity}.
\end{align}

\begin{lem}\label{lem:IBPIdentityLaplacian}
For all  $a \in \mathbb{W}(H^2(\Gamma), L^2(\Gamma))$, we have
\begin{align*}
\int_0^t\int_{\Gamma(s)}\mdd a(s) \slap a(s) 
&=\frac 12\int_{\Gamma_0}|\sgrad    a (0)|^2-\frac 12\int_{\Gamma(t)}|\sgrad    a (t)|^2  + \frac 12\int_0^t\int_{\Gamma(s)} \sgrad a(s)^T \mathbf{H}(s)\sgrad a(s) .
\end{align*}
\end{lem}
\begin{proof}
Take functions $a, b \in C^\infty_{H^2(\Gamma)}$. 
By integrating by parts and using the formula
(see \cite[Equation (5.8)]{Dziuk2013})
\begin{align*}
    \frac{d}{dt}\int_{\Gamma(t)} \sgrad a \cdot \sgrad b 
    &= \int_{\Gamma(t)}\sgrad \mdd a \cdot \sgrad b + \sgrad a \cdot  \sgrad \mdd b + \sgrad a^T \mathbf{H}(t) \sgrad b,
\end{align*}
we get
\begin{align*}
\int_{\Gamma(t)}\mdd a(t) \slap a(t) &= - \int_{\Gamma(t)}\sgrad a(t) \sgrad \mdd a(t) =-\frac 12\frac{d}{dt}\int_{\Gamma(t)}|\sgrad    a (t)|^2  + \frac 12\int_{\Gamma(t)} \sgrad a(t)^T \mathbf{H}(t)\sgrad a(t), 
\end{align*}
giving
\begin{align*}
\int_0^t\int_{\Gamma(s)}\mdd a(s) \slap a(s) 
&=\frac 12\int_{\Gamma_0}|\sgrad    a (0)|^2-\frac 12\int_{\Gamma(t)}|\sgrad    a (t)|^2  + \frac 12\int_0^t\int_{\Gamma(s)} \sgrad a(s)^T \mathbf{H}(s)\sgrad a(s).
\end{align*}
Since $C^\infty_{H^2(\Gamma)} \subset \mathbb{W}(H^2(\Gamma), L^2(\Gamma))$ is dense (see the proof of Lemma 2.14 in \cite{AEStefan}) and using $\mathbb{W}(H^2(\Gamma), L^2(\Gamma)) \cts C^0_{H^1(\Gamma)}$, it follows that the above equality also holds for $a \in \mathbb{W}(H^2(\Gamma), L^2(\Gamma))$.

\end{proof}

\section{Verification of assumptions related to $g$}\label{sec:app_on_g}
\begin{proof}[Proof of \cref{lem:verification_of_g1}]
We split the proof by the two choices of $g$.
\begin{enumerate}[label=({\roman*}), wide, labelwidth=!, labelindent=0pt]

\item Let us first address the case \eqref{eq:g_usual} where $g(u,w)=uw$.
 Since at least one of the sequences converges strongly and the other converges weakly in $L^2_{L^2(\Gamma)}$ (from continuous embeddings) by hypothesis, \eqref{ass:new_on_g_1} and \eqref{ass:new_on_g_2} are trivially met. 

\item  Now we focus on the case of  \eqref{eq:g_mm} where we have $g(u,w) = u^nw\slash (1+u^n)$  for $n>1$. Let us consider the condition \eqref{ass:new_on_g_1} first. 
By pulling back onto the initial surface, we have (by hypothesis) that
\[\int_0^T \int_{\Gamma_0} g(\tilde u_k, \tilde w_k)J^0_t \to 0\]
where $J^0_t$ is the determinant of the transformation. Instead of writing the tildes, we will abuse notation and write simply $u_k$ and $w_k$ for simplicity. Since $J^0_t$ is strictly positive and $g$ and $u_k, w_k$ are non-negative, the above implies that for a subsequence (which we relabel), we have 
\[g(u_k(y),w_k(y)) \to 0 \text{ pointwise a.e. $y \in (0,T)\times \Gamma_0$},\]  i.e., the convergence holds for all $y \in (0,T)\times \Gamma_0\setminus N_1$ where $N_1$ is a null set. By the assumed strong convergence, we have (again for a subsequence that has been relabelled)
\[w_k \to w \text{ pointwise a.e $y \in (0,T)\times \Gamma_0$},\]
i.e., the convergence holds for all $y \in (0,T)\times \Gamma_0 \setminus N_2$ where $N_2$ is a null set.

We have the identity
$B = B \cap A \cup B \cap A^c$ for any two sets $A$ and $B$, hence 
\[\{w\neq 0\} = \{w\neq 0\} \cap (N_1 \cup N_2) \cup \{w\neq 0\} \cap (N_1^c \cap N_2^c).\]
The first intersection on the right-hand side is a null set since it is an intersection with null sets, hence, we have
\[\{w\neq 0\} = \{w\neq 0\} \cap (N_1^c \cap N_2^c)\]
up to sets of measure zero.




Take $y^* \in \{w\neq 0\}$. Then recalling the definition of $g$ and the above decomposition,
\begin{align*}
\frac{u_k(y^*)^nw_k(y^*)}{1+u_k(y^*)^n} \to 0\qquad\text{and}\qquad w_k(y^*) \to w(y^*).
\end{align*}
Since $w(y^*) \neq 0$ and $w \geq 0$, we can find a $K(y^*)$ such that $k \geq K(y^*)$ implies that $w_k(y^*) > 0$ for all $k \geq K(y^*)$. It follows from the above displayed equation (by dividing by $w_k(y^*)$ and taking the limit) that
\[\frac{u_k(y^*)^n}{1+u_k(y^*)^n} \to 0\]
and hence $u_k(y^*) \to 0$. This holds for all $y^* \in \{w \neq 0\}$
. It follows that
\[u_k \chi_{\{w \neq 0\}} \to 0 \text{ pointwise a.e. in $(0,T)\times \Gamma_0$}.\]
By the continuous embedding, $u_k \weaklyto u$ in $L^2((0,T)\times \Gamma_0)$, and because the characteristic function is bounded, $u_k\chi_{\{w\neq 0\}} \weaklyto u\chi_{\{w \neq 0\}}$ in the same space. Because the pointwise a.e. and weak limit coincide, we must then have
\[u\chi_{\{w\neq0\}} = 0,\]
and hence $u=0$ on $\{w\neq 0\}$. It follows that $uw=0$, which implies $g(u,w)=0$. This shows that \eqref{ass:new_on_g_1} holds.

For \eqref{ass:new_on_g_2}, this is much simpler because we have an $L^\infty$ bound and strong convergence for $u$ (and hence strong convergence of the nonlinear part of $g(u,w)$).
\end{enumerate}
\end{proof}
\begin{proof}[Proof of \cref{lem:verification_of_g3}]
If $g(u,w) = uw$, we can manipulate and use the strong and weak convergences:
\begin{align*}
\int_0^T\intGt g(u_k,w_k) 
&= \int_0^T \langle w_k, u_k \rangle_{H^{-1\slash 2}(\Gamma(t)), H^{1\slash 2}(\Gamma(t))}
&\to \int_0^T \langle w, u \rangle_{H^{-1\slash 2}(\Gamma(t)), H^{1\slash 2}(\Gamma(t))} = \int_0^T \intGt uw.
\end{align*}
\end{proof}

\section{Nondimensionalisation}\label{sec:nondim}
Inspired by \cite{garcia2013mathematical} as in \cite{ERV, AET}, consider the system 
\begin{equation}\label{eq:modelLag}
\begin{aligned}
 \mdd{u}+u\nabla\cdot\mathbf V_\Omega-D_\Omega\Delta u&=0&\text{ in }\Omega(t),\\
D_\Omega\nabla u \cdot \nu + u(\mathbf V_\Gamma-\mathbf V_\Omega)\cdot \nu &=  k_{\text{off}}z-k_{\text{on}}uw&\text{ on }\Gamma(t),\\
\mdd{w} + w\nabla_\Gamma\cdot\mathbf V_\Gamma-D_\Gamma\Delta w&= k_{\text{off}}z-k_{\text{on}}uw&\text{ on }\Gamma(t),\\
\mdd{z} + z\nabla_\Gamma\cdot\mathbf V_\Gamma-D_{\Gamma'}\Delta z&=-k_{\text{off}}z+k_{\text{on}}uw&\text{ on }\Gamma(t).
\end{aligned}
\end{equation}
Introducing suitable scaling parameters, we define the dimensionless variables
\[
\bar{u}=\frac{u}{U},\ \bar{w}=\frac{w}{W},\ \bar{z}=\frac{z}{Z},\ \bar{x}=\frac{x}{L},\ \bar{t}=\frac{t}{S},\ \bar{\mathbf V}_i=\frac{s\mathbf V_i}{L},\qquad i=\Omega\text{ or }\Gamma ,
\]
we obtain the following dimensionless system (where we drop the bars for notational simplicity)
\begin{equation}\label{eq:modelND}
\begin{aligned}
 \mdd{u}+u\nabla\cdot\mathbf V_\Omega-\delta_\Omega^{-1}\Delta u&=0&\text{ in }\Omega(t),\\
\nabla u \cdot \nu + \delta_\Omega u(\mathbf V_\Gamma-\mathbf V_\Omega)\cdot \nu &=  \delta^{-1}_{k^\prime} z-\delta_k^{-1}uw&\text{ on }\Gamma(t),\\
 \mdd{w} + w\nabla_\Gamma\cdot\mathbf V_\Gamma-\delta_\Gamma\Delta w&=  \mu(\delta^{-1}_{k^\prime} z-\delta_k^{-1}uw)&\text{ on }\Gamma(t),\\
 \mdd{z} + z\nabla_\Gamma\cdot\mathbf V_\Gamma-\delta_{\Gamma'}\Delta z&=  \mu^\prime(-\delta^{-1}_{k^\prime} z+\delta_k^{-1}uw)&\text{ on }\Gamma(t),
\end{aligned}
\end{equation}
where we have introduced the dimensionless variables
\[
\delta_\Omega=\frac{L^2}{D_\Omega S},\;\; \delta_\Gamma=\frac{D_\Gamma S}{L^2},\;\;
\delta_{\Gamma^\prime}=\frac{D_{\Gamma^\prime} S}{L^2},\;\;\delta_k=\frac{D_\Omega}{k_{\text{on}}LW},\;\; \delta^{-1}_{k^\prime}=\frac{k_{\text{off}}ZL}{D_\Omega U},\;\; \mu=\frac{SUD_\Omega}{LW},\;\; \mu^\prime=\frac{SUD_\Omega}{LZ}.
\]
\begin{table}
  \centering
  \begin{tabular}{ccc}
    \hline
    Parameter & Value & Source \\
    \hline
    $L$ & $7.5 \cdot 10^{-6} \, \mathrm{m}$ & \cite{garcia2013mathematical} \\
    $U$ & $1.0 \cdot 10^{-3} \, \mathrm{mol} \, \mathrm{m}^{-3}$ & \cite{garcia2013mathematical} \\
    $W$ & $2.3 \cdot 10^{-8} \, \mathrm{mol} \, \mathrm{m}^{-2}$ & \cite{garcia2013mathematical} \\
    $Z$ & $2.3 \cdot 10^{-8} \, \mathrm{mol} \, \mathrm{m}^{-2}$ & limited by total receptor concentration \\
    $D_\Omega$ & $1.0 \cdot 10^{-11} \, \mathrm{m}^2 \, \mathrm{s}^{-1}$ & \cite{LINDERMAN1986295} \\
    $D_\Gamma$ & $1.0 \cdot 10^{-15} \, \mathrm{m}^2 \, \mathrm{s}^{-1}$ & \cite{LINDERMAN1986295} \\
    $D_{\Gamma^\prime}$ & $1.0 \cdot 10^{-15} \, \mathrm{m}^2 \, \mathrm{s}^{-1}$ & \cite{LINDERMAN1986295} \\
    $k_{\text{on}}$ & $1.0 \cdot 10^3 \, \mathrm{m}^3 \, \mathrm{mol}^{-1} \, \mathrm{s}^{-1}$ & \cite{garcia2013mathematical} \\
    $k _{\text{off}}$ & $5.0 \cdot 10^{-3} \mathrm{s}^{-1}$ & \cite{garcia2013mathematical} \\
    \hline
  \end{tabular}

  \caption{Parameters used for rescaling equations. The values for $U$ and $W$ are extreme values taken from within a physical range from \cite{garcia2013mathematical}.}
  \label{tab:parameters}
\end{table}
Taking values from \cref{tab:parameters}, we infer that
\begin{gather*}
  \delta_\Omega = (5.6 \, \mathrm{s}) \cdot S^{-1}, \quad
  \delta_k = 5.7 \cdot 10^{-2}, \quad
  \delta^{-1}_{k^\prime} = 8.7 \cdot 10^{-2}, \\
  \delta_\Gamma=\delta_{\Gamma^\prime} = (1.8 \cdot 10^{-4} \, \mathrm{s}^{-1} ) \cdot S, \quad
  \mu =   \mu' = (5.7 \cdot 10^{-2} \, \mathrm{s}^{-1} ) \cdot S.
\end{gather*}
We see $\delta_k \ll 1$ motivating the consideration of the limit $\delta_k\to 0$.
For $S = L^2 / D_\Omega = 5.6 \, \mathrm{s}$, i.e.,  the timescale such that $\delta_\Omega=1$ we have
\begin{align*}
  \delta_\Omega = 1, \qquad
  \delta_\Gamma=  \delta_{\Gamma^\prime} = 1.0 \cdot 10^{-3} \ll 1, \qquad
  \mu=\mu^\prime = 3.2 \cdot 10^{-1} \approx 1,
\end{align*}
motivating the limit $\delta_k=\delta_\Gamma=\delta_{\Gamma^\prime}\to 0$.
Alternatively, taking $S = 10^2 \, \mathrm{s}$, i.e., a timescale such that $\mu,\mu^\prime\approx 1$, we have
\begin{align*}
  \delta_\Omega = 5.7 \cdot 10^{-2} \ll 1, \qquad
  \delta_\Gamma =  \delta_{\Gamma^\prime} =  1.8 \cdot 10^{-2} \ll 1, \qquad
  \mu = \mu^\prime=5.7 \approx 1,
\end{align*}
motivating the limit $\delta_k=\delta^{-1}_{k^\prime}=\delta_\Gamma=\delta_{\Gamma^\prime}=\delta_\Omega\to 0$.
\nocite{}
\bibliographystyle{abbrvnat}
\bibliography{main}
\end{document}